\documentclass[11pt,twoside,ngerman,english]{scrartcl}
\usepackage[T1]{fontenc}
\usepackage[latin9]{luainputenc}
\usepackage{geometry}
\usepackage{xcolor}
\usepackage{babel}
\usepackage{float}
\usepackage{amsmath}
\usepackage{amsthm}
\usepackage{amssymb}
\usepackage{esint}
\usepackage[unicode=true,pdfusetitle,
 bookmarks=true,bookmarksnumbered=false,bookmarksopen=false,
 breaklinks=false,pdfborder={0 0 1},backref=false,colorlinks=false]
 {hyperref}

\makeatletter
\numberwithin{equation}{section}
\theoremstyle{plain}
\newtheorem{thm}{\protect\theoremname}[section]
  \theoremstyle{plain}
  \newtheorem{assumption}[thm]{\protect\assumptionname}
  \theoremstyle{remark}
  \newtheorem{rem}[thm]{\protect\remarkname}
  \theoremstyle{plain}
  \newtheorem{lem}[thm]{\protect\lemmaname}
  \theoremstyle{definition}
  \newtheorem{defn}[thm]{\protect\definitionname}
  \theoremstyle{plain}
  \newtheorem{prop}[thm]{\protect\propositionname}
  \theoremstyle{remark}
  \newtheorem{notation}[thm]{\protect\notationname}
  \theoremstyle{plain}
  \newtheorem{cor}[thm]{\protect\corollaryname}

\allowdisplaybreaks[1]

\makeatother

  \addto\captionsenglish{\renewcommand{\assumptionname}{Assumption}}
  \addto\captionsenglish{\renewcommand{\corollaryname}{Corollary}}
  \addto\captionsenglish{\renewcommand{\definitionname}{Definition}}
  \addto\captionsenglish{\renewcommand{\lemmaname}{Lemma}}
  \addto\captionsenglish{\renewcommand{\notationname}{Notation}}
  \addto\captionsenglish{\renewcommand{\propositionname}{Proposition}}
  \addto\captionsenglish{\renewcommand{\remarkname}{Remark}}
  \addto\captionsenglish{\renewcommand{\theoremname}{Theorem}}
  \addto\captionsngerman{\renewcommand{\assumptionname}{Annahme}}
  \addto\captionsngerman{\renewcommand{\corollaryname}{Korollar}}
  \addto\captionsngerman{\renewcommand{\definitionname}{Definition}}
  \addto\captionsngerman{\renewcommand{\lemmaname}{Lemma}}
  \addto\captionsngerman{\renewcommand{\notationname}{Notation}}
  \addto\captionsngerman{\renewcommand{\propositionname}{Satz}}
  \addto\captionsngerman{\renewcommand{\remarkname}{Bemerkung}}
  \addto\captionsngerman{\renewcommand{\theoremname}{Theorem}}
  \providecommand{\assumptionname}{Assumption}
  \providecommand{\corollaryname}{Corollary}
  \providecommand{\definitionname}{Definition}
  \providecommand{\lemmaname}{Lemma}
  \providecommand{\notationname}{Notation}
  \providecommand{\propositionname}{Proposition}
  \providecommand{\remarkname}{Remark}
\providecommand{\theoremname}{Theorem}

\newcommand{\N}{\mathbb{N}}
\newcommand{\R}{\mathbb{R}}
\newcommand{\eps}{\epsilon}

\begin{document}
\global\long\def\d{\mathrm{d}}
\global\long\def\gt{\Gamma_{t}\left(\delta\right)}
\global\long\def\we{\tilde{\mathbf{w}}_{1}}
\global\long\def\vt{\tilde{\mathbf{v}}_{A}^{\epsilon}}
\global\long\def\wei{\mathbf{w}_{1}^{\epsilon}}
\global\long\def\wzw{\mathbf{w}_{2}^{\epsilon}}
\global\long\def\yb{\left(\Psi_{1}^{0}\right)^{\bot}}
\global\long\def\dh{\d\mathcal{H}^{n-1}}
\global\long\def\twe{\tilde{\mathbf{w}}_{1}^{\epsilon}}
\global\long\def\tweh{\tilde{\mathbf{w}}_{1}^{\epsilon,H}}
\global\long\def\twz{\tilde{\mathbf{w}}_{2}^{\epsilon}}
\global\long\def\cte{\tilde{c}^{\epsilon}}
\global\long\def\mte{\tilde{\mu}^{\epsilon}}
\global\long\def\vte{\tilde{\mathbf{v}}^{\epsilon}}
\global\long\def\pte{\tilde{p}^{\epsilon}}
\global\long\def\di{\text{div}}
\global\long\def\vn{\mathbf{u}^{\epsilon,\mathbf{n}}}
\global\long\def\vt{\mathbf{u}^{\epsilon,\tau}}
\global\long\def\cae{\overline{c_{A}^{\epsilon}}}
\global\long\def\caeh{\overline{c_{A}^{\epsilon,H}}}
\global\long\def\wei{\mathbf{w}_{1}^{\epsilon}}
\global\long\def\rs{\mathbf{r}_{\mathrm{S}}^{\epsilon}}
\global\long\def\rdiv{r_{\mathrm{div}}^{\epsilon}}
\global\long\def\rc{r_{\mathrm{CH1}}^{\epsilon}}
\global\long\def\rh{r_{\mathrm{CH2}}^{\epsilon}}
\global\long\def\ra{\mathcal{R}_{\alpha}}
\global\long\def\rsi{\mathbf{r}_{\mathrm{S},I}^{\epsilon}}
\global\long\def\rso{\mathbf{r}_{\mathrm{S},O}^{\epsilon}}
\global\long\def\rdivi{r_{\mathrm{div},I}^{\epsilon}}
\global\long\def\rdivo{r_{\mathrm{div},O}^{\epsilon}}
\global\long\def\rci{r_{\mathrm{CH1},I}^{\epsilon}}
\global\long\def\rco{r_{\mathrm{CH1},O}^{\epsilon}}
\global\long\def\rhi{r_{\mathrm{CH2},I}^{\epsilon}}
\global\long\def\rhO{r_{\mathrm{CH2},O}^{\epsilon}}
\global\long\def\iome{\int_{\Omega_{T_{\epsilon}}}}
\global\long\def\rsb{\mathbf{r}_{\text{S},\mathbf{B}}^{\epsilon}}
\global\long\def\rdivb{r_{\text{div},\mathbf{B}}^{\epsilon}}
\global\long\def\rcb{r_{\text{CH1},\mathbf{B}}^{\epsilon}}
\global\long\def\rhb{r_{\text{CH2},\mathbf{B}}^{\epsilon}}
\global\long\def\trhb{\tilde{r}_{\text{CH2},\mathbf{B}}^{\epsilon}}

\title{Sharp Interface Limit of a Stokes/Cahn-Hilliard System, Part II:
Approximate Solutions}

\author{Helmut Abels\thanks{  \textit{Fakultät für Mathematik,   Universität Regensburg,   93040 Regensburg,   Germany}   \textsf {helmut.abels@ur.de} } \  and Andreas Marquardt\thanks{\textit   {Fakultät für Mathematik,   Universität Regensburg,   93040 Regensburg,   Germany}   }}

\maketitle
\begin{abstract}
We construct rigorously suitable approximate solutions to the
Stokes/Cahn-Hilliard system by using the method of matched
asymptotics expansions. This is a main step in the proof of convergence given in the first part of this contribution, \cite{NSCH1}, where the rigorous sharp
interface limit of a coupled Stokes/Cahn-Hilliard system
in a two dimensional, bounded and smooth domain is shown. As a novelty compared to earlier works, we introduce fractional order terms, which
are of significant importance, but share the problematic feature that
they may not be uniformly estimated in $\epsilon$ in arbitrarily
strong norms. As a consequence, gaining necessary estimates for the error, which occurs when considering the
approximations in the Stokes/Cahn-Hilliard system, is rather involved.
\end{abstract}

{\small\noindent
{\textbf {Mathematics Subject Classification (2000):}}
Primary: 76T99; Secondary:
35Q30, 
35Q35, 
35R35,
76D05, 
76D45\\ 
{\textbf {Key words:}} Two-phase flow, diffuse interface model, sharp interface limit, Cahn-Hilliard equation, Free boundary problems
}

\section{Introduction and Overview}

Let $T>0$, $\Omega\subset\mathbb{R}^{2}$ be a bounded and smooth
domain, $\Omega_{T}:=\Omega\times(0,T)$, $\partial\Omega_T= \partial\Omega\times (0,T)$ and $\alpha_{0}>0$
be a fixed constant. We consider the Stokes/Cahn-Hilliard system
\begin{align}
-\Delta\mathbf{v}^{\epsilon}+\nabla p^{\epsilon} & =\mu^{\epsilon}\nabla c^{\epsilon} &  & \text{in }\Omega_{T},\label{eq:StokesPart}\\
\operatorname{div}\mathbf{v}^{\epsilon} & =0 &  & \text{in }\Omega_{T},\label{eq:StokesPart2}\\
\partial_{t}c^{\epsilon}+\mathbf{v}^{\epsilon}\cdot\nabla c^{\epsilon} & =\Delta\mu^{\epsilon} &  & \text{in }\Omega_{T},\label{eq:CH-Part1}\\
\mu^{\epsilon} & =-\epsilon\Delta c^{\epsilon}+\tfrac{1}{\epsilon}f'(c^{\epsilon}) &  & \text{in }\Omega_{T},\label{eq:CH-Part2}\\
c^{\epsilon}|_{t=0} & =c_{0}^{\epsilon} &  & \text{in }\Omega,\label{eq:CH-Anfang}\\
\left(-2D_{s}\mathbf{v}^{\epsilon}+p^{\epsilon}\mathbf{I}\right)\cdot\mathbf{n}_{\partial\Omega} & =\alpha_{0}\mathbf{v}^{\epsilon}  &  & \text{on }\partial_{T}\Omega,\label{eq:StokesBdry}\\
(\mu^{\epsilon}, c^\epsilon) & =(0,-1)  &  & \text{on }\partial_{T}\Omega,\label{eq:Dirichlet2}
\end{align}
$\mathbf{v}^{\epsilon}$ and $p^{\epsilon}$ represent the mean velocity
and pressure, $D_{s}\mathbf{v}^{\epsilon}:=\tfrac{1}{2}\big(\nabla\mathbf{v}^{\epsilon}+(\nabla\mathbf{v}^{\epsilon})^{T}\big)$,
$c^{\epsilon}$ is related to the concentration difference of the
fluids and $\mu^{\epsilon}$ is the chemical potential of the mixture.
Moreover, $c_{0}^{\epsilon}$ is a suitable initial value, specified
in Theorem \ref{Main} and $f\colon \mathbb{R}\rightarrow\mathbb{R}$ is
a double well potential. It is the aim of \cite{NSCH1} to establish
that the sharp interface limit of (\ref{eq:StokesPart})\textendash (\ref{eq:Dirichlet2})
is given by the system
\begin{align}
-\Delta\mathbf{v}+\nabla p & =0 &  & \text{in }\Omega^{\pm}(t),t\in[0,T_0],\label{eq:S-SAC1}\\
\operatorname{div}\mathbf{v} & =0 &  & \text{in }\Omega^{\pm}(t),t\in[0,T_0],\label{eq:S-SAC2}\\
\Delta\mu & =0 &  & \text{in }\Omega^{\pm}(t),t\in[0,T_0],\label{eq:S-SAC3}\\
\left(-2D_{s}\mathbf{v}+p\mathbf{I}\right)\mathbf{n}_{\partial\Omega} & =\alpha_{0}\mathbf{v} &  & \text{on }\partial_{T_{0}}\Omega,\\
\mu & =0 &  & \text{on }\partial_{T_{0}}\Omega,\\
\left[2D_{s}\mathbf{v}-p\mathbf{I}\right]\mathbf{n}_{\Gamma_{t}} & =-2\sigma H_{\Gamma_{t}}\mathbf{n}_{\Gamma_{t}} &  & \text{on }\Gamma_{t},t\in[0,T_0],\label{eq:S-SAC4}\\
\mu & =\sigma H_{\Gamma_{t}} &  & \text{on }\Gamma_{t},t\in[0,T_0],\label{eq:S-SAC5}\\
-V_{\Gamma_{t}}+\mathbf{n}_{\Gamma_{t}}\cdot\mathbf{v} & =\tfrac{1}{2}\left[\mathbf{n}_{\Gamma_{t}}\cdot\nabla\mu\right] &  & \text{on }\Gamma_{t},t\in[0,T_0],\label{eq:S-SAC6}\\
\left[\mathbf{v}\right] & =0 &  & \text{on }\Gamma_{t},t\in[0,T_0],\\
\Gamma(0) & =\Gamma_{0}.\label{eq:S-SAC8}
\end{align}
Here, $\Gamma_{0}\subset\subset\Omega$ is a given, smooth, non-intersecting,
closed initial curve. We assume that $\Gamma=\bigcup_{t\in[0,T_0]}\Gamma_{t}\times\left\{ t\right\} $
is a smoothly evolving hypersurface in $\Omega$, where $\left(\Gamma_{t}\right)_{t\in[0,T_0]}$
are compact, non-intersecting, closed curves in $\Omega$. Moreover, 
$\Omega^{+}(t)$ is defined as the inside of $\Gamma_{t}$ and
$\Omega^{-}(t)$ is such that $\Omega$ is the disjoint
union of $\Omega^{+}(t)$, $\Omega^{-}(t)$
and $\Gamma_{t}$. Furthermore, we define $\Omega_{T}^{\pm}=\cup_{t\in\left[0,T\right]}\Omega^{\pm}(t)\times\left\{ t\right\} $
for $T\in[0,T_0]$ and define $\mathbf{n}_{\Gamma_{t}}(p)$
for $p\in\Gamma_{t}$ as the exterior normal with respect to $\Omega^{-}(t)$
and $V_{\Gamma_{t}}$, and $H_{\Gamma_{t}}$ as the normal velocity
and mean curvature of $\Gamma_{t}$ with respect to $\mathbf{n}_{\Gamma_{t}}$,
$t\in[0,T_0]$.   We use the definitions
\begin{align}
\left[g\right](p,t) & :=\lim_{h\searrow0}\left(g(p+\mathbf{n}_{\Gamma_{t}}(p)h)-g(p-\mathbf{n}_{\Gamma_{t}}(p)h)\right)\text{ for }p\in\Gamma_{t},\nonumber \\
\sigma & :=\frac{1}{2}\int_{-\infty}^{\infty}\theta_{0}'(s)^{2}\d s,\label{eq:sigma}
\end{align}
where $\theta_{0}\colon\mathbb{R}\rightarrow\mathbb{R}$ is the solution
to the ordinary differential equation 
\begin{equation}
-\theta_{0}''+f'(\theta_{0})=0\quad\text{in }\mathbb{R},\quad \theta_{0}(0)=0,\;\lim_{\rho\rightarrow\pm\infty}\theta_{0}(\rho)=\pm1.\label{eq:optprofdef}
\end{equation}
We refer to the introduction of \cite{NSCH1} for a review of known analytic results for the previous systems. 

Throughout this work we consider the following assumptions and notations:
Let $\left(\mathbf{v},p,\mu,\Gamma\right)$ be a smooth solution to
(\ref{eq:S-SAC1})\textendash (\ref{eq:S-SAC8}) and $\left(c^{\epsilon},\mu^{\epsilon},\mathbf{v}^{\epsilon},p^{\epsilon}\right)$
be smooth solutions to (\ref{eq:StokesPart})\textendash (\ref{eq:Dirichlet2})
for some $T_{0}>0$ and $\epsilon\in (0,1)$. More precisely $(\mathbf{v},p,\mu)$ are assumed to be smooth in $\Omega_{T_0}^{\pm}$ such that the function and their derivatives extend continuously  to $\overline{\Omega_{T_0}^{\pm}}$. Let 
\[
d_{\Gamma}:\Omega_{T_{0}}\rightarrow\mathbb{R},\;(x,t)\mapsto\begin{cases}
\mbox{dist}\left(\Omega^{-}(t),x\right) & \mbox{if }x\notin\Omega^{-}(t),\\
-\mbox{dist}\left(\Omega^{+}(t),x\right) & \mbox{if }x\in\Omega^{-}(t)
\end{cases}
\]
denote the signed distance function to $\Gamma$ such that $d_{\Gamma}$
is positive inside $\Omega_{T_{0}}^{+}$. We write $\Gamma_{t}(\alpha):=\left\{ \left.x\in\Omega\right|\left|d_{\Gamma}(x,t)\right|<\alpha\right\} $
for $\alpha>0$ and set $\Gamma(\alpha;T):=\bigcup_{t\in\left[0,T\right]}\Gamma_{t}(\alpha)\times\left\{ t\right\} $
for $T\in[0,T_0]$. Moreover, we assume that $\delta>0$
is a small positive constant such that $\text{dist}\left(\Gamma_{t},\partial\Omega\right)>5\delta$
for all $t\in[0,T_0]$ and such that $\operatorname{Pr}_{\Gamma_{t}}\colon\Gamma_{t}(3\delta)\rightarrow\Gamma_{t}$
is well-defined and smooth for all $t\in[0,T_0]$. In
the following we often use the notation $\Gamma(2\delta):=\Gamma(2\delta;T_{0})$
as a simplification. We also define a tubular neighborhood around
$\partial\Omega$: For this let $d_{\mathbf{B}}\colon\Omega\rightarrow\mathbb{R}$
be the signed distance function to $\partial\Omega$ such that $d_{\mathbf{B}}<0$
in $\Omega$. As for $\Gamma_{t}$ we define a tubular neighborhood
by $\partial\Omega(\alpha):=\left\{ x\in\Omega\left|-\alpha<d_{\mathbf{B}}(x)<0\right.\right\} $
and $\partial_{T}\Omega(\alpha):=\left\{ \left.(x,t)\in\Omega_{T}\right|d_{\mathbf{B}}(x)\in(-\alpha,0)\right\} $
for $\alpha>0$ and $T\in(0,T_{0}]$. Moreover, we denote
the outer unit normal to $\Omega$ by $\mathbf{n}_{\partial\Omega}$
and denote the normalized tangent by $\tau_{\partial\Omega}$, which
is fixed by the relation
\[
\mathbf{n}_{\partial\Omega}(p)=\left(\begin{array}{cc}
0 & -1\\
1 & 0
\end{array}\right)\tau_{\partial\Omega}(p)
\]
for $p\in\partial\Omega$. Finally we assume that $\delta>0$ is chosen
small enough such that the projection $\operatorname{Pr}_{\partial\Omega}\colon\partial\Omega(\delta)\rightarrow\partial\Omega$
along the normal $\mathbf{n}_{\partial\Omega}$ is also well-defined
and smooth.

Considering the potential $f$, we assume that it is a fourth order
polynomial, satisfying 
\begin{equation}
  f(\pm1)=f'(\pm1)=0,\,f''(\pm1)>0,\:f(s)=f(-s)>0\quad\text{for all }s\in\mathbb{R}  
  \label{eq:f}
\end{equation}
for some $C>0$ and fulfilling $k_{f}:=f^{(4)}>0$. Then
the ordinary differential equation (\ref{eq:optprofdef}) allows for
a unique, monotonically increasing solution $\theta_{0}\colon \mathbb{R}\rightarrow(-1,1)$.
This solution furthermore satisfies the decay estimate
\begin{equation}
\big|\theta_{0}^{2}(\rho)-1\big|+\big|\theta_{0}^{(n)}(\rho)\big|\leq C_{n}e^{-\alpha\left|\rho\right|}\quad\text{for all }\rho\in\mathbb{R},\;n\in\mathbb{N}\backslash\left\{ 0\right\} \label{eq:optimopti}
\end{equation}
for constants $C_{n}>0$, $n\in\mathbb{N}\backslash\left\{ 0\right\} $
and fixed $\alpha\in\left(0,\min\left\{ \sqrt{f''(-1)},\sqrt{f''(1)}\right\} \right)$.
As it will be needed a lot in this work, we denote by $\xi\in C^{\infty}(\mathbb{R})$ a
cut-off function such that
\begin{equation}
\xi(s)=1\text{ if }\left|s\right|\leq\delta,\,\xi(s)=0\text{ if }\left|s\right|>2\delta,\text{ and }0\geq s\xi'(s)\geq-4\text{ if }\delta\leq\left|s\right|\leq2\delta.\label{eq:cut-off}
\end{equation}
The main result of \cite{NSCH1} is the following (for an explanation
of the used notations see the preliminaries section): 
\begin{thm}[Main Result]
  \label{Main}~\\ Let $\left(\mathbf{v},p,\mu,\Gamma\right)$ be a smooth solution to
\eqref{eq:S-SAC1}-\eqref{eq:S-SAC8} for some $T_{0}>0$. Moreover, let $M\in\N$ with $M\geq 4$, let
  $\xi$ satisfy \eqref{eq:cut-off} and let $\gamma(x):=\xi(4d_{\mathbf{B}}(x))$
for all $x\in\Omega$ and let for $\epsilon\in(0,1)$ a
smooth function $\psi_{0}^{\epsilon}\colon\Omega\rightarrow\mathbb{R}$
be given, which satisfies $\left\Vert \psi_{0}^{\epsilon}\right\Vert _{C^{1}(\Omega)}\leq C_{\psi_{0}}\epsilon^{M}$
for some $C_{\psi_{0}}>0$ independent of $\epsilon$. Then there
are smooth functions $c_{A}^{\epsilon}\colon \Omega\times[0,T_0]\rightarrow\mathbb{R},\mathbf{v}_{A}^{\epsilon}:\Omega\times[0,T_0]\rightarrow\mathbb{R}^{2}$
for $\epsilon\in(0,1)$ such that the following holds: 

There is some $c_A^\epsilon\colon \Omega\to \R$, $\eps\in (0,1]$, depending only on $(\mathbf{v},p,\mu,\Gamma)$ such that, if $\left(\mathbf{v}^{\epsilon},p^{\epsilon},c^{\epsilon},\mu^{\epsilon}\right)$
are smooth solutions to \eqref{eq:StokesPart})-\eqref{eq:Dirichlet2}
with initial value
\begin{align}
c_{0}^{\epsilon}(x) & =c_{A}^{\epsilon}(x,0)+\psi_{0}^{\epsilon}(x)\quad \text{for all }x\in\Omega,\label{eq:canf}
\end{align}
then there are some $\epsilon_{0}\in(0,1]$,
$K>0$, $T\in(0,T_{0}]$ such that
\begin{subequations}\label{eq:Main}
\begin{align}
\left\Vert c^{\epsilon}-c_{A}^{\epsilon}\right\Vert _{L^{2}\left(0,T;L^{2}(\Omega)\right)}+\left\Vert \nabla^{\Gamma}\big(c^{\epsilon}-c_{A}^{\epsilon}\big)\right\Vert _{L^{2}\left(0,T;L^{2}\left(\Gamma_{t}(\delta)\right)\right)} & \le K\epsilon^{M-\frac{1}{2}},\label{eq:Main1}\\
\epsilon\left\Vert \nabla\big(c^{\epsilon}-c_{A}^{\epsilon}\big)\right\Vert _{L^{2}\left(0,T;L^{2}\left(\Omega\backslash\Gamma_{t}(\delta)\right)\right)}+\left\Vert c^{\epsilon}-c_{A}^{\epsilon}\right\Vert _{L^{2}\left(0,T;L^{2}\left(\Omega\backslash\Gamma_{t}(\delta)\right)\right)} & \leq K\epsilon^{M+\frac{1}{2}},\label{eq:Main2}\\
\epsilon^{\frac{3}{2}}\left\Vert \partial_{\mathbf{n}}\big(c^{\epsilon}-c_{A}^{\epsilon}\big)\right\Vert _{L^{2}\left(0,T;L^{2}\left(\Gamma_{t}(\delta)\right)\right)}+\left\Vert c^{\epsilon}-c_{A}^{\epsilon}\right\Vert _{L^{\infty}\left(0,T;H^{-1}(\Omega)\right)} & \leq K\epsilon^{M},\label{eq:Main3}\\
\int_{\Omega_{T}}\epsilon\left|\nabla\big(c^{\epsilon}-c_{A}^{\epsilon}\big)\right|^{2}+\tfrac1{\epsilon}f''(c_{A}^{\epsilon})\left(c^{\epsilon}-c_{A}^{\epsilon}\right)^{2}\d (x,t) & \leq K^{2}\epsilon^{2M},\label{eq:Main4}\\
\left\Vert \gamma\big(c^{\epsilon}-c_{A}^{\epsilon}\big)\right\Vert _{L^{\infty}\left(0,T;L^{2}(\Omega)\right)}+\epsilon^{\frac{1}{2}}\left\Vert \gamma\Delta\big(c^{\epsilon}-c_{A}^{\epsilon}\big)\right\Vert _{L^{2}(\Omega_{T})} & \leq K\epsilon^{M-\frac{1}{2}},\label{eq:Main5}\\
\left\Vert \gamma\nabla\big(c^{\epsilon}-c_{A}^{\epsilon}\big)\right\Vert _{L^{2}(\Omega_{T})}+\left\Vert \gamma\big(c^{\epsilon}-c_{A}^{\epsilon}\big)\nabla\big(c^{\epsilon}-c_{A}^{\epsilon}\big)\right\Vert _{L^{2}(\Omega_{T})} & \leq K\epsilon^{M},\label{eq:Main6}
\end{align}
\end{subequations} and for $q\in(1,2)$
\begin{equation}
\left\Vert \mathbf{v}^{\epsilon}-\mathbf{v}_{A}^{\epsilon}\right\Vert _{L^{1}\left(0,T;L^{q}(\Omega)\right)}\leq C(K,q)\epsilon^{M-\frac{1}{2}},\label{eq:Mainv}
\end{equation}
 hold for all $\epsilon\in(0,\epsilon_{0})$ and some $C(K,q)>0$.
Moreover, we have
\begin{align}
\lim_{\epsilon\rightarrow0}c_{A}^{\epsilon} & =\pm1\text{ in }L^{\infty}((s,t)\times \Omega')\label{eq:Maincconverge}
\end{align}
and
\begin{equation}
\lim_{\epsilon\rightarrow0}\mathbf{v}_{A}^{\epsilon}=\mathbf{v}^{\pm}\quad \text{ in }L^{6}\left((s,t);H^{2}(\Omega')^{2}\right)\label{eq:Mainvconverge}
\end{equation}
for every $(s,t)\times \Omega'\subset\subset\Omega_{T}^{\pm}$.
\end{thm}
\begin{rem}
  Here $c_A^\epsilon$ is determined by formally matched asymptotic calcultations in the following proof. In highest we have
  \begin{equation*}
    c_A^\epsilon (x)= \theta_0\left(\tfrac{d_{\Gamma_0}(x)+\epsilon h_0(s,t)}\epsilon\right)+ O(\eps) \qquad \text{uniformly as }\epsilon\to 0
  \end{equation*}
  for some $h_0\colon \Gamma\to \R$, where $s=\operatorname{Pr}_{\Gamma_t}(x)$.
\end{rem}

It will be beneficial to the readability of many results throughout
this contribution to introduce the following set of assumptions, which
will be cited often later on.
\begin{assumption}
\label{assu:Main-est}Let $\gamma(x):=\xi\left(4d_{\mathbf{B}}(x)\right)$
for all $x\in\Omega$. We assume that $c_{A}:\Omega\times[0,T_0]\rightarrow\mathbb{R}$
is a smooth function and that there are $\epsilon_{0}\in(0,1)$,
$K\geq1$ and a family $\left(T_{\epsilon}\right)_{\epsilon\in(0,\epsilon_{0})}\subset(0,T_{0}]$
such that the following holds: if $c^{\epsilon}$ is given as in Theorem
\ref{Main} with $c_{0}^{\epsilon}(x)=c_{A}(x,0)$,
then it holds for $R:=c^{\epsilon}-c_{A}^{\epsilon}$

\begin{subequations}\label{eq:Main-est}\foreignlanguage{ngerman}{\vspace{-5mm}
}
\begin{align}
\left\Vert R\right\Vert _{L^{2}\left(\Omega_{T_{\epsilon}}\right)}+\left\Vert \nabla^{\Gamma}R\right\Vert _{L^{2}\left(0,T_{\epsilon};L^{2}\left(\Gamma_{t}(\delta)\right)\right)}+\left\Vert \left(\tfrac{1}{\epsilon}R,\nabla R\right)\right\Vert _{L^{2}\left(0,T_{\epsilon};L^{2}\left(\Omega\backslash\Gamma_{t}(\delta)\right)\right)} & \le K\epsilon^{M-\frac{1}{2}},\label{eq:Main-est-a}\\
\epsilon^{\frac{3}{2}}\left\Vert \partial_{\mathbf{n}}R\right\Vert _{L^{2}\left(0,T_{\epsilon};L^{2}\left(\Gamma_{t}(\delta)\right)\right)}+\left\Vert R\right\Vert _{L^{\infty}\left(0,T_{\epsilon};H^{-1}(\Omega)\right)} & \leq K\epsilon^{M},\label{eq:Main-est-b}\\
\int_{\Omega_{T_{\epsilon}}}\epsilon\left|\nabla R\right|^{2}+\frac{1}\epsilon f''\left(c_{A}^{\epsilon}\right)R^{2}\d(x,t) & \leq K^{2}\epsilon^{2M},\label{eq:Main-est-c}\\
\epsilon^{\frac{1}{2}}\left\Vert \gamma R\right\Vert _{L^{\infty}\left(0,T_{\epsilon};L^{2}(\Omega)\right)}+\left\Vert \left(\epsilon\gamma\Delta R,\gamma\nabla R,\gamma R\left(\nabla R\right)\right)\right\Vert _{L^{2}\left(\Omega_{T_{\epsilon}}\right)} & \leq K\epsilon^{M}\label{eq:Main-est-d}
\end{align}
\end{subequations}for all $\epsilon\in(0,\epsilon_{0})$.
\end{assumption}

It is the aim of this article to show the following theorem and to
provide the additional structural information gathered in \cite[Subsection 4.1]{NSCH1}.
\begin{thm}
\label{thm:Main-Apprx-Structure}For every $\epsilon\in(0,1)$
there are
$\mathbf{v}_{A}^{\epsilon},\mathbf{w}_{1}^{\epsilon}\colon \Omega_{T_{0}}\rightarrow\mathbb{R}^{2}$, $c_{A}^{\epsilon},\mu_{A}^{\epsilon},p_{A}^{\epsilon}\colon\Omega_{T_{0}}\rightarrow\mathbb{R}$
and $\rs\colon \Omega_{T_{0}}\rightarrow\mathbb{R}^{2}$, $\rdiv,\rc,\rh\colon \Omega_{T_{0}}\rightarrow\mathbb{R}$
such that
\begin{alignat}{2}
-\Delta\mathbf{v}_{A}^{\epsilon}+\nabla p_{A}^{\epsilon} & =\mu_{A}^{\epsilon}\nabla c_{A}^{\epsilon}+\rs&\quad &\text{in }\Omega_{T_{0}},\label{eq:Stokesapp}\\
\mathrm{div}\mathbf{v}_{A}^{\epsilon} & =\rdiv&\quad &\text{in }\Omega_{T_{0}},\label{eq:Divapp}\\
\partial_{t}c_{A}^{\epsilon}+\left(\mathbf{v}_{A}^{\epsilon}+\epsilon^{M-\frac{1}{2}}\left.\mathbf{w}_{1}^{\epsilon}\right|_{\Gamma}\xi\left(d_{\Gamma}\right)\right)\cdot\nabla c_{A}^{\epsilon} & =\Delta\mu_{A}^{\epsilon}+\rc&\quad &\text{in }\Omega_{T_{0}},\label{eq:Cahnapp}\\
\mu_{A}^{\epsilon} & =-\epsilon\Delta c_{A}^{\epsilon}+\epsilon^{-1}f'\left(c_{A}^{\epsilon}\right)+\rh&\quad &\text{in }\Omega_{T_{0}}.\label{eq:Hilliardapp}
\end{alignat}
Furthermore, the boundary conditions
\begin{equation}
c_{A}^{\epsilon}=-1,\quad \mu_{A}^{\epsilon}=0,\quad \left(-2D_{s}\mathbf{v}_{A}^{\epsilon}+p_{A}^{\epsilon}\mathbf{I}\right)\mathbf{n}_{\partial\Omega}=\alpha_{0}\mathbf{v}_{A}^{\epsilon}, \quad \rdiv=0\quad \text{on }\partial_{T_{0}}\Omega\label{eq:boundaryconditions}
\end{equation}
are satisfied. If additionally
Assumption \ref{assu:Main-est} holds for $\epsilon_{0}\in(0,1)$,
$K\geq1$ and a family $\left(T_{\epsilon}\right)_{\epsilon\in(0,\epsilon_{0})}\subset(0,T_{0}]$,
then there are some $\epsilon_{1}\in(0,\epsilon_{0}]$,
$C(K)>0$ depending on $K$ and $C_{K}\colon (0,T_{0}]\times(0,1]\rightarrow(0,\infty)$
(also depending on $K$), which satisfies $C_{K}(T,\epsilon)\rightarrow0$
as $(T,\epsilon)\rightarrow0$, such that
\begin{align}
\int_{0}^{T_{\epsilon}}\left|\int_{\Omega}\rc(x,t)\varphi(x,t)\d x\right|\d t & \leq C_{K}(T_{\epsilon},\epsilon)\epsilon^{M}\left\Vert \varphi\right\Vert _{L^{\infty}\left(0,T_{\epsilon};H^{1}(\Omega)\right)},\label{eq:remcahn}\\
\int_{0}^{T_{\epsilon}}\left|\int_{\Omega}\rh(x,t)\left(c^{\epsilon}(x,t)-c_{A}^{\epsilon}(x,t)\right)\d x\right|\d t & \leq C_{K}(T_{\epsilon},\epsilon)\epsilon^{2M},\label{eq:remhill}\\
\left\Vert \rs\right\Vert _{L^{2}\left(0,T_{\epsilon};\left(H^{1}(\Omega)\right)'\right)}+\left\Vert \rdiv\right\Vert _{L^{2}\left(\Omega_{T_{\epsilon}}\right)} & \leq C(K)\epsilon^{M},\label{eq:remstokes}\\
\left\Vert \rh\nabla c_{A}^{\epsilon}\right\Vert _{L^{2}\left(0,T_{\epsilon};\left(H^{1}(\Omega)^{2}\right)'\right)} & \leq C(K)C(T_{\epsilon},\epsilon)\epsilon^{M}\label{eq:rch2-nablacae}\\
\left\Vert \rc\right\Vert _{L^{2}\left(\partial_{T_{\epsilon}}\Omega\left(\frac{\delta}{2}\right)\right)} & \leq C(K)\epsilon^{M}\label{eq:rch1-rch2-Linfbdry}
\end{align}
for all $\epsilon\in(0,\epsilon_{1})$ and $\varphi\in L^{\infty}\left(0,T_{\epsilon};H^{1}(\Omega)\right)$.
\end{thm}

This work is organized as follows: Section \ref{chap:Fundamentals}
gives a short overview over the needed mathematical tools, particularly
existence results for parabolic equations on $\Gamma$ and a short
summary of the differential geometric properties that will be needed
later on.

Section \ref{constrappr} is based on the approaches in \cite{nsac, abc,chenAC, schatz};
here we present results for the construction of inner, outer and boundary
terms of arbitrarily high order of the asymptotic expansions for solutions
of (\ref{eq:StokesPart})\textendash (\ref{eq:Dirichlet2}). Due to
constraints to the length of this contribution, many details 
are left out, but can be found in
\cite{ichPhD}. In Subsection \ref{sec:Estimating-the-error-vel},
we introduce the auxiliary function $\mathbf{w}_{1}^{\epsilon}$,
which turns out in \cite{NSCH1} to be a representation of the leading
term of the error in the velocity $\mathbf{v}_{A}^{\epsilon}-\mathbf{v}^{\epsilon}$.
Subsection \ref{sec:Constructing-the--M-0,5} is then concerned with
constructing fractional order terms in the asymptotic expansion, which
are defined with the help of solutions to a nonlinear evolution equation
involving $\mathbf{w}_{1}^{\epsilon}$ .

To rigorously justify that the ``approximate solutions'' constructed
in the work really are a good approximation of solutions,
it is necessary to estimate the remainder terms in Section~\ref{chap:Estimates-Remainder},
i.e., the functions $\rc$, $\rh$, $\rs$ and $\rdiv$ presented in
Theorem~\ref{thm:Main-Apprx-Structure}. Thus, in Section~\ref{chap:Estimates-Remainder},
we analyze these terms in detail, starting with a proper definition
of the involved approximate solutions and a subsequent structural
representation of $r_{CH1}^{\epsilon}$, etc. The facts that the terms
of fractional order may not be estimated uniformly in $\epsilon$
in arbitrarily strong norms and that there appear terms of relatively
low orders of $\epsilon$ in the representations of the remainder,
when discussing the region close to the interface, account for many
technical difficulties. The involved estimates rely heavily on Lemma~\ref{ABstruc},
which is a direct consequence of our construction scheme of the fractional
order terms. The actual proof for Theorem~\ref{thm:Main-Apprx-Structure}
is given at the end of this article.

\section{Preliminaries\label{chap:Fundamentals}}

\subsection{Differential-Geometric Background\label{sec:Diffgeo}}

The following overview was already given in \cite{NSCH1} in more
detail; due to the importance of the concepts in view of later considerations
in this article and for the sake of completeness, we give a brief
reminder.

We parameterize the curves $\left(\Gamma_{t}\right)_{t\in[0,T_0]}$
by choosing a family of smooth diffeomorphisms $X_{0}\colon \mathbb{T}^{1}\times[0,T_0]\rightarrow\Omega$
such that $\partial_{s}X_{0}(s,t)\neq0$ for all $s\in\mathbb{T}^{1}$,
$t\in[0,T_0]$. In particular $\bigcup_{t\in[0,T_0]}X_{0}\left(\mathbb{T}^{1}\times\left\{ t\right\} \right)\times\left\{ t\right\} =\Gamma$.
Moreover, we define the tangent and normal vectors on $\Gamma_{t}$
at $X_{0}(s,t)$ as 
\begin{equation}
\boldsymbol{\tau}(s,t):=\frac{\partial_{s}X_{0}(s,t)}{\left|\partial_{s}X_{0}(s,t)\right|}\text{ and }\mathbf{n}(s,t):=\left(\begin{array}{cc}
0 & -1\\
1 & 0
\end{array}\right)\boldsymbol{\tau}(s,t)\label{eq:ntau}
\end{equation}
for all $(s,t)\in\mathbb{T}^{1}\times[0,T_0]$.
We choose $X_{0}$ (and thereby the orientation of $\Gamma_{t}$)
such that $\mathbf{n}(.,t)$ is the exterior normal with
respect to $\Omega^{-}(t)$. Thus, for a point $p\in\Gamma_{t}$
with $p=X_{0}(s,t)$ it holds $\mathbf{n}_{\Gamma_{t}}(p)=\mathbf{n}(s,t)$.
Furthermore, $V(s,t):=V_{\Gamma_{t}}(X_{0}(s,t))$
and $H(s,t):=H_{\Gamma_{t}}(X_{0}(s,t))$
and $V(s,t)=\partial_{t}X_{0}(s,t)\cdot\mathbf{n}(s,t)$
for all $(s,t)\in\mathbb{T}^{1}\times[0,T_0]$
by definition of the normal velocity. We write for a function $\mathbf{v}\colon\Gamma\rightarrow\mathbb{R}^{d}$,
$d\in\mathbb{N}$, $\left(X_{0}^{*}\mathbf{v}\right)(s,t):=\mathbf{v}(X_{0}(s,t),t)$
for all $(s,t)\in\mathbb{T}^{1}\times[0,T_0]$
for a function $h\colon \mathbb{T}^{1}\times[0,T_0]\rightarrow\mathbb{R}$
we set $\big(X_{0}^{*,-1}h\big)(p):=h\big(X_{0}^{-1}(p)\big)$
for all $p\in\Gamma_{t}$, $t\in[0,T_0]$.

Choosing $\delta>0$ small enough, the orthogonal projection $\operatorname{Pr}_{\Gamma_{t}}\colon\Gamma_{t}\left(3\delta\right)\rightarrow\Gamma_{t}$
is well defined and smooth for all $t\in[0,T_0]$ and
the mapping $\phi_{t}(x)=\left(d_{\Gamma}(x,t),\operatorname{Pr}_{\Gamma_{t}}(x)\right)$
is a diffeomorphism from $\Gamma\left(3\delta\right)$ onto its image.
Its inverse is given by $\phi_{t}^{-1}\left(r,p\right)=p+r\mathbf{n}_{\Gamma_{t}}(p)$.
Although $\operatorname{Pr}_{\Gamma_{t}}$ and $\phi_{t}$ are well defined in $\Gamma_{t}(3\delta)$,
almost all computations later on are performed in $\Gamma_{t}(2\delta)$,
which is why, for the sake of readability, we work on $\Gamma_{t}(2\delta)$
in the following.

Combining $\phi_{t}^{-1}$ and $X_{0}$ we may define a diffeomorphism
\begin{align}
X(r,s,t) & =\left(\phi_{t}^{-1}(r,X_{0}(s,t)),t\right)=\left(X_{0}(s,t)+r\mathbf{n}(s,t),t\right)\label{eq:diffeo}
\end{align}
for $(r,s,t)\in\left(-2\delta,2\delta\right)\times\mathbb{T}^{1}\times[0,T_0]$
with inverse $X^{-1}(x,t)=\left(d_{\Gamma}(x,t),S(x,t),t\right)$
where we define 
\begin{equation}
S(x,t):=\left(X_{0}^{-1}(\operatorname{Pr}_{\Gamma_{t}}(x))\right)_{1}\label{eq:Sfett}
\end{equation}
for $(x,t)\in\Gamma(2\delta)$ and where $\left(.\right)_{1}$
signifies that we take the first component. In particular it holds
$S(x,t)=S(\operatorname{Pr}_{\Gamma_{t}}(x),t)$.
In the following we will write $\mathbf{n}(x,t):=\mathbf{n}(S(x,t),t)$
and $\boldsymbol{\tau}(x,t):=\boldsymbol{\tau}(S(x,t),t)$
for $(x,t)\in\Gamma(2\delta)$. 

For $(x,t)\in\Gamma(2\delta)$ it holds
\begin{equation}
\nabla d_{\Gamma}(x,t)=\mathbf{n}(x,t),\quad \left|\nabla d_{\Gamma}(x,t)\right|=1,\quad\nabla S(x,t)\cdot d_{\Gamma}(x,t)=0.\label{eq:Diff-Prop}
\end{equation}
In order to connect $d_{\Gamma}$ to the curvature and mean velocity,
we observe that for $s\in\mathbb{T}^{1}$, $r\in(-2\delta,2\delta)$
and $t\in[0,T_0]$ it holds $\Delta d_{\Gamma}(X_{0}(s,t),t)=-H(s,t)$
and $-\partial_{t}d_{\Gamma}(X(r,s,t))=V(s,t)$.

For a function $\phi\colon\Gamma(2\delta)\rightarrow\mathbb{R}$
we define $\tilde{\phi}(r,s,t):=\phi(X(r,s,t))$
and often write $\phi(r,s,t)$ instead of $\tilde{\phi}(r,s,t)$.
In the case that $\phi$ is twice continuously differentiable,
we introduce 
\begin{align}
\partial_{t}^{\Gamma}\tilde{\phi}(r,s,t) & :=\left(\partial_{t}+\partial_{t}S(X(r,s,t))\partial_{s}\right)\tilde{\phi}(r,s,t),\nonumber \\
\nabla^{\Gamma}\tilde{\phi}(r,s,t) & :=\nabla S(X(r,s,t))\partial_{s}\tilde{\phi}(r,s,t),\nonumber \\
\Delta^{\Gamma}\tilde{\phi}(r,s,t) & :=\left(\Delta S(X(r,s,t))\partial_{s}+\left(\nabla S\cdot\nabla S\right)(X(r,s,t))\partial_{ss}\right)\tilde{\phi}(r,s,t).\label{eq:surf-diff}
\end{align}
Similarly, if $\mathbf{v}\colon\Gamma(2\delta)\rightarrow\mathbb{R}^{2}$
is continuously differentiable, we will also write $\tilde{\mathbf{v}}(r,s,t):=\mathbf{v}\left(X(r,s,t)\right)$
and introduce
\begin{equation}
\mathrm{div}^{\Gamma}\tilde{\mathbf{v}}(r,s,t)=\nabla S\left(X(r,s,t)\right)\cdot\partial_{s}\tilde{\mathbf{v}}(r,s,t).\label{eq:divop}
\end{equation}
For later use we introduce
\begin{align*}
  \nabla^{\Gamma}\phi(x,t)&:=\nabla S(x,t)\partial_{s}\tilde{\phi}\left(d_{\Gamma}(x,t),S(x,t),t\right)\qquad \text{and}\\
\operatorname{div}^{\Gamma}\mathbf{v}(x,t)&:=\nabla S(x,t)\partial_{s}\tilde{\mathbf{v}}\left(d_{\Gamma}(x,t),S(x,t),t\right)
\end{align*}
for $(x,t)\in\Gamma(2\delta)$.

With these notations we have the decompositions 
\begin{align}
\nabla\phi(x,t) & =\partial_{\mathbf{n}}\phi(x,t)\mathbf{n}+\nabla^{\Gamma}\phi(x,t),\label{eq:surfgraddecomp}\\
\operatorname{div}\mathbf{v}(x,t) & =\partial_{\mathbf{n}}\mathbf{v}(x,t)\cdot\mathbf{n}+\operatorname{div}^{\Gamma}\mathbf{v}(x,t)\label{eq:surfdivdecomp}
\end{align}
for all $(x,t)\in\Gamma(2\delta)$, as 
\[
\frac{d}{dr}\left(\phi\circ X\right)|_{(r,s,t)=\left(d_{\Gamma}(x,t),S(x,t),t\right)}=\partial_{\mathbf{n}}\phi(x,t).
\]
\begin{rem}
\label{hnotations} If $h\colon\mathbb{T}^{1}\times[0,T_0]\rightarrow\mathbb{R}$
is a function that is independent of $r\in(-2\delta,2\delta)$,
the functions $\partial_{t}^{\Gamma}h,\nabla^{\Gamma}h$ and $\Delta^{\Gamma}h$
will nevertheless depend on $r$ via the derivatives of $S$. To connect
the presented concepts with the classical surface operators we introduce
the following notations:
\begin{align*}
D_{t,\Gamma}h(s,t) & :=\partial_{t}^{\Gamma}h(0,s,t),\quad
\nabla_{\Gamma}h(s,t)  :=\nabla^{\Gamma}h(0,s,t),\quad
\Delta_{\Gamma}h(s,t)  :=\Delta^{\Gamma}h(0,s,t).
\end{align*}
Later in this work (from Subsection \ref{subsec:The-Inner-Expansion}
on forward) we will often consider $h (S(x,t),t)$
and thus will write for simplicity
\begin{align}
\partial_{t}^{\Gamma}h(x,t) & :=\left(\partial_{t}+\partial_{t}S(x,t)\partial_{s}\right)h(S(x,t),t),\nonumber \\
\nabla^{\Gamma}h(x,t) & :=\left(\nabla S(x,t)\partial_{s}\right)h(S(x,t),t),\nonumber \\
\Delta^{\Gamma}h(x,t) & :=\left(\Delta S(x,t)\partial_{s}+\nabla S(x,t)\cdot\nabla S(x,t)\partial_{ss}\right)h(S(x,t),t)\label{eq:hsurf}
\end{align}
for $(x,t)\in\Gamma(2\delta)$. Using the definitions
and notations from this section we gain the identity
\begin{equation}
\partial_{t}^{\Gamma}h(x,t)=X_{0}^{*}\left(\partial_{t}^{\Gamma}h\right)(s,t)=\partial_{t}^{\Gamma}h(0,s,t)=D_{t,\Gamma}h(s,t)\label{eq:identities}
\end{equation}
for $(s,t)\in\mathbb{T}^{1}\times[0,T_0]$
and $\left(X_{0}(s,t),t\right)=(x,t)\in\Gamma$.
This might seem cumbersome but turns out to be convenient throughout
this work.
\end{rem}

In later parts of this article, we will introduce stretched coordinates
of the form 
\begin{equation}
\rho^{\epsilon}(x,t)=\frac{d_{\Gamma}(x,t)-\epsilon h(S(x,t),t)}{\epsilon}\label{eq:diffpartrho}
\end{equation}
for $(x,t)\in\Gamma(2\delta)$, $\epsilon\in\left(0,1\right)$
and for some smooth function $h\colon \mathbb{T}^{1}\times[0,T_0]\rightarrow\mathbb{R}$
(which will later on also depend on $\epsilon$). Writing $\rho=\rho^{\epsilon}$,
the relation between the regular and the stretched variables can be
expressed as
\begin{equation}
\hat{X}(\rho,s,t):=X(\epsilon\left(\rho+h(s,t)\right),s,t)=\left(X_{0}(s,t)+\epsilon\left(\rho+h(s,t)\right)\mathbf{n}(s,t),t\right).\label{eq:Xhat}
\end{equation}
\begin{lem}
\label{lem:surfright} Let $\phi\colon\mathbb{R}\times\Gamma(2\delta)\rightarrow\mathbb{R}$
be twice continuously differentiable and let $\rho$ be given
as in (\ref{eq:diffpartrho}). Then the following formulas hold for
$(x,t)\in\Gamma(2\delta)$ and $\epsilon\in\left(0,1\right)$
\begin{align*}
\partial_{t}\left(\phi(\rho(x,t),x,t)\right) & =\left(-\epsilon^{-1}V(S(x,t),t)-\partial_{t}^{\Gamma}h(x,t)\right)\partial_{\rho}\phi(\rho(x,t),x,t)+\partial_{t}\phi(\rho(x,t),x,t),\\
\nabla\left(\phi(\rho(x,t),x,t)\right) & =\left(\epsilon^{-1}\mathbf{n}(S(x,t),t)-\nabla^{\Gamma}h(x,t)\right)\partial_{\rho}\phi(\rho(x,t),x,t)+\nabla_{x}\phi(\rho(x,t),x,t),\\
\Delta\left(\phi(\rho(x,t),x,t)\right) & =\left(\epsilon^{-2}+\left|\nabla^{\Gamma}h(x,t)\right|^{2}\right)\partial_{\rho\rho}\phi(\rho(x,t),x,t)+\Delta_{x}\phi\left(\rho(x,t),x,t\right)\\
 & \quad+\left(\epsilon^{-1}\Delta d_{\Gamma}(x,t)-\Delta^{\Gamma}h(x,t)\right)\partial_{\rho}\phi(\rho(x,t),x,t)\\
 & \quad+2\left(\epsilon^{-1}\mathbf{n}(S(x,t),t)-\nabla^{\Gamma}h(x,t)\right)\cdot\nabla_{x}\partial_{\rho}\phi(\rho(x,t),x,t),
\end{align*}
Here $\nabla_{x}$ and $\Delta_{x}$ operate solely on the $x$-variable
of $\phi$.
\end{lem}

\begin{proof}
This follows from the chain rule, (\ref{eq:Diff-Prop}) and the notations
introduced in Remark \ref{hnotations}.
\end{proof}

\subsection{Remainder Terms \label{sec:Remainder}}

Most of the following was already discussed in \cite{NSCH1} and is
only presented for the convenience of the reader.

For $t\in[0,T_0]$ and $1\leq p<\infty$ we define
\[
L^{p,\infty}\left(\Gamma_{t}(2\delta)\right):=\left\{ \left.f:\Gamma_{t}(2\delta)\rightarrow\mathbb{R}\;\text{measurable}\right|\left\Vert f\right\Vert _{L^{p,\infty}\left(\Gamma_{t}(2\delta)\right)}<\infty\right\} ,
\]
where 
\[
\left\Vert f\right\Vert _{L^{p,\infty}\left(\Gamma_{t}(2\delta)\right)}:=\left(\int_{\mathbb{T}^{1}}\text{esssup}_{\left|r\right|\leq2\delta}\left|f\left(\left(X(r,s,t)\right)_{1}\right)\right|^{p}\d s\right)^{\frac{1}{p}}.
\]
Here $X_{1}(r,s,t):=X_{0}(s,t)+r\mathbf{n}(s,t)$
denotes the first component of $X$. Let $T\in[0,T_0]$, $1\leq p,q<\infty$ and $\alpha\in\left(0,3\delta\right)$
be given and let . Then we set 
\[
L^{q}\left(0,T;L^{p}\left(\Gamma_{t}\left(\alpha\right)\right)\right):=\left\{ \left.f:\Gamma\left(\alpha,T\right)\rightarrow\mathbb{R}\;\text{measurable}\right|\left\Vert f\right\Vert _{L^{q}\left(0,T;L^{p}\left(\Gamma_{t}\left(\alpha\right)\right)\right)}<\infty\right\} ,
\]
\[
\left\Vert f\right\Vert _{L^{q}\left(0,T;L^{p}\left(\Gamma_{t}\left(\alpha\right)\right)\right)}:=\left(\int_{0}^{T}\left(\int_{\Gamma_{t}\left(\alpha\right)}\left|f(x,t)\right|^{p}\d x\right)^{\frac{q}{p}}\d t\right)^{\frac{1}{q}}.
\]
Analogously, we define $L^{q}\left(0,T;L^{p}\left(\Omega\backslash\Gamma_{t}(\alpha)\right)\right)$ and $L^{q}\left(0,T;L^{p}\left(\Omega^\pm(t))\right)\right)$
and the corresponding norms. Furthermore, for $m\in\mathbb{N}_0$ we define for $U(t)= \Omega^\pm (t)$ or $U(t)= \Gamma_t(\alpha)$
\begin{align*}
  L^p(0,T;H^m(U(t)))&:=\{f\in L^p(0,T;L^2(\Omega^\pm(t))): \partial_x^\alpha f \in L^p(0,T;L^2(U(t)))\forall |\alpha|\leq m  \},\\
  \|f\|_{L^p(0,T;H^m(U(t)))} &:= \sum_{|\alpha|\leq m} \|\partial_x^\alpha f\|_{L^p(0,T;L^2(U(t)))}.
\end{align*}

The following embedding was already remarked in \cite[Subsection~2.5]{nsac}.
\begin{lem}
\label{L4inf} We have
$
H^{1}\left(\Gamma_{t}(2\delta)\right)\hookrightarrow L^{4,\infty}\left(\Gamma_{t}(2\delta)\right)
$
with operator norm uniformly bounded with respect to $t\in [0,T_0]$.
\end{lem}

The following estimates will be frequently used:
\begin{lem}
\label{Linfeig}Let $h\colon \mathbb{T}^{1}\times[0,T_0]\rightarrow\mathbb{R}$
be continuous, $\epsilon\in (0,1)$, $t\in[0,T_0]$.
Then there are constants $C_{1},C_{2}>0$ independent of $h$, $\epsilon$
and $t$ such that 
\begin{enumerate}
\item for all $\psi\in L^{1,\infty}(\Gamma_{t}(2\delta))$,
$\eta\in L^{1}(\mathbb{R})$
\[
\left\Vert \eta\left(\tfrac{d_{\Gamma}(.,t)}{\epsilon}-h(S(.,t),t)\right)\psi\right\Vert _{L^{1}\left(\Gamma_{t}(2\delta)\right)}\!\leq C_{1}\epsilon\left\Vert \eta\right\Vert _{L^{1}(\mathbb{R})}\left\Vert \psi\right\Vert _{L^{1,\infty}\left(\Gamma_{t}(2\delta)\right)}.
\]
\item for all $\psi\in L^{2,\infty}(\Gamma_{t}(2\delta))$,
$\eta\in L^{2}(\mathbb{R})$ and $u\in L^{2}(\Gamma_{t}(2\delta))$
\[
\left\Vert \eta\left(\tfrac{d_{\Gamma}(.,t)}{\epsilon}-h(S(.,t),t)\right)\psi u\right\Vert _{L^{1}\left(\Gamma_{t}(2\delta)\right)}\leq C_{2}\epsilon^{\frac{1}{2}}\left\Vert \eta\right\Vert _{L^{2}(\mathbb{R})}\left\Vert \psi\right\Vert _{L^{2,\infty}\left(\Gamma_{t}(2\delta)\right)}\left\Vert u\right\Vert _{L^{2}\left(\Gamma_{t}(2\delta)\right)}.
\]
\end{enumerate}
\end{lem}

\begin{proof}
Ad 1.: With two changes of variables we obtain 
\begin{align*}
 & \left\Vert \eta\left(\tfrac{d_{\Gamma}(.,t)}{\epsilon}-h(S(.,t),t)\right)\psi\right\Vert _{L^{1}(\Gamma_{t}(2\delta))}\\
 & =\int_{\mathbb{T}^{1}}\int_{-2\delta}^{2\delta}\left|\eta\left(\tfrac{r}{\epsilon}-h(s,t)\right)\psi(X_{1}(r,s,t))\right|\left|\det\left(\nabla X_{1}(r,s,t)\right)\right|\d r\d s\\
 & \leq C\int_{\mathbb{T}^{1}}\left\Vert \psi\circ X_{1}\right\Vert _{L^{\infty}\left(-2\delta,2\delta\right)}\int_{-\frac{2\delta}{\epsilon}-h(s,t)}^{\frac{2\delta}{\epsilon}-h(s,t)}\epsilon\left|\eta(\rho)\right|\d\rho\d s \leq C\epsilon\left\Vert \psi\right\Vert _{L^{1,\infty}\left(\Gamma_{t}(\delta)\right)}\left\Vert \eta\right\Vert _{L^{1}(\mathbb{R})}.
\end{align*}
Here we used the uniform boundedness of $\left|\det(\nabla X_{1})\right|$
in $\left(-2\delta,2\delta\right)\times\mathbb{T}^{1}\times[0,T_0]$
in the second inequality.

Ad 2.: This can be shown in the same way as the first statement.
\end{proof}
For future use, we introduce the concept of remainder terms, similar
to \cite[Definition 2.5]{nsac}.
\begin{defn}
\label{def:Ralphadef}Let $n\in\mathbb{N}$, $\epsilon_{0}>0$. For
$\alpha>0$ let $\ra$ denote the vector space of all families $\left(\hat{r}_{\epsilon}\right)_{\epsilon\in\left(0,\epsilon_{0}\right)}$
of continuous functions $\hat{r}_{\epsilon}\colon\mathbb{R}\times\Gamma(2\delta)\rightarrow\mathbb{R}^{n}$
which satisfy
\[
\left|\hat{r}_{\epsilon}(\rho,x,t)\right|\leq Ce^{-\alpha\left|\rho\right|}\qquad \text{ for all }\rho\in\mathbb{R},(x,t)\in\Gamma(2\delta),\epsilon\in(0,1).
\]
Moreover, let $\mathcal{R}_{\alpha}^{0}$ be the subspace of all $\left(\hat{r}_{\epsilon}\right)_{\epsilon\in\left(0,\epsilon_{0}\right)}\in\mathcal{R}_{\alpha}$
such that 
\[
\hat{r}_{\epsilon}(\rho,x,t)=0\qquad\text{ for all }\rho\in\mathbb{R},(x,t)\in\Gamma.
\]
\end{defn}

\subsection{Parabolic Equations on Evolving Surfaces \label{sec:Parabolic-Equations-on}}

We introduce the space 
\begin{eqnarray}
X_{T} & = & L^{2}\big(0,T;H^{\frac{7}{2}}(\mathbb{T}^{1})\big)\cap H^{1}\big(0,T;H^{\frac{1}{2}}(\mathbb{T}^{1})\big)\label{eq:XT}
\end{eqnarray}
for $T\in(0,\infty)$, where we equip
$X_{T}$ with the norm 
\[
\left\Vert h\right\Vert _{X_{T}}=\left\Vert h\right\Vert _{L^{2}(0,T;H^{\frac{7}{2}}(\mathbb{T}^{1}))}+\left\Vert h\right\Vert _{H^{1}(0,T;H^{\frac{1}{2}}(\mathbb{T}^{1}))}+\left\Vert h|_{t=0}\right\Vert _{H^{2}(\mathbb{T}^{1})}.
\]
\begin{prop}
\label{embedding}Let $T\in\left(0,\infty\right)$. Then we have
\begin{enumerate}
\item $X_{T}\hookrightarrow C^{0}([0,T];H^{2}(\mathbb{T}^{1}))$
where the operator norm of the embedding is bounded independently
of $T$,
\item $X_{T}\hookrightarrow H^{\frac{1}{2}}(0,T;H^{2}(\mathbb{T}^{1}))\cap  H^{\frac{1}{3}}(0,T;H^{\frac{5}{2}}(\mathbb{T}^{1}))$.
\end{enumerate}
\end{prop}

\begin{proof}

Ad 1.: See e.g.\ \cite[Lemma~A.8]{moser}.

Ad 2.: According to \cite[Proposition~3.2]{schnaubelt}  we have
$X_{T}\hookrightarrow H^{\sigma}\big(0,T;H^{\frac{1}{2}+\left(1-\sigma\right)3}(\mathbb{T}^{1})\big)$.
Thus the statement follows for $\sigma=\frac{1}{2}$ and $\sigma=\frac{1}{3}$.
\end{proof}
 The following result on solvability of a linearized Mullins-Sekerka/Stokes system is shown in \cite{ichPhD} and in a more general form in \cite{LinStMS} and will be important for the construction of the approximate solution.
 \begin{thm}
   \label{stokesthe}
  Let $T\in (0,T_{0}]$ and $t\in [0,T]$.
For every $\mathbf{f}\in L^{2}(\Omega)^{2}$, $\mathbf{s}\in H^{\frac{3}{2}}(\Gamma_{t})^{2}$, $\mathbf{a}\in H^{\frac{1}{2}}(\Gamma_{t})^{2}$
and $\mathbf{g}\in H^{\frac{1}{2}}\left(\partial\Omega\right)^{2}$
the system 
\begin{align}
-\Delta\mathbf{v}^{\pm}+\nabla p^{\pm} & =\mathbf{f} &  & \text{in }\Omega^{\pm}(t),\label{eq:zweiphasen1}\\
\mathrm{div}\mathbf{v}^{\pm} & =0 &  & \text{in }\Omega^{\pm}(t),\\
\left(-2D_{s}\mathbf{v}^{-}+p^{-}\mathbf{I}\right)\mathbf{n}_{\partial\Omega} & =\alpha_{0}\mathbf{v}^{-}+\mathbf{g} &  & \text{on }\partial\Omega,\\
\left[\mathbf{v}\right] & =\mathbf{s} &  & \text{on }\Gamma_{t},\\
\left[2D_{s}\mathbf{v}-p^{-}\mathbf{I}\right]\mathbf{n}_{\Gamma_{t}} & =\mathbf{a} &  & \text{on }\Gamma_{t}\label{eq:zweiphasen5}
\end{align}
has a unique solution $\left(\mathbf{v}^{\pm},p^{\pm}\right)\in H^{2}(\Omega^{\pm}(t))\times H^{1}(\Omega^{\pm}(t))$.
Moreover, there is a constant $C>0$ independent of $t\in[0,T_0]$
such that
\begin{equation}
\left\Vert \left(\mathbf{v},p\right)\right\Vert _{H^{2}(\Omega^{\pm}(t))\times H^{1}(\Omega^{\pm}(t))}\!\leq\!C\!\left(\left\Vert \mathbf{f}\right\Vert _{L^{2}(\Omega)}+\left\Vert \mathbf{s}\right\Vert _{H^{\frac{3}{2}}(\Gamma_{t})}+\left\Vert \mathbf{a}\right\Vert _{H^{\frac{1}{2}}(\Gamma_{t})}+\left\Vert \mathbf{g}\right\Vert _{H^{\frac{1}{2}}(\partial\Omega)}\right)\label{eq:instoab}
\end{equation}
holds.
\end{thm}
\begin{proof}
  See \cite[Theorem~2.36]{ichPhD} or \cite{LinStMS}.
\end{proof}
\begin{thm}
\label{scharfthe} Let $T\in (0,T_{0}]$. Let $\mathbf{b}\colon \mathbb{T}^{1}\times [0,T]\rightarrow\mathbb{R}^{2}$,
$b\colon \mathbb{T}^{1}\times [0,T]\rightarrow\mathbb{R}$, $a_{1}\colon \Omega\times [0,T]\rightarrow\mathbb{R}$,
$a_{2},a_{3},a_{5}\colon \Gamma\rightarrow\mathbb{R}$, $a_{4}\colon \partial_{T}\Omega\rightarrow\mathbb{R}$,
$\mathbf{a}_{1}\colon \Omega\times [0,T]\rightarrow\mathbb{R}^{2}$,
$\mathbf{a}_{2},\mathbf{a}_{3},\mathbf{a}_{4},\mathbf{a}_{5}\colon\Gamma\rightarrow\mathbb{R}^{2}$
and $\mathbf{a}_{6}\colon \partial_{T}\Omega\rightarrow\mathbb{R}^{2}$
be smooth given functions. For every $g\in L^{2}\big(0,T;H^{\frac{1}{2}}(\mathbb{T}^1)\big)$
and $h_{0}\in H^{2}(\mathbb{T}^1)$ there exists a unique
solution $h\in X_{T}$ of 
\begin{align*}
D_{t,\Gamma}h+\mathbf{b}\cdot\nabla_{\Gamma}h-bh+\tfrac{1}{2}X_{0}^{*}\big((\mathbf{v}^{+}+\mathbf{v}^{-})\cdot\mathbf{n}_{\Gamma_{t}}\big)+\tfrac{1}{2}X_{0}^{*}\big(\big[\partial_{\mathbf{n}_{\Gamma_{t}}}\mu\big]\big) & =g &  & \text{in }\mathbb{T}^{1}\times\left(0,T\right),\\
h\left(.,0\right) & =h_{0} &  & \text{in }\mathbb{T}^{1},
\end{align*}
where for every $t\in\left[0,T\right]$, the functions $\mathbf{v}^{\pm}=\mathbf{v}^{\pm}(x,t)$,
$p^{\pm}=p^{\pm}(x,t)$ and $\mu^{\pm}=\mu^{\pm}(x,t)$
for $(x,t)\in\Omega_{T}^{\pm}$ with $\mathbf{v}^{\pm}\in H^{2}(\Omega^{\pm}(t))$,
$p^{\pm}\in H^{1}(\Omega^{\pm}(t))$ and $\mu^{\pm}\in H^{2}(\Omega^{\pm}(t))$
are the unique solutions to 
\begin{align}
\Delta\mu^{\pm} & =a_{1} &  & \text{in }\Omega^{\pm}(t),\label{eq:laplmu}\\
\mu^{\pm} & = X_{0}^{*,-1}\big(\sigma\Delta_{\Gamma}h\pm a_{2}h\big)+a_{3} &  & \text{on }\Gamma_{t},\\
\mu^{-} & =a_{4} &  & \text{on }\partial\Omega,\\
-\Delta\mathbf{v}^{\pm}+\nabla p^{\pm} & =\mathbf{a}_{1} &  & \text{in }\Omega^{\pm}(t),\label{eq:stokes}\\
\operatorname{div}\mathbf{v}^{\pm} & =0 &  & \text{in }\Omega^{\pm}(t),\\
[\mathbf{v}] & =\mathbf{a}_{2} &  & \text{on }\Gamma_{t},\label{eq:stokesjmp}\\
\left[2D_{s}\mathbf{v}-p\mathbf{I}\right]\mathbf{n}_{\Gamma_{t}} & =X_{0}^{*,-1}\big(\mathbf{a}_{3}h+\mathbf{a}_{4}\Delta_{\Gamma}h+a_{5}\nabla_{\Gamma}h+\mathbf{a}_{5}\big) &  & \text{on }\Gamma_{t},\label{eq:stokesbdry}\\
\left(-2D_{s}\mathbf{v}^{-}+p^{-}\mathbf{I}\right)\mathbf{n}_{\partial\Omega} & =\alpha_{0}\mathbf{v}^{-}+\mathbf{a}_{6} &  & \text{on }\partial\Omega.\label{eq:stokesoutbdry}
\end{align}
Moreover, if $g$, $h_{0}$ and $b$, $\mathbf{b}$, $a_{i}$, and
$\mathbf{a}_{j}$ are smooth on their respective domains for $i\in\left\{ 1,\ldots,5\right\} $, $j\in\left\{ 1,\ldots,6\right\}$,
then $h$ is smooth and $p^{\pm}$, $\mathbf{v}^{\pm}$ and $\mu^{\pm}$
are smooth on $\Omega^{\pm}(t)$. 
\end{thm}
\begin{proof}
  See \cite[Theorem~2.37]{ichPhD} or \cite{LinStMS}.
\end{proof}
We note that, if
$\mu|_{\Omega^{\pm}(t)}\in H^{2}(\Omega^{\pm}(t))$,
for $t\in[0,T]$, is determined by \begin{subequations}\label{eq:musystem}
\begin{align}
\Delta\mu^{\pm} & =0 &  & \text{in }\Omega^{\pm}(t),\label{eq:musystem1}\\
\mu^{\pm} & =X_{0}^{*,-1}(\sigma\Delta_{\Gamma}h\pm b_{2}h) &  & \text{on }\Gamma_{t},\label{eq:musystem2}\\
\mu^{-} & =0 &  & \text{on }\partial\Omega,\label{eq:musystem3}
\end{align}
\end{subequations}
then by standard results for elliptic equations the estimate 
\begin{align}
  \sum_{\pm}\left\Vert \mu^{\pm}\right\Vert _{L^{2}\left(0,T;H^{2}\left(\Omega^{\pm}(t)\right)\right)\cap L^{6}\left(0,T;H^{1}\left(\Omega^{\pm}(t)\right)\right)} & \leq C\left\Vert h\right\Vert _{X_{T}},\label{eq:mumaxab}
\end{align}
holds true for some constant $C>0$ independent of $\mu$ and $h$.

\subsection{Spectral Theory}

In order to be able to access the results from \cite{NSCH1}, Subsection
2.4, we will need to show that our approximate solution $c_{A}^{\epsilon}$
has certain properties. For the readability of presentation, we repeat
these assumptions here.
\begin{assumption}
\label{assu:Spektral}Let $\epsilon\in (0,\epsilon_{0})$,
$T\in (0,T_{0}]$ and $\xi$ be a cut-off function satisfying
\eqref{eq:cut-off}. We assume that $c_{A}^{\epsilon}\colon \Omega_{T}\rightarrow\mathbb{R}$
is a smooth function, which has the structure 
\begin{align}
c_{A}^{\epsilon}(x,t) & =\xi(d_{\Gamma}(x,t))\left(\theta_{0}(\rho(x,t))+\epsilon p^{\epsilon}(\operatorname{Pr}_{\Gamma_{t}}(x),t)\theta_{1}(\rho(x,t))\right)+\xi(d_{\Gamma}(x,t))\epsilon^{2}q^{\epsilon}(x,t)\nonumber \\
 & \quad+\left(1-\xi(d_{\Gamma}(x,t))\right)\Big(c_{A}^{\epsilon,+}(x,t)\chi_{\Omega_{T_{0}}^{+}}(x,t)+c_{A}^{\epsilon,-}(x,t)\chi_{\Omega_{T_{0}}^{-}}(x,t)\Big)\label{eq:fie}
\end{align}
for all $(x,t)\in\Omega_{T}$, where $\rho(x,t):=\frac{d_{\Gamma}(x,t)}{\epsilon}-h^{\epsilon}(S(x,t),t)$.
The occurring functions are supposed to be smooth and satisfy for
some $C^{*}>0$ the following properties:
$\theta_{1}\colon \mathbb{R}\rightarrow\mathbb{R}$ is a bounded function
satisfying 
\begin{equation}
\int_{\mathbb{R}}\theta_{1}(\rho)\left(\theta_{0}'(\rho)\right)^{2}f^{\left(3\right)}(\theta_{0}(\rho))\d\rho=0.\label{eq:teta0}
\end{equation}
 Furthermore, $p^{\epsilon}\colon \Gamma\rightarrow\mathbb{R}$, $q^{\epsilon}\colon\Gamma(2\delta)\rightarrow\mathbb{R}$
satisfy 
\begin{equation}
\sup_{\epsilon\in\left(0,\epsilon_{0}\right)}\sup_{(x,t)\in\Gamma\left(2\delta;T\right)}\Big(\left|p^{\epsilon}(\operatorname{Pr}_{\Gamma_{t}}(x),t)\right|+\frac{\epsilon}{\epsilon+\left|d_{\Gamma}(x,t)-\epsilon h^{\epsilon}(S(x,t),t)\right|}\left|q^{\epsilon}(x,t)\right|\Big)\leq C^{*},\label{eq:pqab}
\end{equation}
$h^{\epsilon}\colon \mathbb{T}^{1}\times\left[0,T\right]\rightarrow\mathbb{R}$
fulfills
\begin{equation}
\sup_{\epsilon\in\left(0,\epsilon_{0}\right)}\sup_{(s,t)\in\mathbb{T}^{1}\times\left[0,T\right]}\left(\left|h^{\epsilon}(s,t)\right|+\left|\partial_{s}h^{\epsilon}(s,t)\right|\right)\leq C^{*}\label{eq:hbes}
\end{equation}
and $c_{A}^{\epsilon,\pm}\colon \Omega_{T}^{\pm}\rightarrow\mathbb{R}$
satisfy 
\begin{equation}
\pm c_{A}^{\epsilon,\pm}>0\text{ in }\Omega_{T}^{\pm}.\label{eq:caepm}
\end{equation}
 Additionally, we suppose that 
\begin{align}
&\sup_{\epsilon\in\left(0,\epsilon_{0}\right)}\left(\sup_{(x,t)\in\Omega_{T}}\left|c_{A}^{\epsilon}(x,t)\right|+\sup_{x\in\Gamma\left(\delta\right)}\left|\nabla^{\Gamma}c_{A}^{\epsilon}(x,t)\right|\right)\leq C^{*},\label{eq:obfi}\\
&\inf_{\epsilon\in\left(0,\epsilon_{0}\right)}\inf_{(x,t)\in\Omega_{T}\backslash\Gamma\left(\delta;T\right)}f''\left(c_{A}^{\epsilon}(x,t)\right)\geq\frac{1}{C^{*}}\label{eq:fnachunten}
\end{align}
holds. 
\end{assumption}

\section{Construction of Approximate Solutions\label{constrappr}}

In the following we use the method of matched asymptotic expansions
to construct approximate solutions $\left(c_{A}^{\epsilon},\mu_{A}^{\epsilon},\mathbf{v}_{A}^{\epsilon},p_{A}^{\epsilon}\right)$
of (\ref{eq:StokesPart})\textendash (\ref{eq:Dirichlet2}). Throughout
this chapter the formalism ``$\approx$'' will represent a formal
asymptotic expansion ansatz, that is, writing $u^{\epsilon}\approx\sum_{k\geq\text{0}}\epsilon^{k}u_{k}$
means that for every integer $K\in\mathbb{N}$ we have 
\begin{equation}
u^{\epsilon}=\sum_{k=0}^{K}\epsilon^{k}u_{k}+\tilde{u}_{K+1}\epsilon^{K+1},\label{eq:asympansatz}
\end{equation}
 where $\tilde{u}_{K+1}$ is uniformly bounded in $\epsilon$. 

\subsection{The First $M+1$ Terms\label{sec:The-first-M-1}}

Many of the following steps are based on ideas taken from \cite{abc},
\cite{nsac} and \cite{chenAC}. In order to present the results for
the construction of terms of arbitrarily high order in Lemmata \ref{zeroorder}
and \ref{k-thorder}, we devise an inductive scheme similar to the
approach in \cite{abc}. However, in favor of the brevity of presentation,
we did not include this scheme in this article and simply state the
results. For the background of the construction and the proofs, see
\cite{ichPhD}.

\subsubsection{The Outer Expansion \label{subsec:The-Outer-Expansion}}

We assume that in $\Omega_{T_{0}}^{\pm}\backslash\Gamma(2\delta)$
the solutions of (\ref{eq:StokesPart})\textendash (\ref{eq:Dirichlet2})
have the expansions
\begin{align}
c^{\epsilon}(x,t) & \approx\sum_{k\geq0}\epsilon^{k}c_{k}^{\pm}(x,t), & \mu^{\epsilon}(x,t) & \approx\sum_{k\geq0}\epsilon^{k}\mu_{k}^{\pm}(x,t),\nonumber \\
\mathbf{v}^{\epsilon}(x,t) & \approx\sum_{k\geq0}\epsilon^{k}\mathbf{v}_{k}^{\pm}(x,t), & p^{\epsilon}(x,t) & \approx\sum_{k\geq0}\epsilon^{k}p_{k}^{\pm}(x,t),\label{eq:OutExp}
\end{align}
where $c_{k}^{\pm}$, $\mu_{k}^{\pm}$, $\mathbf{v}_{k}^{\pm}$
and $p_{k}^{\pm}$ are smooth functions defined in $\Omega_{T_{0}}^{\pm}$.
Plugging this ansatz into (\ref{eq:StokesPart}), (\ref{eq:StokesPart2}),
(\ref{eq:CH-Part1}) and (\ref{eq:CH-Part2}) yields
\begin{align}
-\sum_{k\geq0}\epsilon^{k}\Delta\mathbf{v}_{k}^{\pm}+\sum_{k\geq0}\epsilon^{k}\nabla p_{k}^{\pm} & =\sum_{k,j\geq0}\epsilon^{k}\mu_{k}^{\pm}\epsilon^{j}\nabla c_{j}^{\pm},\label{eq:Stokes-Outer1}\\
\sum_{k\geq0}\epsilon^{k}\mbox{div}\mathbf{v}_{k}^{\pm}&=0,\label{eq:Stokes-Outer2}\\
\sum_{k\geq0}\epsilon^{k}\partial_{t}c_{k}^{\pm}+\Big(\sum_{k\geq0}\epsilon^{k}\mathbf{v}_{k}^{\pm}\Big)\cdot\Big(\sum_{k\geq0}\epsilon^{k}\nabla c_{k}^{\pm}\Big) & =\Big(\sum_{k\geq0}\epsilon^{k}\Delta\mu_{k}^{\pm}\Big),\label{eq:CH-Outer1}
\end{align}
and 
\begin{align}
\sum_{k\geq0}\epsilon^{k}\mu_{k}^{\pm} & =-\epsilon\sum_{k\geq0}\epsilon^{k}\Delta c_{k}^{\pm}+\frac{1}{\epsilon}f'(c_{0}^{\pm})+f''\left(c_{0}^{\pm}\right)\sum_{k\geq1}\epsilon^{k-1}c_{k}^{\pm} +\sum_{k\geq1}\epsilon^{k}f_{k}(c_{0}^{\pm},\ldots,c_{k}^{\pm}),\label{eq:CH-Outer2f}
\end{align}
where for fixed $c_{0}^{\pm}$ the functions $f_{k}$ are polynomials
in $(c_{1}^{\pm},\ldots,c_{k}^{\pm})$ and are the result
of a Taylor expansion. Moreover, $f_{k}(c_{0}^{\pm},\ldots,c_{k}^{\pm})$
are chosen such that they do not depend on $\epsilon$.
 Matching the $\mathcal{O}(\epsilon^{-1})$ terms yields
$f'(c_{0}^{\pm})=0$ and in view of the Dirichlet boundary
data for $c^{\epsilon}$ we set 
\begin{equation}
c_{0}^{\pm}=\pm1.\label{eq:c0out}
\end{equation}
Comparing the higher order terms $\mathcal{O}\left(\epsilon^{k}\right)$,
where $k\geq1$, yields:
\begin{alignat}{2}
  c_{k}^{\pm}&=\frac{\mu_{k-1}^{\pm}+\Delta c_{k-2}^{\pm}-f_{k-1}(c_{0}^{\pm},\ldots,c_{k-1}^{\pm})}{f''(\pm1)}&\qquad&\text{in }\Omega_{T_{0}}^{\pm},\label{eq:Outercdefine}\\
  \Delta\mu_{k}^{\pm}&=\partial_{t}c_{k}^{\pm}+\sum_{j=0}^{k}\mathbf{v}_{j}^{\pm}\cdot\nabla c_{k-j}^{\pm}&&\text{in }\Omega_{T_{0}}^{\pm},\label{eq:Outermudefine}\\
  -\Delta\mathbf{v}_{k}^{\pm}+\nabla p_{k}^{\pm} & =\sum_{j=0}^{k-1}\mu_{j}^{\pm}\nabla c_{k-j}^{\pm} &  & \text{in }\Omega_{T_{0}}^{\pm}, \\
\operatorname{div}\mathbf{v}_{k}^{\pm} & =0 &  & \text{in }\Omega_{T_{0}}^{\pm}.\label{eq:Outervpdefine}
\end{alignat}
\begin{rem}
\label{Outer-Rem}
\begin{enumerate}
\item As we will only construct $c_{0}^{\pm},\ldots,c_{M+1}^{\pm}$, we
need to consider the remainder of the Taylor expansion of $f'$. In
this case, we choose to expand $f'$ up to order $M+2$ and get 
\begin{align*}
f'\Big(\sum_{k=0}^{M+1}\epsilon^{k}c_{k}^{\pm}\Big) & =f'(c_{0}^{\pm})+\epsilon f''(c_{0}^{\pm})\sum_{k=1}^{M+1}\epsilon^{k-1}c_{k}^{\pm}+\epsilon^{2}\sum_{k=1}^{M}\epsilon^{k-1}f_{k}(c_{0}^{\pm},\ldots,c_{k}^{\pm})\\
 & \quad+\epsilon^{M+2}\tilde{f}_{\epsilon}(c_{0}^{\pm},\ldots,c_{M+1}^{\pm}).
\end{align*}
Here $\tilde{f}_{\epsilon}(c_{0}^{\pm},\ldots,c_{M+1}^{\pm})$
consists of polynomials in $(c_{1}^{\pm},\ldots,c_{M+1}^{\pm})$,
which may be of even higher order in $\epsilon$ and which are either
multiplied by $f^{\left(j\right)}(c_{0}^{\pm})$ for $j\in\left\{ 2,\ldots,M+1\right\} $
or by $f^{\left(M+2\right)}\big(\xi(c_{0}^{\pm},\ldots,c_{M+1}^{\pm})\big)$
for suitable $\xi\in\big[c_{0}^{\pm},\sum_{k=0}^{M+1}\epsilon^{k}c_{k}^{\pm}\big]$.
If $c_{k}^{\pm}\in L^{\infty}(\Omega_{T_{0}}^{\pm})$ for
all $k\in\left\{ 0,\ldots,M+1\right\} $, it holds
\[
\Vert \tilde{f}_{\epsilon}(c_{0}^{\pm},\ldots,c_{M+1}^{\pm})\Vert _{L^{\infty}(\Omega_{T_{0}}^{\pm})}\leq C\quad \text{for all }\epsilon\in (0,1).
\]
\item We will need $(c_{k}^{\pm},\mu_{k}^{\pm},\mathbf{v}_{k}^{\pm},p_{k}^{\pm})$,
for $k\geq0$, to not only be defined in $\Omega_{T_{0}}^{\pm}$,
but we have to extend them onto $\Omega_{T_{0}}^{\pm}\cup\Gamma(2\delta;T_{0})$.
For $\mu_{k}^{\pm}$ and $p_{k}^{\pm}$ we may use any smooth extension. One possibility is to use the extension operator defined in \cite[Part VI, Theorem 5]{stein},
. It is trivial to extend $c_{0}^{\pm}$ and if
all $(c_{i}^{\pm},\mu_{i}^{\pm})$ for $i\leq k-1$ have
been defined on $\Omega_{T_{0}}^{\pm}\cup\Gamma(2\delta;T_{0})$,
then $c_{k}^{\pm}$ is as well, by (\ref{eq:Outercdefine}). For $\mathbf{v}_{k}^{\pm}$
we employ the same extension operator and then use the Bogovskii operator
to ensure that the extension is divergence free in $\Gamma_{t}(2\delta)$.
In particular we may construct a divergence free extension $\text{\ensuremath{\mathcal{E}}}^{\pm}(\mathbf{v}_{k}^{\pm})$
such that $\text{\ensuremath{\mathcal{E}}}^{\pm}(\mathbf{v}_{k}^{\pm})|_{\Omega^{\pm}(t)}=\mathbf{v}_{k}^{\pm}$
in $\Omega^{\pm}(t)$ and 
\begin{equation}
\left\Vert \text{\ensuremath{\mathcal{E}}}^{\pm}(\mathbf{v}_{k}^{\pm})\right\Vert _{H^{2}(\Omega^{\pm}(t)\cup\Gamma_{t}(2\delta))}\leq C\Vert \mathbf{v}_{k}^{\pm}\Vert _{H^{2}(\Omega^{\pm}(t))}.\label{eq:sp=0000E4terwichtig}
\end{equation}
\end{enumerate}
\end{rem}

For later use we define 
\begin{align}
U_{k}^{\pm}(x,t) & =\Delta\mu_{k}^{\pm}(x,t)-\partial_{t}c_{k}^{\pm}(x,t)-\sum_{j=0}^{k}\mathbf{v}_{j}^{\pm}(x,t)\cdot\nabla c_{k-j}^{\pm}(x,t), \quad U^{\pm}  =\sum_{k\geq0}\epsilon^{k}U_{k}^{\pm},\label{eq:Uk}\\
\mathbf{W}_{k}^{\pm}(x,t) & =-\Delta\mathbf{v}_{k}^{\pm}(x,t)+\nabla p_{k}^{\pm}(x,t)-\sum_{j=0}^{k-1}\mu_{j}^{\pm}(x,t)\nabla c_{k-j}^{\pm}(x,t),\quad \mathbf{W}^{\pm}  =\sum_{k\geq0}\epsilon^{k}\mathbf{W}_{k}^{\pm},\label{eq:Wk}
\end{align}
for $(x,t)\in\Omega_{T_{0}}^{\pm}\cup\Gamma(2\delta)$.
Note that by (\ref{eq:Outermudefine}) and (\ref{eq:Outervpdefine})
we have $\mathbf{W}_{k}^{\pm}(x,t)=U_{k}^{\pm}(x,t)=0$
for all $(x,t)\in\overline{\Omega_{T_{0}}^{\pm}}.$

\subsubsection{The Inner Expansion \label{subsec:The-Inner-Expansion}}

Close to the interface $\Gamma$ we introduce a stretched variable
\begin{equation}
\rho^{\epsilon}(x,t):=\frac{d_{\Gamma}(x,t)-\epsilon h^{\epsilon}(S(x,t),t)}{\epsilon}\qquad \text{for all }(x,t)\in\Gamma(2\delta)\label{eq:rhoeps}
\end{equation}
for $\epsilon\in\left(0,1\right)$. Here $h^{\epsilon}\colon \mathbb{T}^{1}\times [0,T_{0}]\rightarrow\mathbb{R}$
is a given smooth function and can heuristically be interpreted as
the distance of the zero level set of $c^{\epsilon}$ to $\Gamma$,
see also \cite[Chapter 4.2]{chenAC}. In the following, we will often
drop the $\epsilon$\textendash dependence and write $\rho(x,t)=\rho^{\epsilon}(x,t)$. 

Now assume that, in $\Gamma(2\delta)$, the identities
\begin{align}
c^{\epsilon}(x,t) & =\tilde{c}^{\epsilon}\big(\tfrac{d_{\Gamma}(x,t)}{\epsilon}-h^{\epsilon}(S(x,t),t),x,t\big), & \mu^{\epsilon}(x,t)&=\tilde{\mu}^{\epsilon}\big(\tfrac{d_{\Gamma}(x,t)}{\epsilon}-h^{\epsilon}(S(x,t),t),x,t\big),\nonumber \\
p^{\epsilon}(x,t) & =\tilde{p}^{\epsilon}\big(\tfrac{d_{\Gamma}(x,t)}{\epsilon}-h^{\epsilon}(S(x,t),t),x,t\big), & \mathbf{v}^{\epsilon}(x,t)&=\tilde{\mathbf{v}}^{\epsilon}\big(\tfrac{d_{\Gamma}(x,t)}{\epsilon}-h^{\epsilon}(S(x,t),t),x,t\big)\label{eq:Ansatz-Inner}
\end{align}
hold for the solutions of (\ref{eq:StokesPart})\textendash (\ref{eq:Dirichlet2})
and some smooth functions $\tilde{c}^{\epsilon},\tilde{\mu}^{\epsilon},\tilde{p}^{\epsilon}\colon \mathbb{R}\times\Gamma(2\delta)\rightarrow\mathbb{R}$,
$\tilde{\mathbf{v}}^{\epsilon}\colon \mathbb{R}\times\Gamma(2\delta)\rightarrow\mathbb{R}^{2}$.
Furthermore, we assume that we have the expansions 
\begin{align}
\tilde{c}^{\epsilon}(\rho,x,t) & \approx\sum_{k\geq0}\epsilon^{k}c_{k}(\rho,x,t), & \tilde{\mu}^{\epsilon}(\rho,x,t) & \approx\sum_{k\geq0}\epsilon^{k}\mu_{k}(\rho,x,t),\nonumber \\
\tilde{p}^{\epsilon}(\rho,x,t) & \approx\sum_{k\geq0}\epsilon^{k}p_{k}(\rho,x,t), & \tilde{\mathbf{v}}^{\epsilon}(\rho,x,t) & \approx\sum_{k\geq0}\epsilon^{k}\mathbf{v}_{k}(\rho,x,t)\label{eq:Ansatz}
\end{align}
for all $(\rho,x,t)\in\mathbb{R}\times\Gamma(2\delta)$
and also 
\begin{equation}
h^{\epsilon}(s,t)\approx\sum_{k\geq0}\epsilon^{k}h_{k+1}(s,t),\label{eq:Ansatz2}
\end{equation}
where $c_{k},\mu_{k},p_{k}\colon \mathbb{R}\times\Gamma(2\delta)\rightarrow\mathbb{R}$,
$\mathbf{v}_{k}\colon \mathbb{R}\times\Gamma(2\delta)\rightarrow\mathbb{R}^{2}$
and $h_{k}\colon \mathbb{T}^{1}\times [0,T_{0}]\rightarrow\mathbb{R}$
are smooth functions for all $k\geq0$. When referring to $\tilde{c},\tilde{\mu},\tilde{p},\tilde{\mathbf{v}}$
and the expansion terms we write $\nabla=\nabla_{x}$ and $\Delta=\Delta_{x}$.
The expressions $\partial_{t}^{\Gamma}h^{\epsilon}(x,t)$,
$\nabla^{\Gamma}h^{\epsilon}(x,t)$, $\Delta^{\Gamma}h^{\epsilon}(x,t)$
and $D_{\Gamma}^{2}h^{\epsilon}(x,t)$ are for $(x,t)\in\Gamma(2\delta)$
to be understood in the sense of Remark \ref{hnotations}.

In order to match the inner and outer expansions, we require that
for all $k$ the so-called \emph{inner-outer matching conditions}
\begin{align}
\sup_{(x,t)\in\Gamma(2\delta)}\left|\partial_{x}^{m}\partial_{t}^{n}\partial_{\rho}^{l}\left(\varphi\left(\pm\rho,x,t\right)-\varphi^{\pm}(x,t)\right)\right| & \leq Ce^{-\alpha\rho},\label{eq:matchcon}
\end{align}
where $\varphi=c_{k},\mu_{k},\mathbf{v}_{k},p_{k}$ and $k\geq0$
hold for constants $\alpha,C>0$ and all $\rho>0$, $m,n,l\geq0$. 
\begin{rem}
We will only use the matching conditions for $m,n,l\in \{ 0,1,2\} $.
However, since the ordinary differential equations for $(c_{k},c_{k}^{\pm},\mu_{k},\mu_{k}^{\pm},\mathbf{v}_{k},\mathbf{v}_{k}^{\pm},p_{k},p_{k}^{\pm})$
(cf.\ (\ref{eq:Stokesk}), (\ref{eq:Divk}), (\ref{eq:Cahnk}), (\ref{eq:Hilliardk}))
are dependent on derivatives of lower order terms, it is necessary
and sufficient for the matching conditions to hold for $m,n,l\in\{ 0,\ldots,C(M)\} $
for some $C(M)\in\mathbb{N}$ depending on the general
number of terms in the expansion.
\end{rem}

We interpret $\{ (x,t)\in\Gamma(2\delta)|d_{\Gamma}(x,t)=\epsilon h^{\epsilon}(S(x,t),t)\}$
as an approximation of the $0$-level set of $c^{\epsilon}$. Thus, we normalize $c^{k}$
such that 
\[
  c^{k}(0,x,t)=0\qquad \text{for all }(x,t)\in\Gamma(2\delta), k\geq 0.
\]
 Similarly as in \cite{abc}, we introduce
auxiliary functions $g^{\epsilon}(x,t)$, $j^{\epsilon}(x,t)$
and $l^{\epsilon}(x,t)$ as well as $\mathbf{u}^{\epsilon}(x,t)$
and $\mathbf{q}^{\epsilon}(x,t)$ for $(x,t)\in\Gamma(2\delta)$.
As a rough guideline, the functions $g^{\epsilon},j^{\epsilon}$,
and $\mathbf{q}^{\epsilon}$ will enable us to fulfill the compatibility
conditions in $\Gamma(2\delta)\backslash\Gamma$. $l^{\epsilon}$
and $\mathbf{u}^{\epsilon}$ on the other hand are of importance when
it comes to fulfilling the matching conditions in $\Gamma(2\delta)\backslash\Gamma$.
Moreover we choose $\eta\colon \mathbb{R}\rightarrow[0,1]$ such
that $\eta=0$ in $(-\infty,-1]$, $\eta=1$ in $[1,\infty)$
and $\eta'\geq0$ in $\mathbb{R}$ and such that
\begin{equation}
\int_{\mathbb{R}}\Big(\eta(\rho)-\frac{1}{2}\Big)\theta_{0}'(\rho)\d\rho=0\label{eq:etateta}
\end{equation}
is satisfied. For later use we also define
\[
\eta^{C,\pm}(\rho)=\eta(-C\pm\rho)
\]
for an arbitrary constant $C>0$ and $\rho\in\mathbb{R}$.

Now we may rewrite (\ref{eq:StokesPart})-(\ref{eq:CH-Part2}) as
\begin{align}
-\partial_{\rho\rho}\vte & =\epsilon\big(\partial_{\rho}\vte\Delta d_{\Gamma}+2(\nabla\partial_{\rho}\tilde{\mathbf{v}}^{\epsilon})^{T}\mathbf{n}+\tilde{\mu}^{\epsilon}\partial_{\rho}\tilde{c}^{\epsilon}\mathbf{n}\big)\nonumber \\
 & \quad+\epsilon^{2}\Big(-2(\nabla\partial_{\rho}\tilde{\mathbf{v}}^{\epsilon})^{T}\cdot\nabla^{\Gamma}h^{\epsilon}+\partial_{\rho\rho}\vte|\nabla^{\Gamma}h^{\epsilon}|^{2}-\partial_{\rho}\vte\Delta^{\Gamma}h^{\epsilon}-\tilde{\mu}^{\epsilon}\partial_{\rho}\tilde{c}^{\epsilon}\nabla^{\Gamma}h^{\epsilon}\nonumber \\
 & \quad\quad+\partial_{\rho}\pte\nabla^{\Gamma}h^{\epsilon}+\Delta\vte-\nabla\pte+\tilde{\mu}^{\epsilon}\nabla\tilde{c}^{\epsilon})-\mathbf{u}^{\epsilon}\eta''(\rho)\big(d_{\Gamma}-\epsilon(\rho+h^{\epsilon})\Big)\nonumber \\
 & \quad+\mathbf{q}^{\epsilon}\eta'(\rho)\big(d_{\Gamma}-\epsilon (\rho+h^{\epsilon})\big)+\epsilon^{2}(\mathbf{W}^{+}\eta^{C_{S},+}+\mathbf{W}^{-}\eta^{C_{S},-}),\label{eq:expx-stokes}\\
\partial_{\rho}\vte\cdot\mathbf{n} & =\epsilon\partial_{\rho}\vte\nabla^{\Gamma}h^{\epsilon}-\epsilon\di\mathbf{\vte}+\big(\mathbf{u^{\epsilon}}\cdot (\mathbf{n}-\epsilon\nabla^{\Gamma}h^{\epsilon})\big)\eta'(\rho)\big(d_{\Gamma}-\epsilon (\rho+h^{\epsilon})\big),\label{eq:expx-div}\\
  \partial_{\rho\rho}\cte-f'(c^{\epsilon}) & =\epsilon\big(-\mte-\partial_{\rho}\cte\Delta d_{\Gamma}-2\nabla\partial_{\rho}\tilde{c}^{\epsilon}\cdot\mathbf{n}\big)+g^{\epsilon}\eta'(\rho)\left(d_{\Gamma}-\epsilon\left(\rho+h^{\epsilon}\right)\right)\nonumber \\
                         & \quad+\epsilon^{2}\big(-\partial_{\rho\rho}\cte\left|\nabla^{\Gamma}h^{\epsilon}\right|^{2}+\partial_{\rho}\cte\Delta^{\Gamma}h^{\epsilon}+2\nabla\partial_{\rho}\tilde{c}^{\epsilon}\cdot\nabla^{\Gamma}h^{\epsilon}-\Delta\cte\big)\label{eq:expx-cahn}
\\
\partial_{\rho\rho}\mte & =\epsilon\big(\partial_{\rho}\cte(\partial_{t}d_{\Gamma}+\vte\cdot\mathbf{n})-\partial_{\rho}\mte\Delta d_{\Gamma}-2\nabla\partial_{\rho}\tilde{\mu}^{\epsilon}\cdot\mathbf{n}\big)\nonumber \\
 & \quad+\epsilon^{2}\Big(-\partial_{\rho}\cte(\partial_{t}^{\Gamma}h^{\epsilon}+\vte\cdot\nabla^{\Gamma}h^{\epsilon})+\partial_{\rho}\mte\Delta^{\Gamma}h^{\epsilon}\nonumber \\
 & \quad\quad-\partial_{\rho\rho}\mte|\nabla^{\Gamma}h^{\epsilon}|^{2}+2\nabla\partial_{\rho}\tilde{\mu}^{\epsilon}\cdot\nabla^{\Gamma}h^{\epsilon}+\vte\cdot\nabla\cte+\partial_{t}\cte-\Delta\mte\Big)\nonumber \\
 & \quad+(l^{\epsilon}\eta''(\rho)+j^{\epsilon}\eta'(\rho))\big(d_{\Gamma}-\epsilon (\rho+h^{\epsilon})\big)+\epsilon^{2}\big(U^{+}\eta^{C_{S},+}+U^{-}\eta^{C_{S},-}\big),\label{eq:expx-hilliard}
\end{align}
where the equalities are only assumed to hold in
\[
S^{\epsilon}:=\big\{ (\rho,x,t)\in\mathbb{R}\times\Gamma(2\delta)|\rho=\tfrac{d_{\Gamma}(x,t)}{\epsilon}-h^{\epsilon}(S(x,t),t)\big\} ,
\]
but we consider them as ordinary differential equations in $\rho\in\mathbb{R}$,
where $(x,t)\in \Gamma(2\delta)$ are seen as fixed parameters. Thus we assume from now on that  (\ref{eq:expx-stokes})\textendash (\ref{eq:expx-hilliard})
are fulfilled in $\mathbb{R}\times\Gamma(2\delta)$. The
terms $U^{\pm}$ and $\mathbf{W}^{\pm}$ (cf.\ (\ref{eq:Uk}), (\ref{eq:Wk}))
are used here in order to ensure the exponential decay of the right
hand sides; in this context $C_{S}>0$ is a constant which will be
determined later on (see Remark \ref{WU}). We assume that the auxiliary
functions have expansions of the form 
\begin{alignat}{3}
\mathbf{u}^{\epsilon}(x,t) & \approx\sum_{k\geq0}\mathbf{u}_{k}(x,t)\epsilon^{k},&\quad  l^{\epsilon}(x,t) & \approx\sum_{k\geq0}l_{k}(x,t)\epsilon^{k},\nonumber &\quad
\mathbf{q}^{\epsilon}(x,t) & \approx\sum_{k\geq0}\mathbf{q}_{k}(x,t)\epsilon^{k+1}, \\
j^{\epsilon}(x,t) & \approx\sum_{k\geq0}j_{k}(x,t)\epsilon^{k+1},&\quad
g^{\epsilon}(x,t) & \approx\sum_{k\geq0}g_{k}(x,t)\epsilon^{k+1},\label{eq:zusatzerme}
\end{alignat}
for $(x,t)\in\Gamma(2\delta)$. Matching the
$\epsilon$-orders, we gain the following ordinary differential equations
in $\rho$:
 From (\ref{eq:expx-stokes}) and (\ref{eq:expx-div}) we get 
\begin{align}
-\partial_{\rho\rho}(\mathbf{v}_{0}-\mathbf{u}_{0}\eta d_{\Gamma}) & =0,\label{eq:Stokes0}\\
-\partial_{\rho\rho}\big(\mathbf{v}_{k}-(\mathbf{u}_{k}d_{\Gamma}-\mathbf{u}_{0}h_{k})\eta\big) & +\partial_{\rho}p_{k-1}\mathbf{n}=\mathbf{V}^{k-1}\label{eq:Stokesk}
\end{align}
and 
\begin{align}
\partial_{\rho}\big(\mathbf{v}_{0}\cdot\mathbf{n}-\mathbf{u}_{0}\cdot\mathbf{n}d_{\Gamma}\eta\big) & =0,\label{eq:Div0}\\
\partial_{\rho}\big(\mathbf{v}_{k}\cdot\mathbf{n}-(\mathbf{u}_{k}d_{\Gamma}-\mathbf{u}_{0}h_{k})\cdot\mathbf{n}\eta\big) & =W^{k-1}+\nabla^{\Gamma}h_{k}\cdot (\partial_{\rho}\mathbf{v}_{0}-\mathbf{u}_{0}d_{\Gamma}\eta'),\label{eq:Divk}
\end{align}
respectively, for  $\rho\in\mathbb{R},$ $(x,t)\in\Gamma(2\delta)$ and $k\geq1$, where  $\mathbf{V}^{k-1}=\mathbf{V}^{k-1}(\rho,x,t)$ and $W^{k-1}=W^{k-1}(\rho,x,t)$ are defined
below. Similaly,  from (\ref{eq:expx-cahn}) and (\ref{eq:expx-hilliard})  we get
\begin{align}
\partial_{\rho\rho}c_{0}-f'(c_{0}) & =0,\label{eq:Cahn0}\\
\partial_{\rho\rho}c_{k}-f''(c_{0})c_{k} & =A^{k-1}\label{eq:Cahnk}
\end{align}
and
\begin{align}
\partial_{\rho\rho}\big(\mu_{0}-l_{0}\eta d_{\Gamma}\big) & =0,\label{eq:Hilliard0}\\
\partial_{\rho\rho}\big(\mu_{k}-\left(l_{k}d_{\Gamma}-l_{0}h_{k}\right)\eta\big) & =B^{k-1}\label{eq:Hilliardk}
\end{align}
respectively, for $\rho\in\mathbb{R}$, $(x,t)\in\Gamma(2\delta)$ and $k\geq1$, where $A^{k-1}=A^{k-1}(\rho,x,t)$ and $B^{k-1}=B^{k-1}(\rho,x,t)$ are defined
below. Here we used 
\begin{align}
\mathbf{V}^{k-1} & =\partial_{\rho}\mathbf{v}_{k-1}\Delta d_{\Gamma}+2(\nabla\partial_{\rho}\mathbf{v}_{k-1})^{T}\mathbf{n}-2(\nabla\partial_{\rho}\mathbf{v}_{0})^{T}\nabla^{\Gamma}h_{k-1}-\partial_{\rho}\mathbf{v}_{0}\Delta^{\Gamma}h_{k-1}\nonumber \\
 & \quad+\partial_{\rho}p_{0}\nabla^{\Gamma}h_{k-1}+\beta_{2}^{k}2\partial_{\rho\rho}\mathbf{v}_{0}\nabla^{\Gamma}h_{k-1}\cdot\nabla^{\Gamma}h_{1}+\beta_{1}^{k}(\mu_{0}\partial_{\rho}c_{k-1}+\mu_{k-1}\partial_{\rho}c_{0})\mathbf{n}\nonumber \\
 & \quad-\mu_{0}\partial_{\rho}c_{0}\nabla^{\Gamma}h_{k-1}+\mathbf{q}_{k-1}\eta'd_{\Gamma}-\mathbf{q}_{0}\eta'h_{k-1}+(\rho+\delta_{1}^{k}h_{1})\mathbf{u}_{k-1}\eta''+\mathbf{u}_{1}\eta''h_{k-1}\nonumber \\
 & \quad+\Delta\mathbf{v}_{k-2}-\nabla p_{k-2}+\sum_{i=0}^{k-2}\mu_{i}\nabla c_{k-2-i}+\mathbf{W}_{k-2}^{+}\eta^{C_{S},+}+\mathbf{W}_{k-2}^{-}\eta^{C_{S},-}+\mathcal{V}^{k-2},\label{eq:Vk-1}\\
W^{k-1} & =\delta_{1}^{k}\partial_{\rho}\mathbf{v}_{k-1}\nabla^{\Gamma}h_{1}+\partial_{\rho}\mathbf{v}_{1}\nabla^{\Gamma}h_{k-1}-\operatorname{div}\mathbf{v}_{k-1}-\mathbf{u}_{k-1}\cdot\mathbf{n}\eta'\rho-\delta_{1}^{k}\mathbf{u}_{k-1}h_{1}\cdot\mathbf{n}\eta'\nonumber \\
 & \quad-\mathbf{u}_{1}\cdot\mathbf{n}\eta'h_{k-1}-\delta_{1}^{k}(\mathbf{u}_{k-1}\cdot\nabla^{\Gamma}h_{1}+\mathbf{u}_{1}\cdot\nabla^{\Gamma}h_{k-1})d_{\Gamma}\eta'\nonumber \\
 & \quad+\mathbf{u}_{0}\cdot\big(\nabla^{\Gamma}h_{k-1}\rho+\beta_{2}^{k}(\nabla^{\Gamma}h_{k-1}h_{1}+\nabla^{\Gamma}h_{1}h_{k-1})\big)\eta'+\mathcal{W}^{k-2},\label{eq:Wk-1}\\
A^{k-1} & =-\mu_{k-1}-\partial_{\rho}c_{k-1}\Delta d_{\Gamma}-2\nabla\partial_{\rho}c_{k-1}\cdot\mathbf{n}+f_{k-1}(c_{0},\ldots,c_{k-1})+g_{k-1}\eta'd_{\Gamma}\nonumber \\
 & \quad-\beta_{2}^{k}2\partial_{\rho\rho}c_{0}\nabla^{\Gamma}h_{k-1}\cdot\nabla^{\Gamma}h_{1}+\partial_{\rho}c_{0}\Delta^{\Gamma}h_{k-1}+2\nabla\partial_{\rho}c_{0}\cdot\nabla^{\Gamma}h_{k-1}-g_{0}h_{k-1}\eta'\nonumber \\
 & \quad-\Delta c_{k-2}+\mathcal{A}^{k-2},\label{eq:Ak-1}
\end{align}
and 
\begin{align}
B^{k-1} & =\partial_{\rho}c_{k-1}\partial_{t}d_{\Gamma}+\beta_{1}^{k}(\partial_{\rho}c_{k-1}\mathbf{v}_{0}+\partial_{\rho}c_{0}\mathbf{v}_{k-1})\cdot\mathbf{n}-\partial_{\rho}\mu_{k-1}\Delta d_{\Gamma}-2\nabla\partial_{\rho}\mu_{k-1}\cdot\mathbf{n}\nonumber \\
 & \quad-l_{k-1}\eta''\rho-\delta_{1}^{k}l_{k-1}h_{1}\eta''+j_{k-1}\eta'd_{\Gamma}-\partial_{\rho}c_{0}\mathbf{v}_{0}\cdot\nabla^{\Gamma}h_{k-1}-\partial_{\rho}c_{0}\partial_{t}^{\Gamma}h_{k-1}\nonumber \\
 & \quad-\beta_{2}^{k}2\partial_{\rho\rho}\mu_{0}\nabla^{\Gamma}h_{k-1}\cdot\nabla^{\Gamma}h_{1}+\partial_{\rho}\mu_{0}\Delta^{\Gamma}h_{k-1}+2\nabla\partial_{\rho}\mu_{0}\cdot\nabla^{\Gamma}h_{k-1}-l_{1}h_{k-1}\eta''-j_{0}h_{k-1}\eta'\nonumber \\
 & \quad+\partial_{t}c_{k-2}-\Delta\mu_{k-2}+\sum_{i=0}^{k-2}\mathbf{v}_{i}\nabla c_{k-2-i}+U_{k-2}^{+}\eta^{C_{S},+}+U_{k-2}^{-}\eta^{C_{S},-}+\mathcal{B}^{k-2}.\label{eq:Bk-1}
\end{align}
Here $\mathcal{V}^{k-2}$, $\mathcal{W}^{k-2}$, $\mathcal{A}^{k-2}$,
and $\mathcal{B}^{k-2}$ denote terms of order $k-2$ or lower which
are unimportant in the following - the detailed structure of
these terms can be found in \cite[Subsection~5.1.2]{ichPhD}. In all
of the above identities we used the following conventions:
\begin{notation}
\label{nota:innerM-1}$\quad$
\begin{enumerate}
\item All functions with negative index are supposed to be zero. In particular
$\mathcal{V}^{-1}=\mathcal{W}^{-1}=\mathcal{A}^{-1}=\mathcal{B}^{-1}=0$.
Moreover, $h_{0}:=0$.
\item We introduced the notation
\[
\beta_{i}^{k}=\begin{cases}
\frac{1}{2} & \text{if }i=k,\\
1 & \text{else}
\end{cases}
\]
and $\delta_{i}^{k}$ is an ``inverse'' Kronecker delta, i.e.\ 
\[
\delta_{i}^{k}=\begin{cases}
0 & \text{if }i=k,\\
1 & \text{else}.
\end{cases}
\]
\item $f_{k-1}(c_{0},\ldots,c_{k-1})$ (appearing in
(\ref{eq:Ak-1})) are terms from a Taylor expansion defined in the same way as in Remark \ref{Outer-Rem}. 
In particular, we will later on also use a remainder term $\tilde{f}$
as discussed in Remark \ref{Outer-Rem} for the inner solutions.
Moreover, we use the convention $f_{0}(c_{0})=0$.
\end{enumerate}
\end{notation}

We will see after the construction of the zeroth order terms that
the term $h_{k}$ appearing on the right hand side of (\ref{eq:Divk})
is actually multiplied by $0$.
\begin{rem}
\label{WU}Note that $\mathbf{W}^{\pm}$ and $U^{\pm}$, which we
inserted in (\ref{eq:expx-stokes}) and (\ref{eq:expx-hilliard}),
are not multiplied by terms of the kind $\left(d_{\Gamma}-\epsilon\left(\rho+h^{\epsilon}\right)\right)$.
So we have to make sure they vanish on the set $S^{\epsilon}$. This
is accomplished by choosing the constant $C_{S}>0$ in a suitable
way.

In particular we set
\[
C_{S}:=\left\Vert h_{1}\right\Vert _{C^{0}\left(\mathbb{T}^{1}\times[0,T_0]\right)}+2
\]
and assume that
\begin{equation}
\Big|\sum_{k\geq1}\epsilon^{k}h_{k+1}(S(x,t),t)\Big|\leq1\label{eq:remhglm}
\end{equation}
holds for all $\epsilon>0$ small enough. It turns out that $h_{1}$
does not depend on the term $\epsilon^{2}(U^{+}\eta^{C_{S},+}+U^{-}\eta^{C_{S},-})$
and $\epsilon^{2}(\mathbf{W}^{+}\eta^{C_{S},+}+\mathbf{W}^{-}\eta^{C_{S},-})$,
so this choice of $C_{S}$ does not cause problems. Choosing $C_{S}$
in this way, it is possible to show (see \cite[Remark 4.2 (2)]{abc})
that for $\rho=\frac{d_{\Gamma}(x,t)}{\epsilon}-h^{\epsilon}\left(S(x,t),t\right)$
and $(x,t)\in\Gamma(2\delta)$ such that $d_{\Gamma}(x,t)\geq0$
it follows $\rho\geq-C_{S}+1$. Thus, $\eta^{C_{S},-}(\rho)=0$
and since $(x,t)\in\overline{\Omega^{+}}$ we have $\mathbf{W}^{+}(x,t)=U^{+}(x,t)=0$
and so 
\[
\epsilon^{2}\left(U^{+}\eta^{C_{S},+}+U^{-}\eta^{C_{S},-}\right)=\epsilon^{2}\left(\mathbf{W}^{+}\eta^{C_{S},+}+\mathbf{W}^{-}\eta^{C_{S},-}\right)=0.
\]
A similar statement holds when $d_{\Gamma}(x,t)<0$.
\end{rem}

\subsubsection{The Boundary Layer Expansion\label{subsec:The-Boundary-Layer}}

To be able to guarantee that the approximate solutions satisfy boundary
conditions akin to (\ref{eq:StokesBdry})\textendash (\ref{eq:Dirichlet2}),
we also need to consider a separate expansion close to the boundary
of $\Omega$. In the following we write $\mathbf{n}_{\partial\Omega}(x):=\mathbf{n}_{\partial\Omega}\left(\operatorname{Pr}_{\partial\Omega}(x)\right)$
and $\boldsymbol{\tau}_{\partial\Omega}(x):=\boldsymbol{\tau}_{\partial\Omega}\left(\operatorname{Pr}_{\partial\Omega}(x)\right)$
for $x\in\partial\Omega\left(\delta\right)$.

We assume that for $(x,t)\in\overline{\partial_{T}\Omega\left(\delta\right)}$
the identities 
\begin{alignat}{2}
c^{\epsilon}(x,t) & =c_{\mathbf{B}}^{\epsilon}\big(\tfrac{d_{\mathbf{B}}(x)}{\epsilon},x,t\big), &\qquad  \mu^{\epsilon}(x,t) & =\mu_{\mathbf{B}}^{\epsilon}\big(\tfrac{d_{\mathbf{B}}(x)}{\epsilon},x,t\big),\nonumber \\
p^{\epsilon}(x,t) & =p_{\mathbf{B}}^{\epsilon}\big(\tfrac{d_{\mathbf{B}}(x)}{\epsilon},x,t\big), & \mathbf{v}^{\epsilon}(x,t) & =\mathbf{v}_{\mathbf{B}}^{\epsilon}\big(\tfrac{d_{\mathbf{B}}(x)}{\epsilon},x,t\big)\label{eq:Ansatz-outer}
\end{alignat}
hold for the solutions of (\ref{eq:StokesPart})\textendash (\ref{eq:Dirichlet2})
and smooth functions $c_{\mathbf{B}}^{\epsilon},\mu_{\mathbf{B}}^{\epsilon},p_{\mathbf{B}}^{\epsilon}:\mathbb{R}\times\overline{\partial_{T_{0}}\Omega\left(\delta\right)}\rightarrow\mathbb{R}$,
$\mathbf{v}_{\mathbf{B}}^{\epsilon}:\mathbb{R}\times\overline{\partial_{T_{0}}\Omega\left(\delta\right)}\rightarrow\mathbb{R}^{2}$.
Furthermore, we assume that the expansions 
\begin{alignat}{2}
c_{\mathbf{B}}^{\epsilon}\left(z,x,t\right) & \approx-1+\sum_{k\geq1}\epsilon^{k}c_{k}^{\mathbf{B}}\left(z,x,t\right), &\qquad \mu_{\mathbf{B}}^{\epsilon}\left(z,x,t\right) & \approx\sum_{k\geq0}\epsilon^{k}\mu_{k}^{\mathbf{B}}\left(z,x,t\right),\nonumber \\
p_{\mathbf{B}}^{\epsilon}\left(z,x,t\right) & \approx\sum_{k\geq0}\epsilon^{k}p_{k}^{\mathbf{B}}\left(z,x,t\right), & \mathbf{v}_{\mathbf{B}}^{\epsilon}\left(z,x,t\right) & \approx\sum_{k\geq0}\epsilon^{k}\mathbf{v}_{k}^{\mathbf{B}}\left(z,x,t\right)\label{eq:Ansatz-outer-expansion}
\end{alignat}
are given for all $\left(z,x,t\right)\in\left(-\infty,0\right]\times\overline{\partial_{T_{0}}\Omega\left(\delta\right)}$.
As in the case of the inner expansion, we also assume that the \emph{outer-boundary
matching conditions}
\begin{align}
\sup_{(x,t)\in\overline{\partial_{T_{0}}\Omega\left(\delta\right)}}\big|\partial_{x}^{m}\partial_{t}^{n}\partial_{z}^{l}(\varphi_{k}^{\mathbf{B}}(z,x,t)-\varphi_{k}^{-}(x,t))\big| & \leq Ce^{\alpha z},\label{eq:matchcon-bdry}
\end{align}
hold for $\varphi=c,\mu,\mathbf{v},p$ and some constants $\alpha,C>0$
and all $z\leq0$, $m,n,l\geq0$. Plugging the assumed form of the
exact solutions (\ref{eq:Ansatz-outer}) into the equations (\ref{eq:StokesPart})\textendash (\ref{eq:CH-Part2})
we obtain for $(x,t)\in\overline{\partial_{T_{0}}\Omega\left(\delta\right)}$
and $z=\frac{d_{\mathbf{B}}(x)}{\epsilon}$ the identities
\begin{align*}
-\partial_{zz}\mathbf{v}_{\mathbf{B}}^{\epsilon} & +\partial_{z}p_{\mathbf{B}}^{\epsilon}\nabla d_{\mathbf{B}}=\epsilon\left(2\partial_{z}D\mathbf{v}_{\mathbf{B}}^{\epsilon}\nabla d_{\mathbf{B}}+\partial_{z}\mathbf{v}_{\mathbf{B}}^{\epsilon}\Delta d_{\mathbf{B}}+\mu^{\epsilon}\partial_{z}c_{\mathbf{B}}^{\epsilon}\nabla d_{\mathbf{B}}\right)\\
 & \quad+\epsilon^{2}\left(\Delta\mathbf{v}_{\mathbf{B}}^{\epsilon}-\nabla p_{\mathbf{B}}^{\epsilon}+\mu_{\mathbf{B}}^{\epsilon}\nabla c_{\mathbf{B}}^{\epsilon}\right),\\
\partial_{z}\mathbf{v_{B}^{\epsilon}}\cdot\nabla d_{\mathbf{B}} & =-\epsilon\operatorname{div}\mathbf{v_{\mathbf{B}}^{\epsilon}},
\end{align*}
where the differential operator $\nabla=\nabla_x$, $\operatorname{div}=\operatorname{div}_x$, $\Delta=\Delta_x$ act only on the variable $x$ and not on $z$. In the calculations we used $|\nabla d_{\mathbf{B}}|^{2}=1$
for $(x,t)\in\overline{\partial_{T_{0}}\Omega\left(\delta\right)}$. 

Moreover, we have 
\begin{align*}
\partial_{zz}c_{\mathbf{B}}^{\epsilon}-f'\left(c_{\mathbf{B}}^{\epsilon}\right) & =-\epsilon\left(\mu_{\mathbf{B}}^{\epsilon}+2\partial_{z}\nabla c_{\mathbf{B}}^{\epsilon}\cdot\nabla d_{\mathbf{B}}+\partial_{z}c_{\mathbf{B}}^{\epsilon}\Delta d_{\mathbf{B}}\right)-\epsilon^{2}\Delta c_{\mathbf{B}}^{\epsilon},\\
\partial_{zz}\mu_{\mathbf{B}}^{\epsilon} & =\epsilon\left(-2\partial_{z}\nabla\mu_{\mathbf{B}}^{\epsilon}\cdot\nabla d_{\mathbf{B}}-\partial_{z}\mu_{\mathbf{B}}^{\epsilon}\Delta d_{\mathbf{B}}+\mathbf{v}_{\mathbf{B}}^{\epsilon}\cdot\nabla d_{\mathbf{B}}\partial_{z}c_{\mathbf{B}}^{\epsilon}\right)\\
 & \quad+\epsilon^{2}\left(\partial_{t}c_{\mathbf{B}}^{\epsilon}+\mathbf{v}^{\epsilon}\cdot\nabla c_{\mathbf{B}}^{\epsilon}-\Delta\mu_{\mathbf{B}}^{\epsilon}\right).
\end{align*}
Using  (\ref{eq:Ansatz-outer-expansion})
and equating same orders of $\epsilon$,  we get 
\begin{alignat}{2}
-\partial_{zz}\mathbf{v}_{k}^{\mathbf{B}}+\partial_{z}p_{k-1}^{\mathbf{B}}\nabla d_{\mathbf{B}} & =\mathbf{V}_{\mathbf{B}}^{k-1} &\qquad & \text{for } k\geq0,\label{eq:stokes-bdry}\\
\partial_{z}\mathbf{v}_{k}^{\mathbf{B}}\cdot\nabla d_{\mathbf{B}} & =-\operatorname{div}\mathbf{v}_{k-1}^{\mathbf{B}} & & \text{for } k\geq0,\label{eq:div-bdry}\\
\partial_{zz}c_{k}^{\mathbf{B}}-f''\left(-1\right)c_{k}^{\mathbf{B}} & =A_{\mathbf{B}}^{k-1} & & \text{for } k\geq1,\label{eq:cahn-bdry}\\
\partial_{zz}\mu_{k}^{\mathbf{B}} & =B_{\mathbf{B}}^{k-1} && \text{for }  k\geq0 \label{eq:hilliard-bdry}
\end{alignat}
for all $\left(z,x,t\right)\in(-\infty,0]\times\overline{\partial_{T_{0}}\Omega(\delta)}$,
where $\mathbf{V}_{\mathbf{B}}^{k-1}=\mathbf{V}_{\mathbf{B}}^{k-1}(z,x,t)$,
$A_{\mathbf{B}}^{k-1}=A_{\mathbf{B}}^{k-1}(z,x,t)$ and
$B_{\mathbf{B}}^{k-1}=B_{\mathbf{B}}^{k-1}(z,x,t)$. In
detail, we have
\begin{align}
\mathbf{V}_{\mathbf{B}}^{k-1} & :=2\partial_{z}D\mathbf{v}_{k-1}^{\mathbf{B}}\nabla d_{\mathbf{B}}+\partial_{z}\mathbf{v}_{k-1}^{\mathbf{B}}\Delta d_{\mathbf{B}}+\mu_{0}^{\mathbf{B}}\partial_{z}c_{k-1}^{\mathbf{B}}\nabla d_{\mathbf{B}}\nonumber \\
 & \quad+\Delta\mathbf{v}_{k-2}^{\mathbf{B}}-\nabla p_{k-2}^{\mathbf{B}}+\sum_{i=0}^{k-2}\mu_{i}^{\mathbf{B}}\nabla c_{k-2-i}^{\mathbf{B}},\label{eq:bdryV}\\
A_{\mathbf{B}}^{k-1}&:=-\mu_{k-1}^{\mathbf{B}}-2\partial_{z}\nabla c_{k-1}^{\mathbf{B}}\cdot\nabla d_{\mathbf{B}}-\partial_{z}c_{k-1}^{\mathbf{B}}\Delta d_{\mathbf{B}}-\Delta c_{k-2}^{\mathbf{B}}+f_{k-1}(c_{0}^{\mathbf{B}},\ldots,c_{k-1}^{\mathbf{B}}),\label{eq:bdryA}\\
B_{\mathbf{B}}^{k-1} & :=-2\partial_{z}\nabla\mu_{k-1}^{\mathbf{B}}\cdot\nabla d_{\mathbf{B}}-\partial_{z}\mu_{k-1}^{\mathbf{B}}\Delta d_{\mathbf{B}}+\sum_{i+j=k-1}\mathbf{v}_{i}^{\mathbf{B}}\cdot\nabla d_{\mathbf{B}}\partial_{z}c_{j}^{\mathbf{B}}+\partial_{t}c_{k-2}^{\mathbf{B}}\nonumber \\
 & \quad+\sum_{i+j=k-2}\mathbf{v}_{i}^{\mathbf{B}}\cdot\nabla c_{j}^{\mathbf{B}}-\Delta\mu_{k-2}^{\mathbf{B}}.\label{eq:bdryB}
\end{align}
We used the convention that all terms with negative index are supposed
to be zero, i.e., $\mu_{-2}=\mu_{-1}=0$. 

To ensure the Dirichlet boundary condition we suppose that 
\begin{align}
c_{k}^{\mathbf{B}}(0,x,t) & =\frac{\mu_{k-1}^{\mathbf{B}}(0,x,t)}{f''(-1)} &  & \text{for all }(x,t)\in\overline{\partial_{T_{0}}\Omega(\delta)},k\geq1,\label{eq:cDir}\\
\mu_{k}^{\mathbf{B}}(0,x,t) & =0 &  & \text{for all }(x,t)\in\partial_{T_{0}}\Omega,k\geq0.\label{eq:muDir}
\end{align}
Regarding the boundary condition of the Stokes system we calculate
\begin{align*}
2D_{s}\Big(\mathbf{v}_{k}^{\mathbf{B}}\big(\tfrac{d_{\mathbf{B}}(x)}{\epsilon},x,t\big)\Big)\mathbf{n}_{\partial\Omega}(x) & =\frac{1}{\epsilon}\left(\mathbf{I}+\mathbf{n}_{\partial\Omega}(x)\otimes\mathbf{n}_{\partial\Omega}(x)\right)\partial_{z}\mathbf{v}_{k}^{\mathbf{B}}\big(\tfrac{d_{\mathbf{B}}(x)}{\epsilon},x,t\big)\\
 & \quad+2D_{s}\mathbf{v}_{k}^{\mathbf{B}}\big(\tfrac{d_{\mathbf{B}}(x)}{\epsilon},x,t\big)\mathbf{n}_{\partial\Omega}(x)
\end{align*}
and thus impose 
\begin{align}
-\left(\mathbf{I}+\mathbf{n}_{\partial\Omega}(x)\otimes\mathbf{n}_{\partial\Omega}(x)\right)\partial_{z}\mathbf{v}_{k}^{\mathbf{B}}(0,x,t) & =2D_{s}\mathbf{v}_{k-1}^{\mathbf{B}}(0,x,t)\mathbf{n}_{\partial\Omega}(x)\nonumber \\
 & \quad-p_{k-1}^{\mathbf{B}}(0,x,t)\mathbf{n}_{\partial\Omega}(x)+\alpha_{0}\mathbf{v}_{k-1}^{\mathbf{B}}(0,x,t)\label{eq:vnavier}
\end{align}
for all $(x,t)\in\partial_{T_{0}}\Omega,k\geq0$.
\begin{rem}
\label{cor:bdrycond}It can be shown that by choosing (\ref{eq:cDir}),
the unique solution $c_{1}^{\mathbf{B}}$ to (\ref{eq:cDir}) satisfies
$c_{1}^{\mathbf{B}}(z,x,t)=c_{1}^{-}(x,t)$
for all $(z,x,t)\in\left(-\infty,0\right]\times\overline{\partial_{T_{0}}\Omega(\delta)}.$
\end{rem}

\subsubsection{Existence of Expansion Terms}

For the proofs of the statements in this subsection we refer to \cite[Subsection~5.1.6]{ichPhD}.
\begin{lem}[The zeroth order terms]
 \label{zeroorder}~\\Let $(\mathbf{v}^{\pm},p^{\pm},\mu^{\pm})$
be extended to $\Omega_{T_{0}}^{\pm}\cup\Gamma(2\delta;T_{0})$ as in Remark~\ref{Outer-Rem}.2.
We define the terms of the outer expansion $(c_{0}^{\pm},\mu_{0}^{\pm},\mathbf{v}_{0}^{\pm},p_{0}^{\pm})$
for $(x,t)\in\Omega_{T_{0}}^{\pm}\cup\Gamma(2\delta;T_{0})$
as 
\begin{equation}
c_{0}^{\pm}(x,t)=\pm1,\;\mu_{0}^{\pm}(x,t)=\mu^{\pm}(x,t),\;\mathbf{v}_{0}^{\pm}(x,t)=\mathbf{v}^{\pm}(x,t),\;p_{0}^{\pm}(x,t)=p^{\pm}(x,t),\label{eq:0outdef}
\end{equation}
the terms of the inner expansion $(c_{0},\mu_{0},\mathbf{v}_{0})$
as
\begin{align}
c_{0}(\rho,x,t) & =\theta_{0}(\rho),\label{eq:c0def}\\
\mu_{0}(\rho,x,t) & =\mu_{0}^{+}(x,t)\eta(\rho)-\mu_{0}^{-}(x,t)\left(\eta(\rho)-1\right),\label{eq:mu0def}\\
\mathbf{v}_{0}(\rho,x,t) & =\mathbf{v}_{0}^{+}(x,t)\eta(\rho)-\mathbf{v}_{0}^{-}(x,t)\left(\eta(\rho)-1\right),\label{eq:v0def}
\end{align}
for all $(\rho,x,t)\in\mathbb{R}\times\Gamma(2\delta;T_{0})$ and the terms of the boundary expansion $\left(c_{0}^{\mathbf{B}},\mu_{0}^{\mathbf{B}},\mathbf{v}_{0}^{\mathbf{B}},p_{0}^{\mathbf{B}}\right)$
as 
\[
c_{0}^{\mathbf{B}}(z,x,t)=-1,\;\mu_{0}^{\mathbf{B}}(z,x,t)=\mu_{0}^{-}(x,t),\;\mathbf{v}_{0}^{\mathbf{B}}(z,x,t)=\mathbf{v}_{0}^{-}(x,t),\;p_{0}^{\mathbf{B}}(z,x,t)=p_{0}^{-}(x,t)
\]
for all $(z,x,t)\in\left(-\infty,0\right]\times\overline{\partial_{T_{0}}\Omega(\delta)}$.
Then there are smooth and bounded $l_{0}, j_{0}, g_{0}\colon \Gamma(2\delta)\rightarrow\mathbb{R}$,
and $\mathbf{u}_{0}$, $\mathbf{q}_{0}\colon \Gamma(2\delta)\rightarrow\mathbb{R}^{2}$
such that the outer equations \eqref{eq:c0out}, \eqref{eq:Outermudefine},
\eqref{eq:Outervpdefine} (for $k=0$), the inner equations \eqref{eq:Stokes0},
\eqref{eq:Div0}, \eqref{eq:Cahn0}, \eqref{eq:Hilliard0}, the boundary
equations \eqref{eq:stokes-bdry}\textendash \eqref{eq:hilliard-bdry}
(for $k=0$), the inner-outer matching conditions \eqref{eq:matchcon}
the outer-boundary matching conditions \eqref{eq:matchcon-bdry} and
the boundary conditions \eqref{eq:muDir} and \eqref{eq:vnavier}
(for $k=0$) are satisfied.
\end{lem}

\begin{rem}
\label{rem:zeroprop}$\quad$
 As a consequence of (\ref{eq:mu0def}), (\ref{eq:0outdef}), the equation
for $\mu_{0}^{\pm}$ on $\Gamma_{t}$ (\ref{eq:S-SAC5}) and $\Delta d_{\Gamma}(x,t)=-H_{\Gamma_{t}}(x)$
for $(x,t)\in\Gamma$, we have 
\begin{equation}
\mu_{0}(\rho,x,t)=-\sigma\Delta d_{\Gamma}(x,t)\label{eq:mu0b}
\end{equation}
for $(\rho,x,t)\in\mathbb{R}\times\Gamma$. Moreover, it holds 
\begin{equation}
\mathbf{u}_{0}=0\;\text{on }\Gamma\label{eq:u00}
\end{equation}
and $\partial_{\rho}\mathbf{v}_{0}=\mathbf{u}_{0}d_{\Gamma}\eta'$
in $\mathbb{R}\times\Gamma (2\delta;T_{0})$.
\end{rem}

\begin{lem}[The $k$-th order terms]
 \label{k-thorder}~\\
Let $k\in\{ 1,\ldots,M+1\} $ be given. Then there are
smooth functions
\[
\mathbf{v}_{k},\mathbf{v}_{k}^{\pm},\mathbf{v}_{k}^{\mathbf{B}},\mathbf{u}_{k},\mathbf{q}_{k},\mu_{k},\mu_{k}^{\pm},\mu_{k}^{\mathbf{B}},c_{k},c_{k}^{\pm},c_{k}^{\mathbf{B}},h_{k},l_{k},j_{k},g_{k},p_{k-1},p_{k}^{\pm},p_{k-1}^{\mathbf{B}}
\]
which are bounded on their respective domains, such that for $k$-th
order the outer equations \eqref{eq:Outercdefine}, \eqref{eq:Outermudefine}
and \eqref{eq:Outervpdefine}, the inner equations \eqref{eq:Stokesk},
\eqref{eq:Divk}, \eqref{eq:Cahnk} and \eqref{eq:Hilliardk}, the
boundary equations \eqref{eq:stokes-bdry}\textendash \eqref{eq:hilliard-bdry},
the inner-outer matching conditions \eqref{eq:matchcon}, the outer-boundary
matching conditions \eqref{eq:matchcon-bdry} and the boundary conditions
\eqref{eq:cDir}\textendash \eqref{eq:vnavier} are satisfied. Additionally,
it holds $h_{k}(s,0)=0$ for all $s\in\mathbb{T}^{1}$.
Here $\mathbf{v}_{k}^{\pm}$, $\mu_{k}^{\pm}$, $c_{k}^{\pm}$ and
$p_{k}^{\pm}$ are considered to be extended onto $\Omega_{T_{0}}^{\pm}\cup\Gamma(2\delta;T_{0})$ as in Remark~\ref{Outer-Rem}.2.
\end{lem}

\begin{rem}
\label{rem:noslipnosol} Let us remark upon the difficulties that
would arise if we considered e.g.\ no-slip boundary conditions for
$\mathbf{v}^{\epsilon}$. In that case, we would demand for $\mathbf{v}_{A}^{\epsilon}$
to also satisfy $\mathbf{v}_{A}^{\epsilon}=0$ on $\partial_{T_{0}}\Omega$,
which may be achieved by suitable changes to the presented boundary
layer expansion. As a consequence, the outer solution would need to
satisfy (among other equations) 
\begin{align*}
\operatorname{div}\mathbf{v}_{k}^{\pm} & =0 &  & \text{in }\Omega_{T_{0}}^{\pm},\\
\left[\mathbf{v}_{k}\right] & =\mathbf{a}_{1} &  & \text{on }\Gamma,\\
\mathbf{v}_{k}^{-} & =\mathbf{a}_{2} &  & \text{on }\partial_{T_{0}}\Omega,
\end{align*}
where $\mathbf{a}_{1}$, $\mathbf{a}_{2}$ are smooth functions,
depending only on lower order terms. As a consequence, the divergence
theorem implies 
\begin{align*}
0 & =\int_{\Omega^{+}(t)}\operatorname{div}\mathbf{v}_{k}^{+}\d x+\int_{\Omega^{-}(t)}\operatorname{div}\mathbf{v}_{k}^{-}\d x=-\int_{\Gamma_{t}}\mathbf{a}_{1}\cdot\mathbf{n}_{\Gamma_{t}}\d\mathcal{\mathcal{H}}^{1}(p)+\int_{\partial\Omega}\mathbf{a}_{2}\cdot\mathbf{n}_{\partial\Omega}\d\mathcal{H}^{1}(p)
\end{align*}
for $t\in[0,T_0]$. However, this equality does not have
to be satisfied for arbitrary $k$. To avoid this difficulty,
we restricted ourselves to the case of the boundary condition (\ref{eq:StokesBdry}).
\end{rem}

Now we ``glue'' together the inner and outer expansions of $c^{\epsilon}$
in order to get an approximate solution. We will repeat this later
for approximate solutions of $\mu^{\epsilon},\mathbf{v}^{\epsilon},p^{\epsilon}$,
cf.\ Definition \ref{def:apprxsol}.
\begin{defn}[A First Approximate Solution]
 \label{firstapprx}~\\ Let $\mathfrak{S}_{0},\ldots,\mathfrak{S}_{M+1}$
be the expansions up to order $M+1$ as given in Lemmata \ref{zeroorder}
and \ref{k-thorder}. Let furthermore some $\epsilon_{0}>0$, $T'\in (0,T_{0}]$
and $(\tilde{h}^{\epsilon})_{\epsilon\in(0,\epsilon_{0})}\subset X_{T'}$
 with $\tilde{h}^{\epsilon}|_{t=0}=0$ be given (cf.\ (\ref{eq:XT})
for the definition of $X_{T'}$). In the following, we write $H:=(\tilde{h}^{\epsilon})_{\epsilon\in (0,\epsilon_{0})}$. 

We define
\begin{equation}
h_{A}^{\epsilon,H}(s,t):=\sum_{i=0}^{M}\epsilon^{i}h_{i+1}(s,t)+\epsilon^{M-\frac{3}{2}}\tilde{h}^{\epsilon}(s,t)\label{eq:ha}
\end{equation}
for $(s,t)\in\mathbb{T}^{1}\times [0,T']$. Note
that $h^{\epsilon}(s,t)$ is well-defined for all $(s,t)\in\mathbb{T}^{1}\times [0,T']$
since $X_{T'}\hookrightarrow C^{0}([0,T'];C^{1}(\mathbb{T}^1))$
due to Proposition~\ref{embedding}.2 and Sobolev embeddings.
Furthermore, we set 
\begin{align}
\tilde{c}_{I}(\rho,x,t) & :=\sum_{i=0}^{M+1}\epsilon^{i}c_{i}(\rho,x,t),\label{eq:ctilde}\qquad
c_{I}^{H}(x,t)  :=\sum_{i=0}^{M+1}\epsilon^{i}c_{i}(\rho^{H}(x,t),x,t)
\end{align}
for $\rho\in\mathbb{R}$, $(x,t)\in\Gamma(2\delta;T')$
and 
\begin{equation}
\rho^{H}(x,t):=\frac{d_{\Gamma}(x,t)}{\epsilon}-h_{A}^{\epsilon,H}(S(x,t),t).\label{eq:roha}
\end{equation}
For the outer part we set
\[
c_{O}(x,t):=\sum_{i=0}^{M+1}\epsilon^{i}\left(c_{i}^{+}(x,t)\chi_{\overline{\Omega^{+}}}(x,t)+c_{i}^{-}(x,t)\chi_{\Omega^{-}}(x,t)\right)
\]
for $(x,t)\in\Omega_{T'}$ and for the boundary part we
define 
\[
c_{\mathbf{B}}(x,t):=\sum_{i=0}^{M+1}\epsilon^{i}c_{i}^{\mathbf{B}}\big(\tfrac{d_{\mathbf{B}}(x,t)}{\epsilon},x,t\big)
\]
for $(x,t)\in\overline{\partial_{T'}\Omega(\delta)}$.

Let $\xi\in C^{\infty}(\mathbb{R})$ satisfy (\ref{eq:cut-off}). We now define the approximate solution
\begin{equation}
c_{A}^{\epsilon,H}:=\xi (d_{\Gamma})c_{I}^{H}+(1-\xi(d_{\Gamma}))(1-\xi(2d_{\mathbf{B}}))c_{O}+\xi(2d_{\mathbf{B}})c_{\mathbf{B}}\quad \text{in }\Omega_{T'}.\label{eq:caquer} 
\end{equation}
\end{defn}

Later on, the family $H$ will be replaced by the terms of correct
order $h_{M-\frac{1}{2}}^{\epsilon}$, which will then depend on $\epsilon$.
But in order to find those terms we need some preparations first,
which will turn out to be more flexible and notationally consistent
when they are done with an arbitrary family of functions $H$.

\subsection{A First Estimate of the Error in the Velocity\label{sec:Estimating-the-error-vel}}

Let the assumptions and notations of Definition \ref{firstapprx}
hold throughout this subsection. Moreover, we denote
\begin{equation*}
V_{0}:=\overline{\{\boldsymbol{\varphi}\in C^\infty(\overline{\Omega})^2: \operatorname{div}\boldsymbol{\varphi}=0\}}^{H^{1}(\Omega)}
\end{equation*}
and $\mathbf{a}\otimes_{s}\mathbf{b}:=\mathbf{a}\otimes\mathbf{b}+\mathbf{b}\otimes\mathbf{a}$
for $\mathbf{a},\mathbf{b}\in\mathbb{R}^{n}$.

For $T\in(0,T_{0}]$, $\epsilon\in(0,\epsilon_{0})$
and $H=(\tilde{h}^{\epsilon})_{\epsilon\in (0,\epsilon_{0})}\subset X_{T}$
with $\tilde{h}^{\epsilon}|_{t=0}=0$ we consider weak solutions $\tilde{\mathbf{w}}_{1}^{\epsilon,H}\colon \Omega_{T}\rightarrow\mathbb{R}^{2}$
and $q_{1}^{\epsilon,H}\colon \Omega_{T}\rightarrow\mathbb{R}$ of 
\begin{align}
-\Delta\tilde{\mathbf{w}}_{1}^{\epsilon,H}+\nabla q_{1}^{\epsilon,H} & =-\epsilon\operatorname{div}\big((\nabla c_{A}^{\epsilon,H}-\mathbf{h}^{H})\otimes_{s}\nabla R^{H}\big) &  & \text{in }\Omega_{T},\label{eq:w1}\\
\operatorname{div}\tilde{\mathbf{w}}_{1}^{\epsilon,H} & =0 &  & \text{in }\Omega_{T},\label{eq:w12}\\
\big(-2D_{s}\tilde{\mathbf{w}}_{1}^{\epsilon,H}+q_{1}^{\epsilon,H}\mathbf{I}\big)\cdot\mathbf{n}_{\partial\Omega} & =\alpha_{0}\tilde{\mathbf{w}}_{1}^{\epsilon,H} &  & \text{on }\partial_{T}\Omega\label{eq:w13}
\end{align}
in the sense of \cite[Subsection~2.1]{NSCH1}. Here we denote
\begin{equation*}
  R^{H}:=c^{\epsilon}-c_{A}^{\epsilon,H},
\end{equation*}
where $c^{\epsilon}\colon \Omega_{T_{0}}\rightarrow\mathbb{R}$ is a smooth
solution to (\ref{eq:StokesPart})\textendash (\ref{eq:Dirichlet2})
with $c_{0}^{\epsilon}$ defined as in (\ref{eq:canf}), for $c_{A}^{\epsilon}=c_{A}^{\epsilon,H}$
and fixed $\psi_{0}^{\epsilon}$. Note that $c^{\epsilon}$ does not
depend on $H$, as 
\[
c_{I}^{H}(x,0)=\sum_{i=0}^{M+1}\epsilon^{i}c_{i}\big(\rho^{H}(x,0),x,0\big)=\sum_{i=0}^{M+1}\epsilon^{i}c_{i}\big(\tfrac{d_{\Gamma}(x,0)}{\epsilon},x,t\big)
\]
due to $h_{i}|_{t=0}=0$ by construction for $i\in \{ 1,\ldots,M+1\} $
and $\tilde{h}^{\epsilon}|_{t=0}=0$. Moreover, we define $\mathbf{h}^{H}$
by 
\begin{equation}
\mathbf{h}^{H}(x,t):=-\xi(d_{\Gamma}(x,t))\partial_{\rho}\tilde{c}_{I}\big(\rho^{H}(x,t),x,t\big)\epsilon^{M-\frac{3}{2}}\nabla^{\Gamma}\tilde{h}^{\epsilon}(x,t)\label{eq:hh}
\end{equation}
and calculate 
\begin{align}
 & \big(\nabla c_{A}^{\epsilon,H}-\mathbf{h}^{H}\big)(\rho^{H}(x,t),x,t)\nonumber \\
 & =\xi'(d_{\Gamma}(x,t))\nabla d_{\Gamma}(x,t)c_{I}^{H}(x,t)+\xi(d_{\Gamma}(x,t))\nabla\tilde{c}_{I}(\rho^{H}(x,t),x,t)\nonumber \\
 & \quad+\xi(d_{\Gamma}(x,t))\partial_{\rho}\tilde{c}_{I}(\rho^{H}(x,t),x,t)\Big(\tfrac{1}{\epsilon}\nabla d_{\Gamma}(x,t)-\sum_{i=0}^{M}\epsilon^{i}\nabla^{\Gamma}h_{i+1}(x,t)\Big)\nonumber \\
 & \quad+\nabla\big((1-\xi(d_{\Gamma}(x,t)))(1-\xi(2d_{\mathbf{B}}(x,t)))c_{O}(x,t)+\xi(2d_{\mathbf{B}}(x,t))c_{\mathbf{B}}(x,t)\big)\label{eq:ca-h}
\end{align}
for $(x,t)\in\Omega_{T}$. We understand the right hand
side of equation (\ref{eq:w1}) as a functional in $\left(V_{0}\right)'$
given by
\begin{equation}
\mathbf{f}^{\epsilon,H}(\psi):=\int_{\Omega}\epsilon\left(\big(\nabla c_{A}^{\epsilon,H}-\mathbf{h}^{H}\big)\otimes\nabla R^{H}+\nabla R^{H}\otimes\big(\nabla c_{A}^{\epsilon,H}-\mathbf{h}^{H}\big)\right):\nabla\psi\d x\label{eq:fepsh}
\end{equation}
for all $\psi\in V_{0}$ and fixed $t\in\left[0,T\right]$. As $H\subset X_{T}$,
\cite[Theorem~2.1]{NSCH1} implies the existence of a unique weak
solution. The following technical proposition is a key
element in the proof of existence for the $\left(M-\frac{1}{2}\right)$-th
order of the expansion of $h^{\epsilon}$, cf.\ Theorem \ref{hM-1} below.
\begin{prop}
\label{w1ab} Let $\epsilon_{0}\in (0,1)$ and $T'\in (0,T_{0}]$
be fixed. Furthermore, let for a given family $H=(\tilde{h}^{\epsilon})_{\epsilon\in (0,\epsilon_{0})}\subset X_{T'}$
with $\tilde{h}^{\epsilon}|_{t=0}=0$ the function $\tilde{\mathbf{w}}_{1}^{\epsilon,H}$
be defined as the weak solution to \eqref{eq:w1}\textendash \eqref{eq:w13}
for $\epsilon\in (0,\epsilon_{0})$. Then the following
statements hold:
\begin{enumerate}
\item For all $\epsilon\in (0,\epsilon_{0})$, there exists a
constant $C (\epsilon)>0$ such that 
\[
\Vert \tilde{\mathbf{w}}_{1}^{\epsilon,H}\Vert _{L^{2}(0,T';H^{1}(\Omega))}\leq C(\epsilon)\big((T')^{\frac{1}{2}}+\Vert \tilde{h}^{\epsilon}\Vert _{L^{2}(0,T';H^{1}(\mathbb{T}^1))}\big).
\]
\item Let $H_{1}=(h_{1}^{\epsilon})_{\epsilon\in(0,\epsilon_{0})},H_{2}=(h_{2}^{\epsilon})_{\epsilon\in (0,\epsilon_{0})}\subset X_{T'}$
be given. For all $\epsilon\in (0,\epsilon_{0})$, there
exists a constant $\tilde{C}(\epsilon)>0$ such that 
\[
\Vert \tilde{\mathbf{w}}_{1}^{\epsilon,H_{1}}-\tilde{\mathbf{w}}_{1}^{\epsilon,H_{2}}\Vert_{L^{2}(0,T';H^{1}(\Omega))}\leq\tilde{C}(\epsilon)(T')^{\frac{1}{2}}\big(1+\Vert h_{2}^{\epsilon}\Vert _{X_{T'}}\big)\Vert h_{1}^{\epsilon}-h_{2}^{\epsilon}\Vert _{X_{T'}}.
\]
\end{enumerate}
\end{prop}
\begin{proof}
Ad 1.: By \cite[Theorem 2.1]{NSCH1} there is a constant $C>0$ such
that 
\begin{equation}
\Vert \tilde{\mathbf{w}}_{1}^{\epsilon,H}\Vert _{L^{2}(0,T',H^{1}(\Omega))}\leq C\epsilon\big\Vert (\nabla c_{A}^{\epsilon,H}-\mathbf{h}^{H})\otimes_{s}\nabla R^{H}\big\Vert _{L^{2}\left(0,T';L^{2}(\Omega)\right)}.\label{eq:w1recht}
\end{equation}
 Now in order to estimate the right hand side, we first note
that 
\begin{equation}
\sup_{(x,t)\in\Omega\times\left(0,T'\right)}\left|\nabla c_{A}^{\epsilon,H}(x,t)-\mathbf{h}^{H}(x,t)\right|\leq\frac{C}{\epsilon},\label{eq:linf1}
\end{equation}
with a constant $C>0$ that does not depend on $H$. This can be deduced
from the representation (\ref{eq:ca-h}) and the fact that $c_{\mathbf{B}}$
and its appearing derivatives are in $L^{\infty}\left(\partial_{T_{0}}\Omega(\delta)\right)$,
$c_{O}$ and its derivatives are in $L^{\infty}\left(\Omega_{T_{0}}\right)$
and $\tilde{c}_{I}$ and its appearing derivatives are in $L^{\infty}\left(\mathbb{R}\times\Gamma\left(2\delta;T_{0}\right)\right)$.
So we obtain
\begin{align*}
\left\Vert \epsilon(\nabla c_{A}^{\epsilon,H}-\mathbf{h}^{H})\otimes_s\nabla R^{H}\right\Vert _{L^{2}\left(\Omega_{T'}\right)} & \leq C_{1}(\epsilon)(T')^{\frac{1}{2}}+C_{2}(\epsilon)\Vert \nabla^{\Gamma}\tilde{h}^{\epsilon}\Vert _{L^{2}(\Gamma(2\delta;T'))}\\
 & \leq C(\epsilon)\big((T')^{\frac{1}{2}}+\Vert h^{\epsilon}\Vert _{\left(L^{2}\left(0,T';H^{1}(\mathbb{T}^1)\right)\right)}\big),
\end{align*}
where we used that $c^{\epsilon}$ is a known function and thus 
\begin{equation}
\sup_{t\in\left(0,T'\right)}\left\Vert \nabla c^{\epsilon}\right\Vert _{L^{2}(\Omega)}\leq C(\epsilon)\label{eq:linf2}
\end{equation}
holds for some $\epsilon$-dependent constant $C(\epsilon)$.
An analoguous estimate for $\nabla R^{H}\otimes(\nabla c_{A}^{\epsilon,H}-\mathbf{h}^{H})$
yields the first part of the proposition.

Ad 2.: We write $\mathbf{f}^{\epsilon,H}:=\epsilon\big(\nabla c_{A}^{\epsilon,H}-\mathbf{h}^{H}\big)\otimes_{s}\nabla\big(c^{\epsilon}-c_{A}^{\epsilon,H}\big)$
and get using \cite[Theorem 2.1]{NSCH1} that
\begin{equation}
\Vert \tilde{\mathbf{w}}_{1}^{\epsilon,H_{1}}-\tilde{\mathbf{w}}_{1}^{\epsilon,H_{2}}\Vert _{L^{2}(0,T';H^{1}(\Omega))}\leq C\Vert \mathbf{f}^{\epsilon,H_{1}}-\mathbf{f}^{\epsilon,H_{2}}\Vert _{L^{2}(0,T';L^{2}(\Omega))}.\label{eq:w1lip}
\end{equation}
Now in order to show the statement we first note that 
\[
D_{\rho}^{k}D_{x}^{l}\big(\tilde{c}_{I}\big(\rho^{H_{1}}(x,t),x,t\big)-\tilde{c}_{I}\big(\rho^{H_{2}}(x,t),x,t\big)\big)=D_{\rho}^{k+1}D_{x}^{l}c_{I}\big(\xi(x,t),x,t\big)\epsilon^{M-\frac{3}{2}}(h_{2}^{\epsilon}-h_{1}^{\epsilon})
\]
for all $(x,t)\in\Gamma(2\delta,T')$ and $k,l\in\{ 0,1\} $
due to Taylor's theorem. Here $\xi\colon\Gamma (2\delta,T')\rightarrow\mathbb{R}$
is a suitable function depending on $H_{1}$ and $H_{2}$. Since all
the terms which do not depend on $H_{1}$, $H_{2}$ cancel, we may
estimate
\[
\epsilon\big\Vert \big((\nabla c_{A}^{\epsilon,H_{1}}-\mathbf{h}^{H_{1}})-(\nabla c_{A}^{\epsilon,H_{2}}-\mathbf{h}^{H_{2}})\big)\otimes_s\nabla c^{\epsilon}\big\Vert _{L^{2}(\Omega_{T'})}\leq C(\epsilon)(T')^{\frac{1}{2}}\Vert h_{1}^{\epsilon}-h_{2}^{\epsilon}\Vert _{X_{T'}}
\]
by (\ref{eq:linf2}), a Taylor expansion and $X_{T'}\hookrightarrow C^{0}([0,T'];C^{1}(\mathbb{T}^1))$.
With the help of a similar argumentation the other terms in $\mathbf{f}^{\epsilon,H_{1}}-\mathbf{f}^{\epsilon,H_{2}}$
may be treated, yielding the claim.
\end{proof}

\subsection{Constructing the $\left(M-\frac{1}{2}\right)$-th Terms\label{sec:Constructing-the--M-0,5}}

Our goal is to construct approximate solutions $(\mathbf{v}_{A}^{\epsilon},p_{A}^{\epsilon},c_{A}^{\epsilon},\mu_{A}^{\epsilon})$
which fulfill (\ref{eq:Stokesapp})\textendash (\ref{eq:Hilliardapp})
in $\Omega_{T_{0}}$, where $\mathbf{r}_{\text{S}}^{\epsilon},r_{\operatorname{div}}^{\epsilon},r_{\text{CH}1}^{\epsilon}$
and $r_{\text{CH}2}^{\epsilon}$ are suitable error terms, which will
be discussed in detail in Chapter \ref{chap:Estimates-Remainder}.
In (\ref{eq:Cahnapp}) we consider
\begin{equation}
\mathbf{w}_{1}^{\epsilon,H}=\frac{\tilde{\mathbf{w}}_{1}^{\epsilon,H}}{\epsilon^{M-\frac{1}{2}}}\label{eq:w1eh}
\end{equation}
instead of $\mathbf{w}_{1}^{\epsilon}$, where $\tilde{\mathbf{w}}_{1}^{\epsilon,H}$
is the weak solution to (\ref{eq:w1})\textendash (\ref{eq:w13}).
Moreover, we write
\begin{equation*}
\mathbf{w}_{1}^{\epsilon,H}|_{\Gamma}(x,t):=\mathbf{w}_{1}^{\epsilon,H}\left(\operatorname{Pr}_{\Gamma_{t}}(x),t\right)  \quad\text{for }(x,t)\in\Gamma(2\delta;T_{0})
\end{equation*}
 and we
use a suitable family $H=(\tilde{h}^{\epsilon})_{\epsilon\in (0,\epsilon_{0})}\subset X_{T_{0}}$.
Due to this appearance of a non-integer order term, it is natural
to also consider non-integer order terms in the expansion of $(c^{\epsilon},\mu^{\epsilon},\mathbf{v}^{\epsilon},p^{\epsilon})$.
More precisely, we assume that terms $\epsilon^{M-\frac{1}{2}}(\mathbf{v}_{M-\frac{1}{2}}^{\pm},p_{M-\frac{1}{2}}^{\pm},c_{M-\frac{1}{2}}^{\pm},\mu_{M-\frac{1}{2}}^{\pm})$ (defined in $\Omega_{T_{0}}^{\pm}$)
appear in the outer expansion
and that terms $\epsilon^{M-\frac{1}{2}}(\mathbf{v}_{M-\frac{1}{2}},p_{M-\frac{1}{2}},c_{M-\frac{1}{2}},\mu_{M-\frac{1}{2}})$ (defined in $\mathbb{R}\times\Gamma(2\delta;T_{0})$)
appear in the inner expansion. 

Moreover, we assume that there is a term $\epsilon^{M-\frac{3}{2}}h_{M-\frac{1}{2}}\colon \mathbb{T}^{1}\times[0,T_{0}]\rightarrow\mathbb{R}$
appearing in the expansion of $h^{\epsilon}$ (and we sometimes write
$h_{M-\frac{1}{2}}^{\epsilon}=h_{M-\frac{1}{2}}$) and further that
there are $\epsilon^{M-\frac{1}{2}}\mathbf{u}_{M-\frac{1}{2}}$ and
$\epsilon^{M-\frac{1}{2}}l_{M-\frac{1}{2}}$ appearing in the expansions
of $\mathbf{u}^{\epsilon}$ and $l^{\epsilon}$. We assume that all
these functions are smooth in their respective domains; thus we can
also consider $\mathbf{w}_{1}^{\epsilon,H}$ and $\mathbf{w}_{2}^{\epsilon,H}$
to be smooth, due to regularity theory. Note that we do not introduce
$\mathbf{q}_{M-\frac{1}{2}}$, $j_{M-\frac{1}{2}}$ or $g_{M-\frac{1}{2}}$.
In the following, we will fix $\overline{H}=\left(h_{M-\frac{1}{2}}^{\epsilon}\right)_{\epsilon\in\left(0,\epsilon_{0}\right)}$
and drop the explicit dependence on a family $H$ in the notations
when referring to $\overline{H}$, i.e.\ we write $\mathbf{h}=\mathbf{h}^{\overline{H}}$,
$\twe=\tilde{\mathbf{w}}_{1}^{\epsilon,\overline{H}}$ and so forth. 

In the following, we only assume that the zeroth and first order terms
have been constructed with the help of Lemmata \ref{zeroorder} and
\ref{k-thorder}.

\subsubsection{The Outer Expansion}

Using a Taylor expansion in (\ref{eq:CH-Part2}) as before, we explicitly
get in $\Omega_{T_{0}}^{\pm}$
\begin{equation}
c_{M-\frac{1}{2}}^{\pm}=0,\label{eq:co0,5}
\end{equation}
which can be derived similarly to (\ref{eq:Outercdefine}). From (\ref{eq:StokesPart})\textendash (\ref{eq:StokesPart2}),
we deduce that the equations 
\begin{alignat}{2}
-\Delta\mathbf{v}_{M-\frac{1}{2}}^{\pm}+\nabla p_{M-\frac{1}{2}}^{\pm} & =0 & \qquad & \text{in }\Omega_{T_{0}}^{\pm},\label{eq:sto0,5-1}\\
\operatorname{div}\mathbf{v}_{M-\frac{1}{2}}^{\pm} & =0 &  & \text{in }\Omega_{T_{0}}^{\pm},\label{eq:sto0,5-2}
\end{alignat}
have to hold, as $\nabla c_{M-\frac{1}{2}}^{\pm}=\nabla c_{0}^{\pm}=0$. 
Using $c_{M-\frac{1}{2}}^{\pm}=0$ in (\ref{eq:CH-Part1}), we get
\begin{alignat}{2}
\Delta\mu_{M-\frac{1}{2}}^{\pm} & =\partial_{t}c_{M-\frac{1}{2}}^{\pm}+\mathbf{v}_{M-\frac{1}{2}}^{\pm}\cdot\nabla c_{0}^{\pm}+\mathbf{v}_{0}^{\pm}\cdot\nabla c_{M-\frac{1}{2}}^{\pm}=0 &\quad & \text{in }\Omega_{T_{0}}^{\pm}.\label{eq:muo0,5}
\end{alignat}
We get corresponding boundary conditions for (\ref{eq:sto0,5-1})\textendash (\ref{eq:sto0,5-2})
and (\ref{eq:muo0,5}) on $\Gamma$ from the inner expansion. These
boundary conditions will turn out to  be non-trivial. But note that,
since $c_{M-\frac{1}{2}}^{\pm}=0$, we do not have to construct a
boundary layer expansion, as we may explicitly prescribe the boundary
values 
\[
\big(-2D_{s}\mathbf{v}_{M-\frac{1}{2}}^{-}+p_{M-\frac{1}{2}}^{-}\mathbf{I}\big)\mathbf{n}_{\partial\Omega}=\alpha_{0}\mathbf{v}_{M-\frac{1}{2}}^{-}\quad\text{ on }\partial_{T_{0}}\Omega
\]
for (\ref{eq:sto0,5-1})\textendash (\ref{eq:sto0,5-2}) and the Dirichlet
datum
$
\mu_{M-\frac{1}{2}}^{-}=0
$
for (\ref{eq:muo0,5}).

In the following, we assume that $\big(\mathbf{v}_{M-\frac{1}{2}}^{\pm},p_{M-\frac{1}{2}}^{\pm},c_{M-\frac{1}{2}}^{\pm},\mu_{M-\frac{1}{2}}^{\pm}\big)$
are smoothly extended to $\Omega_{T_{0}}^{\pm}\cup\Gamma(2\delta;T_{0})$,
as discussed in Remark \ref{Outer-Rem} for the integer order terms.

\subsubsection{The Inner Expansion}

We assume that the matching conditions (\ref{eq:matchcon}) hold for
the inner terms $\mathbf{v}_{M-\frac{1}{2}}$, $p_{M-\frac{1}{2}}$,
$c_{M-\frac{1}{2}}$, $\mu_{M-\frac{1}{2}}$. As these are the first
terms of fractional order which we introduce, the following identities
can be derived from (\ref{eq:expx-stokes})\textendash (\ref{eq:expx-hilliard}):
\begin{align}
-\partial_{\rho\rho}\big(\mathbf{v}_{M-\frac{1}{2}}-(\mathbf{u}_{M-\frac{1}{2}}d_{\Gamma}-\mathbf{u}_{0}h_{M-\frac{1}{2}})\eta\big) & =0,\label{eq:sti0,5}\\
\partial_{\rho}\big(\mathbf{v}_{M-\frac{1}{2}}\cdot\mathbf{n}-\big(\mathbf{u}_{M-\frac{1}{2}}d_{\Gamma}-\mathbf{u}_{0}h_{M-\frac{1}{2}}\big)\cdot\mathbf{n}\eta\big) & =0,\label{eq:divi0,5}\\
\partial_{\rho\rho}c_{M-\frac{1}{2}}-f''\left(c_{0}\right)c_{M-\frac{1}{2}} & =0,\label{eq:ci0,5}\\
\partial_{\rho\rho}\big(\mu_{M-\frac{1}{2}}-(l_{M-\frac{1}{2}}d_{\Gamma}-l_{0}h_{M-\frac{1}{2}})\eta\big) & =0\label{eq:mui0,5}
\end{align}
in $\mathbb{R}\times\Gamma (2\delta;T_{0})$. Note that
we have used $\nabla^{\Gamma}h_{M-\frac{1}{2}}\cdot (\partial_{\rho}\mathbf{v}_{0}-\mathbf{u}_{0}d_{\Gamma}\eta')=0$
in $\mathbb{R}\times\Gamma(2\delta;T_{0})$ as stated in
Remark \ref{rem:zeroprop}. As before, we complement (\ref{eq:ci0,5})
with the normalization $c_{M-\frac{1}{2}}(0,x,t)=0$ for
all $(x,t)\in\Gamma(2\delta;T_{0})$. Then we
immediately find that $c_{M-\frac{1}{2}}=0$ is the unique solution
to (\ref{eq:ci0,5}).

Now we introduce terms $V^{M-\frac{1}{2}}$, $W^{M-\frac{1}{2}}$,
$A^{M-\frac{1}{2}}$, $B^{M-\frac{1}{2}}$ which correspond to the
respective terms in (\ref{eq:Stokesk})\textendash (\ref{eq:Hilliardk})
for order $k=M+\frac{1}{2}$, i.e., right hand sides for fictive
terms $\big(\mathbf{v}_{M+\frac{1}{2}},p_{M+\frac{1}{2}},c_{M+\frac{1}{2}},\mu_{M+\frac{1}{2}}\big)$
which we will not construct. These are given by
\begin{align}
A^{M-\frac{1}{2}}&=-\mu_{M-\frac{1}{2}}-2\partial_{\rho\rho}c_{0}\nabla^{\Gamma}h_{M-\frac{1}{2}}\cdot\nabla^{\Gamma}h_{1}+\partial_{\rho}c_{0}\Delta^{\Gamma}h_{M-\frac{1}{2}}-g_{0}h_{M-\frac{1}{2}}\eta',\label{eq:AM-0,5}\\
B^{M-\frac{1}{2}} & =\partial_{\rho}c_{0}\mathbf{v}_{M-\frac{1}{2}}\cdot\mathbf{n}-\partial_{\rho}\mu_{M-\frac{1}{2}}\Delta d_{\Gamma}-2\nabla\partial_{\rho}\mu_{M-\frac{1}{2}}\cdot\mathbf{n}-l_{M-\frac{1}{2}}\eta''\left(\rho+h_{1}\right)\nonumber \\
 & \quad-\partial_{\rho}c_{0}\mathbf{v}_{0}\cdot\nabla^{\Gamma}h_{M-\frac{1}{2}}-2\partial_{\rho\rho}\mu_{0}\nabla^{\Gamma}h_{M-\frac{1}{2}}\cdot\nabla h_{1}-\partial_{\rho}c_{0}\partial_{t}^{\Gamma}h_{M-\frac{1}{2}}+\partial_{\rho}\mu_{\text{0}}\Delta^{\Gamma}h_{M-\frac{1}{2}}\nonumber \\
 & \quad+2\nabla\partial_{\rho}\mu_{0}\cdot\nabla^{\Gamma}h_{M-\frac{1}{2}}-h_{M-\frac{1}{2}}\left(l_{1}\eta''+j_{0}\eta'\right)+\mathbf{w}_{1}^{\epsilon}|_{\Gamma}\cdot\mathbf{n}\partial_{\rho}c_{0},\label{eq:BM-0,5}\\
\mathbf{V}^{M-\frac{1}{2}} & =\partial_{\rho}\mathbf{v}_{M-\frac{1}{2}}\Delta d_{\Gamma}+2\big((\nabla\partial_{\rho}\mathbf{v}_{M-\frac{1}{2}})^{T}\mathbf{n}-(\nabla\partial_{\rho}\mathbf{v}_{0})^{T}\nabla^{\Gamma}h_{M-\frac{1}{2}}\big)-\partial_{\rho}\mathbf{v}_{0}\Delta^{\Gamma}h_{M-\frac{1}{2}}\nonumber \\
 & \quad-\partial_{\rho}p_{M-\frac{1}{2}}\mathbf{n}+\partial_{\rho}p_{0}\nabla^{\Gamma}h_{M-\frac{1}{2}}+2\partial_{\rho\rho}\mathbf{v}_{0}\nabla^{\Gamma}h_{M-\frac{1}{2}}\cdot\nabla^{\Gamma}h_{1}+\mu_{M-\frac{1}{2}}\partial_{\rho}c_{0}\mathbf{n}\nonumber \\
 & \quad-\mu_{0}\partial_{\rho}c_{0}\nabla^{\Gamma}h_{M-\frac{1}{2}}+\left(\rho+h_{1}\right)\mathbf{u}_{M-\frac{1}{2}}\eta''+h_{M-\frac{1}{2}}(\mathbf{u}_{1}\eta''-\mathbf{q}_{0}\eta')\label{eq:VM-0,5}
\end{align}
and 
\begin{align}
W^{M-\frac{1}{2}} & =\partial_{\rho}\mathbf{v}_{M-\frac{1}{2}}\nabla^{\Gamma}h_{1}+\partial_{\rho}\mathbf{v}_{1}\nabla^{\Gamma}h_{M-\frac{1}{2}}-\operatorname{div}\mathbf{v}_{M-\frac{1}{2}}-\mathbf{u}_{M-\frac{1}{2}}\cdot\mathbf{n}\eta'\left(\rho+h_{1}\right)\nonumber \\
 & \quad-\mathbf{u}_{1}\cdot\mathbf{n}\eta'h_{M-\frac{1}{2}}-(\mathbf{u}_{M-\frac{1}{2}}\cdot\nabla^{\Gamma}h_{1}+\mathbf{u}_{1}\cdot\nabla^{\Gamma}h_{M-\frac{1}{2}})d_{\Gamma}\eta'\nonumber \\
 & \quad+\mathbf{u}_{0}\cdot\big(\nabla^{\Gamma}h_{M-\frac{1}{2}}\rho+(\nabla^{\Gamma}h_{M-\frac{1}{2}}h_{1}+\nabla^{\Gamma}h_{1}h_{M-\frac{1}{2}})\big)\eta'.\label{eq:WM-0,5}
\end{align}
 Note the appearance of $\mathbf{w}_{1}^{\epsilon}|_{\Gamma}\cdot\mathbf{n}\partial_{\rho}c_{0}$
in (\ref{eq:BM-0,5}) which is due to the fact that we want to approximate
(\ref{eq:Cahnapp}). In the following corollary we use the notation
$$
\left[u_{k}\right]:=u_{k}^{+}-u_{k}^{-}
$$
for terms $u_{k}^{\pm}$
of the asymptotic expansion.
\begin{cor}
\label{compchM}Let $\epsilon>0$, the zeroth and first order terms
be given as in Lemmata \ref{zeroorder} and \ref{k-thorder} and assume
that $\big(\mathbf{v}_{M-\frac{1}{2}},p_{M-\frac{1}{2}},c_{M-\frac{1}{2}},\mu_{M-\frac{1}{2}}\big)$
satisfy the matching conditions \eqref{eq:matchcon} for $k=M-\frac{1}{2}$.
Then it holds 
\begin{enumerate}
\item $\int_{\mathbb{R}}A^{M-\frac{1}{2}}\theta_{0}'\d\rho=0$ for all
$(x,t)\in\Gamma$ if and only if
\begin{equation}
\tfrac{1}{2}\int_{\mathbb{R}}\mu_{M-\frac{1}{2}}\theta_{0}'\d\rho=\sigma\Delta^{\Gamma}h_{M-\frac{1}{2}}-g_{0}h_{M-\frac{1}{2}}\tfrac{1}{2}\int_{\mathbb{R}}\eta'\theta_{0}'\d\rho\quad\text{ on }\Gamma\label{eq:compcahnM}
\end{equation}
where $\sigma$ is given as in \eqref{eq:sigma}
\item $\int_{\mathbb{R}}B^{M-\frac{1}{2}}\d\rho=0$ for all $(x,t)\in\Gamma$
if and only if
\begin{align}
0 & =\int_{\mathbb{R}}\theta_{0}'\big(\mathbf{v}_{M-\frac{1}{2}}\cdot\mathbf{n}-\mathbf{v}_{0}\cdot\nabla^{\Gamma}h_{M-\frac{1}{2}}\big)\d\rho-\big[\mu_{M-\frac{1}{2}}\big]\Delta d_{\Gamma}-2\big[\nabla\mu_{M-\frac{1}{2}}\big]\cdot\mathbf{n}+l_{M-\frac{1}{2}}\nonumber \\
 & \quad-2\partial_{t}^{\Gamma}h_{M-\frac{1}{2}}+\left[\mu_{0}\right]\Delta^{\Gamma}h_{M-\frac{1}{2}}+2\left[\nabla\mu_{0}\right]\cdot\nabla^{\Gamma}h_{M-\frac{1}{2}}-j_{0}h_{M-\frac{1}{2}}+2\mathbf{w}_{1}^{\epsilon}|_{\Gamma}\cdot\mathbf{n}\label{eq:comphilliM}
\end{align}
on $\Gamma$.
\item $\int_{\mathbb{R}}\mathbf{V}^{M-\frac{1}{2}}\cdot\mathbf{n}\d\rho=0$
for all $(x,t)\in\Gamma$ if and only if
\begin{align}
0 & =-\big[p_{M-\frac{1}{2}}\big]+\big[\mathbf{v}_{M-\frac{1}{2}}\big]\cdot\mathbf{n}\Delta d_{\Gamma}+2\Big(\big[\nabla\mathbf{v}_{M-\frac{1}{2}}\big]^{T}\mathbf{n}-[\nabla\mathbf{v}_{0}]^{T}\nabla^{\Gamma}h_{M-\frac{1}{2}}\Big)\cdot\mathbf{n}\nonumber \\
 & \quad+\int_{\mathbb{R}}\mu_{M-\frac{1}{2}}\theta_{0}'\d\rho-\mathbf{q}_{0}\cdot\mathbf{n}h_{M-\frac{1}{2}}-\mathbf{u}_{M-\frac{1}{2}}\cdot\mathbf{n} \quad\text{ on }\Gamma.\label{eq:compstokesM}
\end{align}
\item $\int_{\mathbb{R}}\mathbf{V}^{M-\frac{1}{2}}\cdot\boldsymbol{\tau}\d\rho=0$
for all $(x,t)\in\Gamma$ if and only if
\begin{align}
0 & =\big[\mathbf{v}_{M-\frac{1}{2}}\big]\cdot\boldsymbol{\tau}\Delta d_{\Gamma}+2\Big(\big[\nabla\mathbf{v}_{M-\frac{1}{2}}\big]^{T}\mathbf{n}-\big[\nabla\mathbf{v}_{0}\big]^{T}\nabla^{\Gamma}h_{M-\frac{1}{2}}\Big)\cdot\boldsymbol{\tau} +\left[p_{0}\right]\nabla^{\Gamma}h_{M-\frac{1}{2}}\cdot\boldsymbol{\tau}\nonumber \\
 & +2\sigma\Delta d_{\Gamma}\nabla^{\Gamma}h_{M-\frac{1}{2}}\cdot\boldsymbol{\tau}-\mathbf{q}_{0}\cdot\boldsymbol{\tau} h_{M-\frac{1}{2}}-\mathbf{u}_{M-\frac{1}{2}}\cdot\boldsymbol{\tau}\quad\text{ on }\Gamma.\label{eq:compstokestM}
\end{align}
\end{enumerate}
\end{cor}

\begin{proof}
This can be shown by direct calculations. 
\end{proof}

\subsubsection{Construction of Expansion Terms}

Considering the conditons (\ref{eq:compcahnM})\textendash (\ref{eq:compstokestM}),
it can be reasoned (see \cite[Subsection 5.3.3]{ichPhD}) that $\big(\mathbf{v}_{M-\frac{1}{2}}^{\pm},\mu_{M-\frac{1}{2}}^{\pm},p_{M-\frac{1}{2}}^{\pm},h_{M-\frac{1}{2}}\big)$
need to satisfy \begin{subequations}\label{eq:hmuv0,5sysgen}\foreignlanguage{ngerman}{\vspace{-5mm}
}
\begin{alignat}{2}
\Delta\mu_{M-\frac{1}{2}}^{\pm} & =0 & \qquad & \text{in }\Omega_{T_{0}}^{\pm},\label{eq:hmuv0,5sys}\\
-\Delta\mathbf{v}_{M-\frac{1}{2}}^{\pm}+\nabla p_{M-\frac{1}{2}}^{\pm} & =0 &  & \text{in }\Omega_{T_{0}}^{\pm},\label{eq:hmuv0,5sys-1}\\
\operatorname{div}\mathbf{v}_{M-\frac{1}{2}}^{\pm} & =0 &  & \text{in }\Omega_{T_{0}}^{\pm},\label{eq:hmuv0,5sys-2}\\
\mu_{M-\frac{1}{2}}^{-} & =0 &  & \text{on }\partial_{T_{0}}\Omega,\label{eq:mu0,5bdry}\\
\big(-2D_{s}\mathbf{v}_{M-\frac{1}{2}}^{-}+p_{M-\frac{1}{2}}^{-}\mathbf{I}\big)\mathbf{n}_{\partial\Omega} & =\alpha_{0}\mathbf{v}_{M-\frac{1}{2}}^{-} &  & \text{on }\partial_{T_{0}}\Omega,\label{eq:v0,5bdry}
\end{alignat}
coupled to
\begin{alignat}{2}
\mu_{M-\frac{1}{2}}^{\pm} & =\sigma\Delta^{\Gamma}h_{M-\frac{1}{2}}+(\mp\tfrac{1}{2}l_{0}-\tilde{\eta}g_{0})h_{M-\frac{1}{2}} \quad   \text{on }\Gamma,\label{eq:hmuv0,5sys-4}\\
\big[2D_{s}\mathbf{v}_{M-\frac{1}{2}}-p_{M-\frac{1}{2}}\big]\mathbf{n} & =\nabla\mathbf{u}_{0}\mathbf{n}h_{M-\frac{1}{2}}-[p_{0}]\nabla^{\Gamma}h_{M-\frac{1}{2}}+\mathbf{q}_{0}h_{M-\frac{1}{2}}+2[\nabla\mathbf{v}_{0}]\nabla^{\Gamma}h_{M-\frac{1}{2}}\nonumber \\
 & \ \ -2\sigma\Delta d_{\Gamma}\nabla^{\Gamma}h_{M-\frac{1}{2}}-2\big(\sigma\Delta^{\Gamma}h_{M-\frac{1}{2}}-g_{0}h_{M-\frac{1}{2}}\tilde{\eta}\big)\mathbf{n}\quad \text{on }\Gamma,\label{eq:hmuv0,5sys-5}\\
\big[\mathbf{v}_{M-\frac{1}{2}}\big] & =0 \quad \text{on }\Gamma,\label{eq:hmuv0,5sys-6}\\
\partial_{t}^{\Gamma}h_{M-\frac{1}{2}} & =\tfrac{1}{2}\left(l_{0}\Delta d_{\Gamma}-j_{0}+\partial_{\mathbf{n}}l_{0}\right)h_{M-\frac{1}{2}}+\mathbf{w}_{1}^{\epsilon}\cdot\mathbf{n}+\tfrac{1}{2}\big(\mathbf{v}_{M-\frac{1}{2}}^{+}+\mathbf{v}_{M-\frac{1}{2}}^{-}\big)\cdot\mathbf{n}\nonumber \\
 & \ \ -\mathbf{v}_{0}\cdot \nabla^{\Gamma}h_{M-\frac{1}{2}}-\tfrac{1}{2}\big(\partial_{\mathbf{n}}\mu_{M-\frac{1}{2}}^{+}-\partial_{\mathbf{n}}\mu_{M-\frac{1}{2}}^{-}\big)\quad \text{on }\Gamma,\label{eq:hmuv0,5sys-7}\\
h_{M-\frac{1}{2}}\big|_{t=0} & =0 \quad \text{on }\Gamma_{0},\label{eq:hmuv0,5sys-8}
\end{alignat}
\end{subequations} at the interface, where $\tilde{\eta}= \frac12\int_\R \eta'(\rho)\theta_0'(\rho)\, d\rho$. Before we may show existence
of solutions together with suitable estimates, we need the following
lemmata.
\begin{lem}
\label{lem:spekholds}Let $\epsilon_{0}>0$, $T\in (0,T_{0}]$
and a family $(T_{\epsilon})_{\epsilon\in (0,\epsilon_{0})}\subset(0,T]$
be given. We assume that there is some $\bar{C}>0$ such that
\begin{equation}
\sup_{\epsilon\in (0,\epsilon_{0})}\big\Vert h_{M-\frac{1}{2}}^{\epsilon}\big\Vert _{X_{T_{\epsilon}}}\leq\bar{C}\label{eq:vorhs}
\end{equation}
holds. Then there is $\epsilon_{1}\in(0,\epsilon_{0}]$
such that $c_{A}^{\epsilon}(.,t)$ satisfies Assumption
\ref{assu:Spektral} for all $t\in [0,T_{\epsilon}]$ and
$\epsilon\in (0,\epsilon_{1})$, where the appearing constant
$C^{*}$ does not depend on $\epsilon$, $T_{\epsilon}$, $h_{M-\frac{1}{2}}^{\epsilon}$
or $\bar{C}$.
\end{lem}

\begin{proof}
First of all, we note that there exists $\epsilon_{1}\in (0,\epsilon_{0}]$,
which depends on $\bar{C}$, such that
\begin{equation}
\Big|\frac{d_{\Gamma}(x,t)}{\epsilon}-h_{A}^{\epsilon}(S(x,t),t)\Big|\geq\frac{\delta}{2\epsilon}\label{eq:epsab1}
\end{equation}
for all $(x,t)\in\Gamma(2\delta;T_{\epsilon})\backslash\Gamma(\delta;T_{\epsilon})$
and $\epsilon\in (0,\epsilon_{1})$. This is due to the
fact that $X_{T}\hookrightarrow C^{0}([0,T];C^{1}(\mathbb{T}^1))$
and that (\ref{eq:vorhs}) holds. After possibly choosing $\epsilon_{1}>0$
smaller, we may ensure that
\begin{equation}
\left|\theta_{0}(\rho(x,t))-\chi_{\Omega^{+}}(x,t)+\chi_{\Omega^{-}}(x,t)\right|+\left|\theta_{0}'(\rho(x,t))\right|\leq C_{1}e^{-C_{2}\frac{\delta}{2\epsilon}}\label{eq:expab}
\end{equation}
holds for all $(x,t)\in\Gamma(2\delta;T_{\epsilon})\backslash\Gamma (\delta;T_{\epsilon})$
and $\epsilon\in (0,\epsilon_{1})$, as a consequence of
(\ref{eq:optimopti}), where $C_{1},C_{2}>0$ can be chosen independently
of $\epsilon_{1}$. As a last condition on $\epsilon_{1}$ we impose
that $\epsilon_{1}^{M-\frac{3}{2}}\leq\frac{1}{\bar{C}}$ such that
\begin{equation}
\epsilon^{M-\frac{3}{2}}\big\Vert h_{M-\frac{1}{2}}^{\epsilon}\big\Vert _{X_{T_{\epsilon}}}\leq1\quad \text{for all }\epsilon\in(0,\epsilon_{1}).\label{eq:epsab2}
\end{equation}
  In particular,
this implies 
\begin{equation}
\left\Vert h_{A}^{\epsilon}\right\Vert _{C^{0}\left(\left[0,T_{\epsilon}\right];C^{1}(\mathbb{T}^1)\right)}\leq C^{*}\label{eq:hHbes}
\end{equation}
 for all $\epsilon\in (0,\epsilon_{1})$, where $C^{*}$
is independent of $\bar{C}$, $h_{M-\frac{1}{2}}^{\epsilon}$ and
$T_{\epsilon}$ since the operator norm of the embedding $X_{T_{\epsilon}}\hookrightarrow C^{0}([0,T_{\epsilon}];H^{2}(\mathbb{T}^1))$
is independent of $T_{\epsilon}$, cf.\ Proposition \ref{embedding}
and since $h_{i}\in C^{0}([0,T_{0}];C^{1}(\mathbb{T}^1))$.
Thus, assumption (\ref{eq:hbes}) follows. (\ref{eq:caepm}) follows
directly from the definition of $c_{A}^{\epsilon}$ and (\ref{eq:obfi})
is a consequence of (\ref{eq:hHbes}). Similarly, (\ref{eq:fnachunten})
follows when taking (\ref{eq:expab}) into account.

Next, we show 
\[
c_{I}=\theta_{0}(\rho)+\epsilon p^{\epsilon}\big(\operatorname{Pr}_{\Gamma_{t}}\big)\theta_{1}(\rho)+\epsilon^{2}q^{\epsilon},
\]
where $\theta_{1}$ satisfies (\ref{eq:teta0}) and $p^{\epsilon}$,
$q^{\epsilon}$ satisfy (\ref{eq:pqab}). As $c_{0}=\theta_{0}$ by
Lemma \ref{zeroorder} and $c_{2},\ldots,c_{M+1}\in L^{\infty}(\mathbb{R}\times\Gamma(2\delta))$,
the only thing we need to show is that $c_{1}$ can be decomposed
suitably. By (\ref{eq:Cahnk}) and (\ref{eq:Ak-1}) $c_{1}$ satisfies
\[
\partial_{\rho\rho}c_{1}-f''(\theta_{0})c_{1}=-\mu_{0}-\theta_{0}'\Delta d_{\Gamma}+g_{0}\eta'd_{\Gamma}\qquad \text{ for all }(\rho,x,t)\in\mathbb{R}\times\Gamma(2\delta).
\]
Thus we find by (\ref{eq:mu0b}) that $\partial_{\rho\rho}c_{1}-f''(\theta_{0})c_{1}=\Delta d_{\Gamma}(\sigma-\theta_{0}')$
for all $(\rho,x,t)\in\mathbb{R}\times\Gamma$.
Hence, $c_{1}(\rho,x,t)=\Delta d_{\Gamma}(x,t)\theta_{1}(\rho)$
for all $(\rho,x,t)\in\mathbb{R}\times\Gamma$, where $\theta_{1}$
is the unique solution to 
\begin{align*}
\theta_{1}''-f''\left(\theta_{0}\right)\theta_{1} & =\sigma-\theta_{0}^{'}\qquad \text{in }\R,\quad
\theta_{1}\left(0\right)  =0
\end{align*}
with $\theta_{1}\in L^{\infty}(\mathbb{R})$.
$\theta_{1}$ exists since $\int_{\mathbb{R}}(\sigma-\theta_{0}')\theta_{0}'\d\rho=0$
by the definition of $\sigma$ (cf.\ \cite[Lemma 4.1]{abc}) . Moreover,
we have 
\begin{align*}
0 & =\int_{\mathbb{R}}\theta_{0}''\big(\sigma-\theta_{0}'\big)\d\rho=\int_{\mathbb{R}}\theta_{0}''\big(\theta_{1}''-f''(\theta_{0})\theta_{1}\big)\d\rho =\int_{\mathbb{R}}f^{\left(3\right)}(\theta_{0})(\theta_{0}')^{2}\theta_{1}\d\rho,
\end{align*}
as a consequence of (\ref{eq:optprofdef}). Thus $\theta_{1}$ satisfies
(\ref{eq:teta0}). Setting $p^{\epsilon}=\Delta d_{\Gamma}$ in $\Gamma(2\delta)$
and 
\[
\tilde{q}^{\epsilon}(x,t):=\tfrac{1}{\epsilon}\left(c_{1}(\rho(x,t),x,t)-p^{\epsilon}\big(\operatorname{Pr}_{\Gamma_{t}}(x),t\big)\theta_{1}(\rho(x,t))\right),
\]
we can write $c_{1}(\rho(x,t),x,t)=p^{\epsilon}\big(\operatorname{Pr}_{\Gamma_{t}}(x),t\big)\theta_{1}(\rho(x,t))+\epsilon\tilde{q}^{\epsilon}(x,t)$.
Now we estimate 
\begin{align*}
\epsilon\left|\tilde{q}^{\epsilon}(x,t)\right| & =\left|c_{1}(\rho(x,t),x,t)-c_{1}\big(\rho(x,t),\operatorname{Pr}_{\Gamma_{t}}(x),t\big)\right|\\
 & =\left|\nabla_{x}c_{1}\big(\rho(x,t),\xi(x),t\big)\cdot\big(x-\operatorname{Pr}_{\Gamma_{t}}(x)\big)\right| \leq\epsilon\left(C |\rho(x,t)|+C^{*}\right),
\end{align*}
where we used a Taylor expansion in the second line and the definition
of $\rho$ as well as (\ref{eq:hHbes}) in the last line. Here $C>0$
only depends on $c_{1}$, as $\left|\nabla_{x}c_{1}\right|\in L^{\infty}\left(\mathbb{R}\times\Gamma(2\delta)\right)$.
This shows assumption (\ref{eq:pqab}).
\end{proof}
\begin{lem}
\label{Wichtig} Let $\epsilon_{0}>0$, $T'\in\left(0,T_{0}\right]$
and a family $\left(T_{\epsilon}\right)_{\epsilon\in\left(0,\epsilon_{0}\right)}\subset\left(0,T'\right]$
be given. Let Assumption \ref{assu:Main-est} hold for $c_{A}=c_{A}^{\epsilon}$
and we assume that there is some $\bar{C}\geq1$ such that
\begin{equation}
\sup_{\epsilon\in\left(0,\epsilon_{0}\right)}\big\Vert h_{M-\frac{1}{2}}^{\epsilon}\big\Vert _{X_{T_{\epsilon}}}\leq\bar{C}\label{eq:bes}
\end{equation}
holds. Then there exists a constant $C(K)>0$, which is
independent of $\epsilon$, $T_{\epsilon}$, $h_{M-\frac{1}{2}}^{\epsilon}$
and $\bar{C}$, and some $\epsilon_{1}\in\left(0,\epsilon_{0}\right)$
such that 
\begin{align}
\left\Vert \tilde{\mathbf{w}}_{1}^{\epsilon}\right\Vert _{L^{2}\left(0,T;H^{1}(\Omega)\right)} & \leq C(K)\epsilon^{M-\frac{1}{2}}\quad \text{for all }\epsilon\in(0,\epsilon_{1}), T\in (0,T_{\epsilon}].\label{eq:w1epsab}
\end{align}
\end{lem}
\begin{proof}
  See \cite[Lemma~4.4]{NSCH1}.
\end{proof}

Now we show an existence result for the fractional order terms.
\begin{thm}
\label{hM-1} Let $\epsilon_{0}\in\left(0,1\right)$.
\begin{enumerate}
\item There exist unique solutions $h_{M-\frac{1}{2}}^{\epsilon}\in X_{T_{0}}$,
$\mu_{M-\frac{1}{2}}^{\pm,\epsilon}\in L^{2}(0,T_{0};H^{2}(\Omega^{\pm}(t)))$
and
\[
\big(\mathbf{v}_{M-\frac{1}{2}}^{\pm,\epsilon},p_{M-\frac{1}{2}}^{\pm,\epsilon}\big)\in L^{2}(0,T_{0};H^{2}(\Omega^{\pm}(t)))\times L^{2}(0,T_{0};H^{1}(\Omega^{\pm}(t)))
\]
 of \eqref{eq:hmuv0,5sys}\textendash \eqref{eq:hmuv0,5sys-8} for
all $\epsilon\in(0,\epsilon_{0})$, where 
$
\epsilon^{M-\frac{1}{2}}\mathbf{w}_{1}^{\epsilon}=\tilde{\mathbf{w}}_{1}^{\epsilon}\in L^{2}(0,T_{0};V_{0})
$
 is the weak solution of \eqref{eq:w1}\textendash \eqref{eq:w13}
with $H=\big(h_{M-\frac{1}{2}}^{\epsilon}\big)_{\epsilon\in (0,\epsilon_{0})}$.
\item If Assumption \ref{assu:Main-est} holds true for $c_{A}=c_{A}^{\epsilon,H}$,
there exist $\epsilon_{1}\in (0,\epsilon_{0}]$ and a constant
$C(K)>0$ independent of $\epsilon$ such that 
\begin{equation}
\big\Vert h_{M-\frac{1}{2}}^{\epsilon}\big\Vert _{X_{T_{\epsilon}}}\leq C(K)\label{eq:heps0,5}
\end{equation}
and, writing $Z_{T_{\epsilon}}:=L^{2}(0,T_{\epsilon};H^{2}(\Omega^{\pm}(t)))\cap L^{6}(0,T_{\epsilon};H^{1}(\Omega^{\pm}(t)))$,
\begin{equation}
\big\Vert \mu_{M-\frac{1}{2}}^{\pm,\epsilon}\big\Vert _{Z_{T_{\epsilon}}}+\big\Vert \mathbf{v}_{M-\frac{1}{2}}^{\pm,\epsilon}\big\Vert _{L^{6}(0,T_{\epsilon};H^{2}(\Omega^{\pm}(t)))}+\big\Vert p_{M-\frac{1}{2}}^{\pm,\epsilon}\big\Vert _{L^{6}(0,T_{\epsilon};H^{1}(\Omega^{\pm}(t)))}\leq C(K)\label{eq:muv0,5est}
\end{equation}
for all $\epsilon\in(0,\epsilon_{1})$. 
\end{enumerate}
\end{thm}
\begin{proof}
It is important to be aware that $\wei$ depends on $h_{M-\frac{1}{2}}^{\epsilon}$
since it is a solution to (\ref{eq:w1}), where $c_{A}^{\epsilon}$
depends on 
\[
\rho(x,t)=\frac{d_{\Gamma}(x,t)}{\epsilon}-h_{A}^{\epsilon}\left(S(x,t),t\right)
\]
inside of $\Gamma(2\delta)$ and $h_{M-\frac{1}{2}}^{\epsilon}$
is a summand in $h_{A}^{\epsilon}$, see (\ref{eq:ha}). Here $\tilde{h}^{\epsilon}:=h_{M-\frac{1}{2}}^{\epsilon}$.
Hence we can use Theorem \ref{scharfthe} to reduce (\ref{eq:hmuv0,5sysgen})
to a fixed point equation 
\[
h_{M-\frac{1}{2}}^{\epsilon}=S_{T}\left(h_{M-\frac{1}{2}}^{\epsilon}\right)\text{ in }X_{T}.
\] 
This equation can be uniquely solved for an $h_{M-\frac{1}{2}}^{\epsilon}\in X_{T_{0}}$
by the same argumentation as in \cite[Proof of Lemma~4.2]{nsac}. The
necessary ingredients in the present case are the existence result
for the linear system, Theorem \ref{scharfthe}, and the estimate
and Lipschitz-continuity of the nonlinearity presented in Proposition
\ref{w1ab}. Details are omitted and can be found in \cite[proof of Theorem~5.32]{ichPhD}. 

It remains to prove the second statement. Let $T_{\epsilon}>0$ be
given for $\epsilon\in (0,\epsilon_{0})$ as in the assumptions.
As a consequence of Theorem \ref{scharfthe}, we have
\begin{equation}
\big\Vert h_{M-\frac{1}{2}}^{\epsilon}\big\Vert _{X_{T'}}\leq C\big\Vert X_{0}^{*}(\mathbf{w}_{1}^{\epsilon,H}\cdot\mathbf{n})\big\Vert _{L^{2}(0,T';H^{\frac{1}{2}}(\mathbb{T}^1))}\leq C_{1}\frac{1}{\epsilon^{M-\frac{1}{2}}}\big\Vert \tilde{\mathbf{w}}_{1}^{\epsilon,H}\big\Vert _{L^{2}(0,T';H^{1}(\Omega))}\label{eq:hwab-1}
\end{equation}
for all $T'\in(0,T_{0})$. Here $H:=\big(h_{M-\frac{1}{2}}^{\epsilon}\big)_{\epsilon\in\big(0,\epsilon_{0}\big)}$
and $C_{1}$ can be chosen independently of $T'$. Now we choose $C(K)$
as in Lemma \ref{Wichtig} (note that this constant is independent
of the choice of $\tilde{h}^{\epsilon}$ in the lemma) and define
$\hat{c}(K):=2C_{1}C(K)$. Then we find that
\[
T_{\epsilon}':=\sup\Big\{ t\in\left(0,T_{\epsilon}\right)\Big|\big\Vert h_{M-\frac{1}{2}}^{\epsilon}\big\Vert _{X_{t}}\leq\hat{c}(K)\Big\} 
\]
satisfies $T_{\epsilon}'>0$, due to the continuity of the norm $\Vert .\Vert _{X_{t}}$
in $t>0$ and since $h_{M-\frac{1}{2}}^{\epsilon}|_{t=0}=0$ in $H^{2}(\mathbb{T}^1)$. 

Using Lemma \ref{Wichtig} again (with $T_{\epsilon}'$ instead of
$T_{\epsilon}$), we get the existence of $\epsilon_{1}\in (0,\epsilon_{0}]$
such that 
\[
\big\Vert \tilde{\mathbf{w}}_{1}^{\epsilon,H}\big\Vert _{L^{2}(0,T_{\epsilon}';H^{1}(\Omega))}\leq C(K)\epsilon^{M-\frac{1}{2}}
\]
for all $\epsilon\in(0,\epsilon_{1})$ with the same constant
$C(K)$ as above. Thus, by (\ref{eq:hwab-1}) we have
\[
\big\Vert h_{M-\frac{1}{2}}^{\epsilon}\big\Vert _{X_{T_{\epsilon}'}}\leq\frac{\hat{c}(K)}{2}<\hat{c}(K)
\]
for all $\epsilon\in (0,\epsilon_{1})$. By the definition
of $T_{\epsilon}'$ this already implies $T_{\epsilon}'=T_{\epsilon}$. 

Finally, (\ref{eq:muv0,5est}) follows from (\ref{eq:mumaxab}) and (\ref{eq:instoab})
taken together with the embedding $H^{\frac{1}{3}}(0,T_{\epsilon};Y)\hookrightarrow L^{6}(0,T_{\epsilon};Y)$
for a Banach space $Y$ and Proposition \ref{embedding}.
\end{proof}
\begin{rem}
\label{extension} Let $\big(h_{M-\frac{1}{2}}^{\epsilon},\mathbf{v}_{M-\frac{1}{2}}^{\pm,\epsilon},p_{M-\frac{1}{2}}^{\pm,\epsilon},\mu_{M-\frac{1}{2}}^{\pm,\epsilon}\big)$
be as in Theorem \ref{hM-1} for some $\epsilon\in (0,\epsilon_{0})$. 
\begin{enumerate}
\item Since $h_{M-\frac{1}{2}}^{\epsilon}\in X_{T_{0}}$ the right hand side
of (\ref{eq:w1}) is already in $L^{2}(\Omega_{T_{0}})$,
so by regularity theory and a bootstrap argumentation, we see that
$h_{M-\frac{1}{2}}^{\epsilon}$ and $\twe$ are smooth functions,
which transfers to $\big(\mathbf{v}_{M-\frac{1}{2}}^{\pm,\epsilon},p_{M-\frac{1}{2}}^{\pm,\epsilon},\mu_{M-\frac{1}{2}}^{\pm,\epsilon}\big)$.
So the true difficulty in the following is not the missing regularity,
but the missing control of higher norms uniformly in $\epsilon$.
\item As for lower order terms, we may also extend $\mu_{M-\frac{1}{2}}^{\pm,\epsilon},$
$p_{M-\frac{1}{2}}^{\pm,\epsilon}$, $\mathbf{v}_{M-\frac{1}{2}}^{\pm,\epsilon}$
onto $\Omega_{T_{0}}^{\pm}\cup\Gamma(2\delta)$ by using
the same extension as discussed in Remark \ref{Outer-Rem}. As the
extension operator $\mathfrak{E}^{\pm}\colon W_{p}^{k}(\Omega^{\pm}(t))\rightarrow W_{p}^{k}(\mathbb{R}^{2})$
is continuous, we get in particular
\[
  \Big\Vert \mathfrak{E}^{\pm}\Big(\mu_{M-\frac{1}{2}}^{\pm,\epsilon}\Big)\Big\Vert _{H^{k}(\Omega^{\pm}(t)\cup\Gamma_{t}(2\delta))}\leq C\Big\Vert \mu_{M-\frac{1}{2}}^{\pm,\epsilon}\Big\Vert _{H^{k}(\Omega^{\pm}(t))}
\]
for $k\in\mathbb{N}$, where we can choose $C$ independently of $t\in[0,T_{0}]$.
Similar estimates hold for $p_{M-\frac{1}{2}}^{\epsilon}$ and $\mathbf{v}_{M-\frac{1}{2}}^{\epsilon}$ (for the latter see (\ref{eq:sp=0000E4terwichtig})).
\item In the following we write $c_{A}^{\epsilon}:=c_{A}^{\epsilon,H}$
for $H=\big(h_{M-\frac{1}{2}}^{\epsilon}\big)_{\epsilon\in(0,\epsilon_{0})}$,
where $c_{A}^{\epsilon,H}$ is defined in Definition \ref{firstapprx}.
\end{enumerate}
\end{rem}

\begin{lem}[The $\left(M-\frac{1}{2}\right)$\textendash th order terms]
\label{lem-constrM-0,5}~\\Let the zeroth and first order terms be given
as in Lemmata \ref{zeroorder} and \ref{k-thorder}, and let $\epsilon\in\left(0,1\right)$. 

Then we define the terms of the outer expansion $\big(h_{M-\frac{1}{2}}^{\epsilon},\mathbf{v}_{M-\frac{1}{2}}^{\pm,\epsilon},p_{M-\frac{1}{2}}^{\pm,\epsilon},\mu_{M-\frac{1}{2}}^{\pm,\epsilon}\big)$
as the unique solution to \eqref{eq:hmuv0,5sysgen} as given by Theorem~\ref{hM-1}.1 and we consider $\mathbf{v}_{M-\frac{1}{2}}^{\pm,\epsilon}$,
$p_{M-\frac{1}{2}}^{\pm,\epsilon}$, $\mu_{M-\frac{1}{2}}^{\pm,\epsilon}$
to be extended onto $\Omega_{T_{0}}^{\pm}\cup\Gamma(2\delta;T_{0})$
(cf.\ Remark \ref{extension}). Moreover, we set $c_{M-\frac{1}{2}}^{\pm,\epsilon}\equiv0$
in $\Omega_{T_{0}}^{\pm}$. We define the terms of the inner expansion
given by the functions 
$
\big(c_{M-\frac{1}{2}}^{\epsilon},\mu_{M-\frac{1}{2}}^{\epsilon},\mathbf{v}_{M-\frac{1}{2}}^{\epsilon},p_{M-\frac{1}{2}}^{\epsilon},h_{M-\frac{1}{2}}^{\epsilon}\big)
$
as $c_{M-\frac{1}{2}}^{\epsilon} \equiv0$ and
\begin{align}
\mu_{M-\frac{1}{2}}^{\epsilon}(\rho,x,t) & :=\mu_{M-\frac{1}{2}}^{+,\epsilon}(x,t)\eta(\rho)+\mu_{M-\frac{1}{2}}^{-,\epsilon}(x,t)\left(1-\eta(\rho)\right),\label{eq:mueps0,5}\\
\mathbf{v}_{M-\frac{1}{2}}^{\epsilon}(\rho,x,t) & :=\mathbf{v}_{M-\frac{1}{2}}^{+,\epsilon}(x,t)\eta(\rho)+\mathbf{v}_{M-\frac{1}{2}}^{-,\epsilon}(x,t)\left(1-\eta(\rho)\right),\label{eq:veps0,5}\\
p_{M-\frac{1}{2}}^{\epsilon}(\rho,x,t) & :=p_{M-\frac{1}{2}}^{+,\epsilon}(x,t)\eta(\rho)+p_{M-\frac{1}{2}}^{-,\epsilon}(x,t)\left(1-\eta(\rho)\right)\label{eq:peps0,5}
\end{align}
for all $(\rho,x,t)\in\mathbb{R}\times\Gamma(2\delta;T_{0})$.
Furthermore, we define 
\begin{align}
l_{M-\frac{1}{2}} & :=\begin{cases}
  \tfrac{1}{d_{\Gamma}}\big(\mu_{M-\frac{1}{2}}^{+}-\mu_{M-\frac{1}{2}}^{-}+l_{0}h_{M-\frac{1}{2}}\big) & \text{in }\Gamma (2\delta;T_{0})\backslash\Gamma,\\
\nabla d_{\Gamma}\cdot\nabla\big(\mu_{M-\frac{1}{2}}^{+}-\mu_{M-\frac{1}{2}}^{-}+l_{0}h_{M-\frac{1}{2}}\big) & \text{on }\Gamma,
\end{cases}\label{eq:lcon0,5}\\
  \mathbf{u}_{M-\frac{1}{2}} & :=\begin{cases}
\tfrac{1}{d_{\Gamma}}\big(\mathbf{v}_{M-\frac{1}{2}}^{+}-\mathbf{v}_{M-\frac{1}{2}}^{-}+\mathbf{u}_{0}h_{M-\frac{1}{2}}\big) & \text{in }\Gamma (2\delta;T_{0})\backslash\Gamma,\\
\nabla d_{\Gamma}\cdot\nabla\big(\mathbf{v}_{M-\frac{1}{2}}^{+}-\mathbf{v}_{M-\frac{1}{2}}^{-}+\mathbf{u}_{0}h_{M-\frac{1}{2}}\big) & \text{on }\Gamma.
\end{cases}\label{eq:ucon0,5}
\end{align}
Then the outer equations \eqref{eq:sto0,5-1}\textendash \eqref{eq:muo0,5},
the inner equations \eqref{eq:sti0,5}\textendash \eqref{eq:mui0,5}
and the identities \eqref{eq:compcahnM})\textendash \eqref{eq:compstokestM}
are all satisfied.
\end{lem}

\begin{proof}
As $\big(h_{M-\frac{1}{2}}^{\epsilon},\mathbf{v}_{M-\frac{1}{2}}^{\pm,\epsilon},p_{M-\frac{1}{2}}^{\pm,\epsilon},\mu_{M-\frac{1}{2}}^{\pm,\epsilon}\big)$
solves (\ref{eq:hmuv0,5sys})\textendash (\ref{eq:hmuv0,5sys-8})
it is immediately clear that the outer equations (\ref{eq:sto0,5-1})\textendash (\ref{eq:muo0,5})
are satisfied.

Concerning (\ref{eq:compcahnM}), we compute 
\[
\tfrac{1}{2}\int_{\mathbb{R}}\theta_{0}'\mu_{M-\frac{1}{2}}^{\epsilon}\d\rho=\tfrac{1}{2}\int_{\mathbb{R}}\theta_{0}'\tfrac{1}{2}\big(\mathbf{\mu}_{M-\frac{1}{2}}^{+,\epsilon}+\mu_{M-\frac{1}{2}}^{-,\epsilon}\big)\d\rho=\sigma\Delta^{\Gamma}h_{M-\frac{1}{2}}^{\epsilon}-g_{0}h_{M-\frac{1}{2}}^{\epsilon}\tfrac{1}{2}\int_{\mathbb{R}}\eta'\theta_{0}'\d\rho,
\]
where we used (\ref{eq:etateta}) in the first equality and (\ref{eq:hmuv0,5sys-4})
in the second. The validity of (\ref{eq:compstokesM}), (\ref{eq:compstokestM})
and (\ref{eq:comphilliM}) then follow by the properties of $\theta_{0}$,
$\eta$ and the definition of the zeroth order terms.

Regarding the inner equations (\ref{eq:sti0,5})\textendash (\ref{eq:mui0,5})
we compute exemplarily 
\begin{align*}
\mu_{M-\frac{1}{2}}^{\epsilon}-\big(l_{M-\frac{1}{2}}^{\epsilon}d_{\Gamma}-l_{0}h_{M-\frac{1}{2}}^{\epsilon}\big)\eta & =\mu_{M-\frac{1}{2}}^{-,\epsilon} &  & \text{ in }\Gamma(2\delta)\backslash\Gamma,\\
\mu_{M-\frac{1}{2}}^{\epsilon}+l_{0}h_{M-\frac{1}{2}}^{\epsilon}\eta & =\mu_{M-\frac{1}{2}}^{-,\epsilon} &  & \text{ on }\Gamma,
\end{align*}
where we used the definition of $l_{M-\frac{1}{2}}^{\epsilon}$ in
the first equality and $\big[\mu_{M-\frac{1}{2}}^{\epsilon}\big]=-l_{0}h_{M-\frac{1}{2}}^{\epsilon}$
on $\Gamma$ in the second equality. The latter is a consequence of
(\ref{eq:hmuv0,5sys-4}). This implies (\ref{eq:mui0,5}). Equations
(\ref{eq:sti0,5}) and (\ref{eq:divi0,5}) follow in the same way,
remarking $\mathbf{u}_{0}=0$ on $\Gamma$ (see (\ref{eq:u00})) and
$\big[\mathbf{v}_{M-\frac{1}{2}}^{\epsilon}\big]=0$ on $\Gamma$.
\end{proof}
\begin{notation}
\label{nota:M0,5}For simplicity, we often write $\mathbf{v}_{M-\frac{1}{2}}=\mathbf{v}_{M-\frac{1}{2}}^{\epsilon}$,
$\mathbf{v}_{M-\frac{1}{2}}^{\pm,\epsilon}=\mathbf{v}_{M-\frac{1}{2}}^{\pm}$
etc., especially if we consider fractional and integer expansion orders
together, as in Section \ref{sec:The-Structure-Rem}.
\end{notation}

The following lemma is a key ingredient in order to estimate the remainder
terms properly.
\begin{lem}
  \label{ABstruc}Let the $\left(M-\tfrac{1}{2}\right)$\textendash th
order terms be given as in Lemma \ref{lem-constrM-0,5}, let the
assumptions of Theorem \ref{hM-1}.2 hold and let $\epsilon\in (0,\epsilon_{1})$. 
\begin{enumerate}
\item There are $L_{1},L_{2}\in\mathbb{N}$ such that 
  \begin{alignat*}{2}
    A^{M-\frac{1}{2}}(\rho,x,t)&=\sum_{k=1}^{L_{1}}\mathtt{A}_{k}^{1}(x,t)\mathtt{A}_{k}^{2}(\rho)&\quad &\text{ for }(\rho,x,t)\in\mathbb{R}\times\Gamma(2\delta)\text{ and}\\
A^{M-\frac{1}{2}}(\rho,x,t)&=\sum_{j=1}^{L_{2}}\mathtt{A}_{j}^{1,\Gamma}(x,t)\mathtt{A}_{j}^{2,\Gamma}(\rho)&\qquad &\text{ for }(\rho,x,t)\in\mathbb{R}\times\Gamma,
  \end{alignat*}
where $\big\Vert \mathtt{A}_{k}^{2}\big\Vert _{L^{\infty}(\mathbb{R})}+\big\Vert \mathtt{A}_{j}^{2,\Gamma}\big\Vert _{L^{\infty}(\mathbb{R})}\leq C$
for some $C>0$ independent of $\epsilon$, and 
\begin{equation}
\big\Vert \mathtt{A}_{k}^{1}\big\Vert _{L^{6}(0,T_{\epsilon};L^{2}(\Gamma_{t}(2\delta)))}+\big\Vert \mathtt{A}_{j}^{1,\Gamma}\big\Vert _{L^{6}(0,T_{\epsilon};L^{2}(\Gamma_{t}))}\leq C(K)\label{eq:a1est}
\end{equation}
for all $k\in \{ 1,\ldots,L_{1}\} $, $j\in\{ 1,\ldots,L_{2}\} $.
Moreover, there are $C,\alpha>0$ independent of $\epsilon$ such
that
\begin{equation}
\left|\int_{-\tau_{1}}^{\tau_{2}}\mathtt{A}_{j}^{2,\Gamma}\theta_{0}'\d\rho\right|\leq Ce^{-\alpha\min\left\{ \tau_{1},\tau_{2}\right\} }\label{eq:expabA}
\end{equation}
for $\tau_{1},\tau_{2}>0$ large enough and all $j\in\big\{ 1,\ldots,L_{2}\big\} $.
\item There are $K_{1},K_{2}\in\mathbb{N}$ such that 
  \begin{alignat*}{2}
B^{M-\frac{1}{2}}(\rho,x,t)&=\sum_{k=1}^{K_{1}}\mathtt{B}_{k}^{1}(x,t)\mathtt{B}_{k}^{2}(\rho)&\quad&\text{for }(\rho,x,t)\in\mathbb{R}\times\Gamma(2\delta)\backslash\Gamma\text{ and}
\\
B^{M-\frac{1}{2}}(\rho,x,t)&=\sum_{j=1}^{K_{2}}\mathtt{B}_{j}^{1,\Gamma}(x,t)\mathtt{B}_{j}^{2,\Gamma}(\rho)&\qquad &\text{for }(\rho,x,t)\in\mathbb{R}\times\Gamma,
  \end{alignat*}
where $\mathtt{B}_{k}^{2},\mathtt{B}_{j}^{2,\Gamma}\in\mathcal{O}(e^{-\alpha\left|\rho\right|})$
for $\rho\rightarrow\pm\infty$ and
\begin{equation}
\big\Vert \mathtt{B}_{k}^{1}\big\Vert _{L^{2}(\Gamma(2\delta;T_{\epsilon}))}+\big\Vert \mathtt{B}_{j}^{1,\Gamma}\big\Vert _{L^{2}(0,T_{\epsilon};L^{2}(\Gamma_{t}))}\leq C(K)\label{eq:b1est}
\end{equation}
for all $k\in\{ 1,\ldots,K_{1}\} $, $j\in\{ 1,\ldots,K_{2}\}$.
Moreover, there are $C,\alpha>0$ independent of $\epsilon$ such
that
\begin{equation}
\left|\int_{-\tau_{1}}^{\tau_{2}}\mathtt{B}_{j}^{2,\Gamma}\d\rho\right|\leq Ce^{-\alpha\min\left\{ \tau_{1},\tau_{2}\right\} }\label{eq:expabB}
\end{equation}
for $\tau_{1},\tau_{2}>0$ large enough and all $j\in\left\{ 1,\ldots,K_{2}\right\} $.
\item There are $N_{1},N_{2}\in\mathbb{N}$ such that 
  \begin{alignat*}{2}
\mathbf{V}^{M-\frac{1}{2}}(\rho,x,t)&=\sum_{k=1}^{N_{1}}\mathtt{V}_{k}^{1}(x,t)\mathtt{V}_{k}^{2}(\rho,x,t)&\qquad&\text{for }(\rho,x,t)\in\mathbb{R}\times\Gamma(2\delta)\text{ and}\\
\mathbf{V}^{M-\frac{1}{2}}(\rho,x,t)&=\sum_{j=1}^{N_{2}}\mathtt{V}_{j}^{1,\Gamma}(x,t)\mathtt{V}_{j}^{2,\Gamma}(\rho,x,t)&\qquad&\text{for }(\rho,x,t)\in\mathbb{R}\times\Gamma,
  \end{alignat*}
where $\mathtt{V}_{k}^{2},\mathtt{V}_{j}^{2,\Gamma}\in\mathcal{R}_{\alpha}$
and
\[
\big\Vert \mathtt{V}_{k}^{1}\big\Vert_{L^{2}(\Gamma(2\delta;T_{\epsilon}))}+\big\Vert \mathtt{V}_{j}^{1,\Gamma}\big\Vert _{L^{2}(0,T_{\epsilon};L^{2}(\Gamma_{t}))}\leq C(K)
\]
for all $k\in\{ 1,\ldots,N_{1}\} $, $j\in\{ 1,\ldots,N_{2}\}$.
Moreover, there are $C,\alpha>0$ independent of $\epsilon$ such
that
\begin{equation}
\sup_{(x,t)\in\Gamma}\left|\int_{-\tau_{1}}^{\tau_{2}}\mathtt{V}_{j}^{2,\Gamma}\d\rho\right|\leq Ce^{-\alpha\min\left\{ \tau_{1},\tau_{2}\right\} }\label{eq:expabV}
\end{equation}
for $\tau_{1},\tau_{2}>0$ large enough and all $j\in\left\{ 1,\ldots,N_{2}\right\} $.
\end{enumerate}
\end{lem}

\begin{proof}
Ad 1.: Plugging the explicit structure of $\mu_{M-\frac{1}{2}}^{\epsilon}$
as given in (\ref{eq:mueps0,5}) into the definition of $A^{M-\frac{1}{2}}$
(see (\ref{eq:AM-0,5})) we get 
\begin{align}
A^{M-\frac{1}{2}} & =-\tfrac{1}{2}\big(\mu_{M-\frac{1}{2}}^{+,\epsilon}+\mu_{M-\frac{1}{2}}^{-,\epsilon}\big)-\big(\mu_{M-\frac{1}{2}}^{+,\epsilon}-\mu_{M-\frac{1}{2}}^{-,\epsilon}\big)(\eta-\tfrac{1}{2})\nonumber \\
 & \quad-2\partial_{\rho\rho}c_{0}\nabla^{\Gamma}h_{M-\frac{1}{2}}^{\epsilon}\cdot\nabla^{\Gamma}h_{1}+\partial_{\rho}c_{0}\Delta^{\Gamma}h_{M-\frac{1}{2}}^{\epsilon}-g_{0}h_{M-\frac{1}{2}}^{\epsilon}\eta'\label{eq:teil1}\\
 & =\Delta^{\Gamma}h_{M-\frac{1}{2}}^{\epsilon}\left(\partial_{\rho}c_{0}-\sigma\right)+g_{0}h_{M-\frac{1}{2}}^{\epsilon}(-\eta'+\tilde{\eta})-\big(\mu_{M-\frac{1}{2}}^{+,\epsilon}-\mu_{M-\frac{1}{2}}^{-,\epsilon}\big)(\eta-\tfrac{1}{2})\nonumber \\
 & \quad-2\partial_{\rho\rho}c_{0}\nabla^{\Gamma}h_{M-\frac{1}{2}}^{\epsilon}\cdot\nabla^{\Gamma}h_{1}\label{eq:teil2}
\end{align}
on $\mathbb{R}\times\Gamma$, where we used (\ref{eq:hmuv0,5sys-4})
in the second line. Since (\ref{eq:teil1}) also holds on $\mathbb{R}\times\Gamma(2\delta)$,
we immediately get the first decomposition, noting that $c_{0}(\rho,x,t)=\theta_{0}(\rho)$. 

Setting $\mathtt{A}_{1}^{1,\Gamma}=\Delta^{\Gamma}h_{M-\frac{1}{2}}^{\epsilon}$,
$\mathtt{A}_{1}^{2,\Gamma}=\partial_{\rho}c_{0}-\sigma$,
etc.\ we get the desired splitting on $\Gamma$ (with $L_{2}=4$).
It is clear by the properties of $c_{0}$ and $\eta$ that all terms
$\mathtt{A}_{k}^{2}$, $\mathtt{A}_{j}^{2,\Gamma}$ are bounded on
$\mathbb{R}$. Now 
\[
  \int_{\mathbb{R}}\Big(\partial_{\rho}c_{0}-\sigma+\Big(-\eta'+\tfrac{1}{2}\int_{\mathbb{R}}\eta'\theta_{0}'\d\rho\Big)-(\eta-\tfrac{1}{2})-\partial_{\rho\rho}c_{0}\Big)\theta_{0}'\,\d\rho=0
\]
by (\ref{eq:sigma}), (\ref{eq:etateta}) and the fact that $\partial_{\rho\rho}c_{0}\theta_{0}'=\frac{1}{2}\frac{d}{d\rho}\left(\theta_{0}'\right)^{2}$.
Since $\theta_{0}'$ has exponential decay by (\ref{eq:optimopti})
we get (\ref{eq:expabA}). 

Now note that by the definition in Remark~\ref{hnotations} we have e.g.\ 
\[
\Delta^{\Gamma}h_{M-\frac{1}{2}}^{\epsilon}(x,t)=\big(\Delta S(x,t)\partial_{s}+|\nabla S(x,t)|^{2}\partial_{ss}\big)h_{M-\frac{1}{2}}^{\epsilon}(S(x,t),t),
\]
where $S$ is a smooth function $\Gamma(2\delta;T_{0})$ with
bounded derivatives. Thus, by (\ref{eq:heps0,5}) and Proposition
\ref{embedding}.3 it follows
\[
\big\Vert \Delta^{\Gamma}h_{M-\frac{1}{2}}^{\epsilon}+g_{0}h_{M-\frac{1}{2}}^{\epsilon}-2\nabla^{\Gamma}h_{M-\frac{1}{2}}^{\epsilon}\cdot\nabla^{\Gamma}h_{1}\big\Vert _{L^{6}(0,T_{\epsilon};L^{2}(\Gamma_{t}(2\delta)))}\leq C(K)
\]
and the same estimate also holds true if we exchange $\Gamma_{t}(2\delta)$
for $\Gamma_{t}$. On the other hand, the $L^{6}(L^{2})$ estimate
for $\mu_{M-\frac{1}{2}}^{\pm,\epsilon}$ follows from the continuity
of the trace operator and the extension operator as discussed in Remark
\ref{extension}, together with (\ref{eq:muv0,5est}).

Ad 2.: We have by definition of $B^{M-\frac{1}{2}}$ in (\ref{eq:BM-0,5})
\begin{align*}
B^{M-\frac{1}{2}} & =\partial_{\rho}c_{0}\Big(\big(\tfrac{1}{2}\big(\mathbf{v}_{M-\frac{1}{2}}^{+,\epsilon}+\mathbf{v}_{M-\frac{1}{2}}^{-,\epsilon}\big)+\big(\mathbf{v}_{M-\frac{1}{2}}^{+,\epsilon}-\mathbf{v}_{M-\frac{1}{2}}^{-,\epsilon}\big)(\eta-\tfrac{1}{2})\big)\cdot\mathbf{n}\Big)-l_{M-\frac{1}{2}}^{\epsilon}\eta''\rho\\
                  & \quad+\partial_{\rho}c_{0}\big(-\mathbf{v}_{0}\cdot\nabla^{\Gamma}h_{M-\frac{1}{2}}^{\epsilon}-\partial_{t}^{\Gamma}h_{M-\frac{1}{2}}^{\epsilon}+\mathbf{w}_{1}^{\epsilon}|_{\Gamma}\cdot\mathbf{n}\big) +\eta''\big(-l_{M-\frac{1}{2}}^{\epsilon}h_{1}-h_{M-\frac{1}{2}}^{\epsilon}l_{1}\big)\\
 & \quad-\eta'\Big(\Delta d_{\Gamma}\big(\mu_{M-\frac{1}{2}}^{+,\epsilon}-\mu_{M-\frac{1}{2}}^{-,\epsilon}\big)+2\partial_{\mathbf{n}}\big(\mu_{M-\frac{1}{2}}^{+,\epsilon}-\mu_{M-\frac{1}{2}}^{-,\epsilon}\big)+j_{0}h_{M-\frac{1}{2}}^{\epsilon}\Big)\\
                  & \quad-2\partial_{\rho\rho}\mu_{0}\nabla^{\Gamma}h_{M-\frac{1}{2}}^{\epsilon}\cdot\nabla^{\Gamma}h_{1}+\partial_{\rho}\mu_{0}\Delta^{\Gamma}h_{M-\frac{1}{2}}^{\epsilon}+2\nabla\partial_{\rho}\mu_{0}\cdot\nabla^{\Gamma}h_{M-\frac{1}{2}}^{\epsilon}
\end{align*}
on $\Gamma(2\delta)$, where we used (\ref{eq:veps0,5})
and (\ref{eq:mueps0,5}). This makes the decomposition on $\mathbb{R}\times\Gamma(2\delta)$
obvious if we note that by (\ref{eq:mu0def}) we have 
\[
\nabla_{x}^{i}\partial_{\rho}^{l}\mu_{0}=\left(\nabla_{x}^{i}\left[\mu_{0}\right]\right)\partial_{\rho}^{l}\eta\qquad \text{in }\mathbb{R}\times\Gamma(2\delta), i\in\{ 0,1\}, l\in\{ 1,2\} 
\]
and it is again clear by the properties of $c_{0}=\theta_{0}$ and
$\eta$ that all terms $\mathtt{B}_{k}^{2}$ exhibit exponential decay. 

Now for the decomposition on $\Gamma$: As a consequence of (\ref{eq:hmuv0,5sys-7}),
we find
\begin{align*}
  B^{M-\frac{1}{2}} & =\big(\mathbf{v}_{M-\frac{1}{2}}^{+,\epsilon}-\mathbf{v}_{M-\frac{1}{2}}^{-,\epsilon}\big)\cdot\mathbf{n}(\eta-\tfrac{1}{2})\partial_{\rho}c_{0}+j_{0}h_{M-\frac{1}{2}}^{\epsilon}(\tfrac{1}{2}\partial_{\rho}c_{0}-\eta')\\
 & \quad-\eta'\Delta d_{\Gamma}\big(\mu_{M-\frac{1}{2}}^{+,\epsilon}-\mu_{M-\frac{1}{2}}^{-,\epsilon}\big)-\tfrac{1}{2}l_{0}\Delta d_{\Gamma}\partial_{\rho}c_{0}h_{M-\frac{1}{2}}^{\epsilon}+\partial_{\mathbf{n}}l_{0}h_{M-\frac{1}{2}}^{\epsilon}\left(-\eta''\rho-\tfrac{1}{2}\partial_{\rho}c_{0}\right)\\
 & \quad+\partial_{\mathbf{n}}\big(\mu_{M-\frac{1}{2}}^{+,\epsilon}-\mu_{M-\frac{1}{2}}^{-,\epsilon}\big)\left(\tfrac{1}{2}\partial_{\rho}c_{0}-\eta''\rho-2\eta'\right)+\eta''\left(-l_{M-\frac{1}{2}}^{\epsilon}h_{1}-h_{M-\frac{1}{2}}^{\epsilon}l_{1}\right)\\
 & \quad-2\partial_{\rho\rho}\mu_{0}\nabla^{\Gamma}h_{M-\frac{1}{2}}^{\epsilon}\cdot\nabla^{\Gamma}h_{1}+\partial_{\rho}\mu_{0}\Delta^{\Gamma}h_{M-\frac{1}{2}}^{\epsilon}+2\nabla\partial_{\rho}\mu_{0}\cdot\nabla^{\Gamma}h_{M-\frac{1}{2}}^{\epsilon}
\end{align*}
on $\mathbb{R}\times\Gamma$, where we used the structure of $l_{M-\frac{1}{2}}^{\epsilon}$
on $\Gamma$ as given in (\ref{eq:lcon0,5}). Using $\mu_{M-\frac{1}{2}}^{+,\epsilon}-\mu_{M-\frac{1}{2}}^{-,\epsilon}=-l_{0}h_{M-\frac{1}{2}}^{\epsilon}$
on $\Gamma$ due to (\ref{eq:hmuv0,5sys-4}), $\partial_{\rho\rho}\mu_{0}=\partial_{\rho}\mu_{0}=0$
on $\Gamma$ due to (\ref{eq:mu0def}) and 
\[
\nabla^{\Gamma}h_{M-\frac{1}{2}}^{\epsilon}\cdot\nabla\partial_{\rho}\mu_{0}=\nabla^{\Gamma}h_{M-\frac{1}{2}}^{\epsilon}\cdot\mathbf{n}\partial_{\mathbf{n}}\left[\mu_{0}\right]\eta'=0
\]
 on $\Gamma$ by (\ref{eq:mu0def}) and $\nabla^{\Gamma}\partial_{\rho}\mu_{0}=0$,
we arrive at
\begin{align*}
B^{M-\frac{1}{2}} & =\big(\mathbf{v}_{M-\frac{1}{2}}^{+,\epsilon}-\mathbf{v}_{M-\frac{1}{2}}^{-,\epsilon}\big)\cdot\mathbf{n}\left(\eta(\rho)-\tfrac{1}{2}\right)\partial_{\rho}c_{0}+h_{M-\frac{1}{2}}^{\epsilon}\left(j_{0}-l_{0}\Delta d_{\Gamma}\right)\left(\tfrac{1}{2}\partial_{\rho}c_{0}-\eta'\right)\\
 & \quad+\partial_{\mathbf{n}}\big(\mu_{M-\frac{1}{2}}^{+,\epsilon}-\mu_{M-\frac{1}{2}}^{-,\epsilon}\big)\left(\tfrac{1}{2}\partial_{\rho}c_{0}-\eta''\rho-2\eta'\right)+\partial_{\mathbf{n}}l_{0}h_{M-\frac{1}{2}}^{\epsilon}\left(-\eta''\rho-\tfrac{1}{2}\partial_{\rho}c_{0}\right)\\
 & \quad-\big(l_{M-\frac{1}{2}}^{\epsilon}h_{1}+h_{M-\frac{1}{2}}^{\epsilon}l_{1}\big)\eta''
\end{align*}
on $\mathbb{R}\times\Gamma$. This implies the desired decomposition
 if we set $\mathtt{B}_{1}^{1,\Gamma}=\left(\mathbf{v}_{M-\frac{1}{2}}^{+,\epsilon}-\mathbf{v}_{M-\frac{1}{2}}^{-,\epsilon}\right)\cdot\mathbf{n}$,
$\mathtt{B}_{1}^{2,\Gamma}=\left(\eta(\rho)-\frac{1}{2}\right)\partial_{\rho}c_{0}$,
etc. As before the $\mathtt{B}_{k}^{2,\Gamma}$ terms possess exponential
decay. The integral over the $\mathtt{B}_{k}^{2,\Gamma}$ terms has
exponential decay due to the properties of $\eta$ and $c_{0}$ since
e.g.\ 
\[
\int_{\mathbb{R}}\left(\tfrac{1}{2}\partial_{\rho}c_{0}-\eta''\rho-2\eta'\right)\d\rho=1+\int_{\mathbb{R}}\eta'\d\rho-2=0,\;\int_{\mathbb{R}}\eta''\d\rho=0.
\]
This implies (\ref{eq:expabB}).

The $L^{2}(L^{2})$ estimate for the terms of kind $\mathtt{B}_{k}^{1,\Gamma}$
and $\mathtt{B}_{k}^{1}$ now follows from (\ref{eq:heps0,5}), (\ref{eq:muv0,5est})
and the continuity of the trace operator $H^{1}(\Omega^{\pm}(t))\rightarrow L^{2}(\Gamma_{t})$
as well as from the continuity of the extension operators for $\mu_{M-\frac{1}{2}}^{\pm,\epsilon}$
and $\mathbf{v}_{M-\frac{1}{2}}^{\pm,\epsilon}$. 

Ad 3.: Follows in a similar fashion as the proof of the second part and is left to the reader.
\end{proof}
\begin{rem}
\label{rem:We-will-not}$\quad$We will not construct terms of order
$M+\frac{1}{2}$ as the right hand sides of the according ordinary
differential equations (similar to (\ref{eq:sti0,5})\textendash (\ref{eq:mui0,5}))
would depend on derivatives of the kind $\partial_{t}^{\Gamma}h_{M-\frac{1}{2}}^{\epsilon}$
and $\Delta h_{M-\frac{1}{2}}^{\epsilon}$ among others. As a result,
the already tenuous control (independent of $\epsilon$) we have over
the terms of order $M-\frac{1}{2}$ would only get worse for terms
of order $M+\frac{1}{2}$. On the other hand, terms like $\Delta\mu_{M+\frac{1}{2}},\partial_{t}\mathbf{v}_{M+\frac{1}{2}}$,
etc. would appear in the remainder and have to be estimated suitably,
which prohibit the missing estimates.
\end{rem}

\section{Estimates for the Remainder Terms\label{chap:Estimates-Remainder}}

In this section we will prove that the constructed approximate
solutions solve the original system (\ref{eq:StokesPart})\textendash (\ref{eq:Dirichlet2}) upto  error terms of a suitable order in $\eps$.
Throughout this section we write
\begin{equation}
I_{q}^{k}:=\left\{ 0,\ldots,k\right\} \cup\left\{ q\right\} \label{eq:Ikq}
\end{equation}
 for $k\in\mathbb{N}$ and $q\in\mathbb{R}$. The following definition
is central for the following.
\begin{defn}[The approximate solutions]
 \label{def:apprxsol}~\\ Let $\epsilon\in (0,1)$ and let
$\xi$ satisfy (\ref{eq:cut-off}). We define
\[
h_{A}^{\epsilon}(s,t):=\sum_{i\in I_{M-\frac{3}{2}}^{M}}\epsilon^{i}h_{i+1}(s,t)
\]
for $(s,t)\in\mathbb{T}^{1}\times [0,T_{0}]$
and $\rho(x,t):=\frac{d_{\Gamma}(x,t)}{\epsilon}-h_{A}^{\epsilon}(S(x,t),t)$
for $(x,t)\in\Gamma(2\delta)$, as well as $z(x,t):=\tfrac{d_{\mathbf{B}}(x,t)}{\epsilon}$
for $(x,t)\in\overline{\partial_{T_{0}}\Omega(\delta)}$.

 We define the \emph{inner solutions} as
\begin{alignat*}{2}
c_{I}(x,t) & :=\sum_{i=0}^{M+1}\epsilon^{i}c_{i}(\rho(x,t),x,t), & \mu_{I}(x,t) & :=\sum_{i\in I_{M-\frac{1}{2}}^{M+1}}\epsilon^{i}\mu_{i}(\rho(x,t),x,t),\\
\mathbf{v}_{I}(x,t) & :=\sum_{i\in I_{M-\frac{1}{2}}^{M+1}}\epsilon^{i}\mathbf{v}_{i}(\rho(x,t),x,t), &\qquad p_{I}(x,t) & :=\sum_{i\in I_{M-\frac{1}{2}}^{M}}\epsilon^{i}p_{i}(\rho(x,t),x,t),
\end{alignat*}
for all $(x,t)\in\Gamma(2\delta)$ and write
\begin{align}
c_{I,k}(x,t) & :=c_{k}\left(\rho(x,t),x,t\right) \qquad \text{for all }(x,t)\in\Gamma(2\delta)\label{eq:cik}
\end{align}
and analoguously for $\mu_{I,k}$, $\mathbf{v}_{I,k}$, $p_{I,k}$. 
We write
\begin{align}
  c_{O,k}(x,t) & :=c_{k}^{+}(x,t)\chi_{\overline{\Omega_{T_{0}}^{+}}}(x,t)+c_{k}^{-}(x,t)\chi_{\Omega_{T_{0}}^{-}}(x,t) \qquad \text{for all }(x,t)\in\Omega_{T_{0}}\label{eq:cok}
\end{align}
and analoguously for $\mu_{O,k}$, $\mathbf{v}_{O,k}$, $p_{O,k}$ and define the \emph{outer solutions} as 
\begin{alignat*}{2}
c_{O}(x,t) & :=\sum_{i=0}^{M+1}\epsilon^{i}c_{O,k}(x,t), &\qquad \mu_{O}(x,t) & :=\sum_{i\in I_{M-\frac{1}{2}}^{M+1}}\epsilon^{i}\mu_{O,k}(x,t)\\
\mathbf{v}_{O}(x,t) & :=\sum_{i\in I_{M-\frac{1}{2}}^{M+1}}\epsilon^{i}\mathbf{v}_{O,k}(x,t), & p_{O}(x,t) & :=\sum_{i\in I_{M-\frac{1}{2}}^{M}}\epsilon^{i}p_{O,k}(x,t)
\end{alignat*}
for $(x,t)\in\Omega_{T_{0}}$. We define the \emph{boundary solutions} as
\begin{alignat*}{2}
  c_{\mathbf{B}}(x,t) & :=-1+\sum_{i=1}^{M+1}\epsilon^{i}c_{i}^{\mathbf{B}}(z(x,t),x,t), \qquad \mu_{\mathbf{B}}(x,t)  :=\sum_{i\in I_{M-\frac{1}{2}}^{M+1}}\epsilon^{i}\mu_{i}^{\mathbf{B}}(z(x,t),x,t),\\
\mathbf{v}_{\mathbf{B}}(x,t)&:=\sum_{i\in I_{M-\frac{1}{2}}^{M+1}}\epsilon^{i}\mathbf{v}_{i}^{\mathbf{B}}(z(x,t),x,t)-\epsilon^{M+1}\mathbf{v}_{M+1}^{\mathbf{B}}(0,x,t),
\end{alignat*}
and $p_{\mathbf{B}}(x,t)  :=\sum_{i\in I_{M-\frac{1}{2}}^{M}}\epsilon^{i}p_{i}^{\mathbf{B}}(z(x,t),x,t)$
for $(x,t)\in\overline{\partial_{T_{0}}\Omega(\delta)}$,
where we set
\begin{align}
\mu_{M-\frac{1}{2}}^{\mathbf{B}}:=\mu_{M-\frac{1}{2}}^{-},\quad \mathbf{v}_{M-\frac{1}{2}}^{\mathbf{B}}:=\mathbf{v}_{M-\frac{1}{2}}^{\mathbf{-}},\quad p_{M-\frac{1}{2}}^{\mathbf{B}}:=p_{M-\frac{1}{2}}^{\mathbf{-}}\quad \text{in }(-\infty,0]\times\overline{\partial_{T_{0}}\Omega(\delta)}
\label{eq:M0,5bdrydef}
\end{align}
and write 
\begin{alignat}{2}
c_{\mathbf{B},k}(x,t) & :=c_{k}^{\mathbf{B}}\left(z(x,t),x,t\right)&\quad &\text{for all } (x,t)\in\overline{\partial_{T_{0}}\Omega(\delta)}\label{eq:cbk}
\end{alignat}
and similarly $\mu_{\mathbf{B},k}$, $\mathbf{v}_{\mathbf{B},k}$, $p_{\mathbf{B},k}$ with the only exception that
\[
\mathbf{v}_{\mathbf{B},M+1}(x,t)=\mathbf{v}_{M+1}^{\mathbf{B}}\left(z(x,t),x,t\right)-\mathbf{v}_{M+1}^{\mathbf{B}}(0,x,t).
\]
 We define the \emph{approximate solutions }
\begin{align}
c_{A}^{\epsilon} & :=\xi(d_{\Gamma})c_{I}+(1-\xi(d_{\Gamma}))(1-\xi(2d_{\mathbf{B}}))c_{O}+\xi (2d_{\mathbf{B}})c_{\mathbf{B}},\label{eq:apprxsol}
\end{align}
in $\Omega_{T_{0}}$ and write 
\begin{align}
c_{A,k}(x,t) & :=\xi(d_{\Gamma})c_{I,k}+\left(1-\xi(d_{\Gamma})\right)\left(1-\xi(2d_{\mathbf{B}})\right)c_{O,k}+\xi(2d_{\mathbf{B}})c_{\mathbf{B},k}\label{eq:cak}
\end{align}
for all $(x,t)\in\Omega_{T_{0}}$ . Analoguously  we define $\mu_{A}^{\epsilon}$, $\mathbf{v}_{A}^{\epsilon}$, $p_{A}^{\epsilon}$ and
$\mu_{A,k}$, $\mathbf{v}_{A,k}$, $p_{A,k}$.
\end{defn}

This definition implies in particular $\mu_{A,M-\frac{1}{2}}=\xi(d_{\Gamma})\mu_{I,M-\frac{1}{2}}+\left(1-\xi(d_{\Gamma})\right)\mu_{O,M-\frac{1}{2}}$
and a similar structure for $\mathbf{v}_{A,M-\frac{1}{2}}$, $p_{A,M-\frac{1}{2}}$. 
\begin{assumption}
\label{assu:remainder}Throughout this section we assume that Assumption
\ref{assu:Main-est} holds for $c_{A}=c_{A}^{\epsilon}$ and $\epsilon_{0}\in\left(0,1\right)$,
the family $(T_{\epsilon})_{\epsilon\in (0,\epsilon_{0})}$
and $K\geq1$. Moreover, we assume $\epsilon_{1}\in (0,\epsilon_{0}]$
is given as in Theorem \ref{hM-1}.2 and such that \eqref{eq:w1epsab}
holds for $\twe$, the weak solution to \eqref{eq:w1}\textendash \eqref{eq:w13}
with $H=\big(h_{M-\frac{1}{2}}^{\epsilon}\big)_{\epsilon\in (0,\epsilon_{0})}$.
\end{assumption}

Note in particular that the assumptions of Lemma \ref{lem:spekholds}
are satisfied in this situation. Addtitionally, there is some $C>0$
such that
\begin{equation}
\left\Vert \nabla c_{A}^{\epsilon}\right\Vert _{L^{\infty}(\Omega_{T_{0}}\backslash\Gamma(2\delta))}\leq C\epsilon\label{eq:nabla-caeps-Linf}
\end{equation}
for all $\epsilon\in (0,1)$ small enough. This is the case
since $c_{0}^{\pm}=\pm1$ in $\Omega_{T_{0}}^{\pm}$ (cf.\ (\ref{eq:c0out}))
and since $c_{0}^{\mathbf{B}}=-1$ and $c_{1}^{\mathbf{B}}=c_{1}^{-}$
in $\overline{\partial_{T_{0}}\Omega(\delta)}$ due to
Corollary (\ref{cor:bdrycond}). Moreover, it holds 
\begin{equation}
  \sup_{0\leq t\leq T_\eps }\left\Vert h_{A}^{\epsilon}(t)\right\Vert _{C^{1}(\Gamma_{t}(2\delta))}\le C(K)\label{eq:haepglm}
\end{equation}
for some $C(K)>0$ and all $\epsilon\in (0,\epsilon_{1})$.
This is a consequence of the uniform boundedness of $h_{k}$, $k\in\left\{ 1,\ldots,M+1\right\} $,
and (\ref{eq:heps0,5}) for $h_{M-\frac{1}{2}}^{\epsilon}$.
\begin{rem}
At this point, we want to remark about the shortened statements in
\cite[Subsection 3.1]{NSCH1}. Concerning the definitions of $c_{I}$,
$c_{O}$, $\mu_{I}$, etc.\ there is a discrepancy between the present contribution
and \cite{NSCH1}. In \cite{NSCH1}, $\mu_{I}$ and $\mathbf{v}_{I}$
are defined\emph{ without} the appearance of fractional order terms
and in the present context, we would define 
\[
\mu_{O,\mathbf{B}}:=\sum_{k=0}^{M+1}\epsilon^{k}\left(\left(1-\xi(2d_{\mathbf{B}})\right)\mu_{O,k}+\xi(2d_{\mathbf{B}})\mu_{\mathbf{B},k}\right),
\]
with a similar representation for $c_{O,\mathbf{B}}$ and $\mathbf{v}_{O,\mathbf{B}}$.
Again, this leaves out the fractional order terms, which are considered
separately. These notational differences help in \cite{NSCH1} to
keep the necessary structural information about the approximate solutions
as compact as possible, while still presenting enough background to
make the proofs self-contained. Now $c_{O,\mathbf{B}}=\pm1+\mathcal{O}(\epsilon)$
in $C^{1}(\Omega_{T_{0}}^{\pm})$ follows by the same arguments
as (\ref{eq:nabla-caeps-Linf}) and $\Vert c_{O,\mathbf{B}}\Vert _{C^{2}(\Omega_{T_{0}}^{\pm})}\leq C$
is a consequence of $D_{x}^{2}c_{\mathbf{B}}=\mathcal{O}(1)$.
$\mu_{O,\mathbf{B}}=\mu^{\pm}+\mathcal{O}(\epsilon)$ and
$\mathbf{v}_{O,\mathbf{B}}=\mathbf{v}^{\pm}+\mathcal{O}(\epsilon)$
in $L^{\infty}(\Omega_{T_{0}}^{\pm})$ as $\epsilon\rightarrow0$
are direct results of Lemma \ref{zeroorder}. $h_{A}^{\epsilon}(s,0)=0$
for all $s\in\mathbb{T}^{1}$ is a consequence of Lemma \ref{k-thorder}
and (\ref{eq:hmuv0,5sys-8}), while $\mu_{M-\frac{1}{2}}^{\epsilon,-}=0$
on $\partial_{T_{0}}\Omega$ is due to (\ref{eq:mu0,5bdry}). 
\end{rem}

\subsection{The Structure of the Remainder Terms\label{sec:The-Structure-Rem}}

\subsubsection{The Inner Remainder Terms}

In the following, let Assumption \ref{assu:remainder} hold and we
work under the notations and assumptions of Definition \ref{def:apprxsol}.
We now analyze up to which order in $\eps$ the equations (\ref{eq:StokesPart})\textendash (\ref{eq:CH-Part2})
are fulfilled by the inner solutions $c_{I},\mu_{I},\mathbf{v}_{I},p_{I}$.
For this we use the ordinary differential equations satisfied by $(c_{k},\mu_{k},\mathbf{v}_{k},p_{k-1})$
for $k\in\left\{ 0,\ldots,M+1\right\} $ as constructed for the inner
terms and evaluate them at 
\begin{equation}
\rho(x,t)=\frac{d_{\Gamma}(x,t)}{\epsilon}-h_{A}^{\epsilon}(S(x,t),t)\label{eq:rho}
\end{equation}
for $(x,t)\in\Gamma(2\delta;T_{\epsilon})$
and $\epsilon\in (0,\epsilon_{1})$. Before we give the
explicit formula, note that we can choose $\epsilon_{1}$ so small
that for all $\epsilon\in (0,\epsilon_{1})$ we have $\left|h_{A}^{\epsilon}-h_{1}\right|\leq1$
due to (\ref{eq:heps0,5}). Thus, (\ref{eq:remhglm}) is satisfied
and using Remark \ref{WU} we get
\[
\epsilon^{2}\left(U^{+}\eta^{C_{S},+}+U^{-}\eta^{C_{S},-}\right)\big|_{\rho=\frac{d_{\Gamma}}{\epsilon}-h_{A}^{\epsilon}}=\epsilon^{2}\left(\mathbf{W}^{+}\eta^{C_{S},+}+\mathbf{W}^{-}\eta^{C_{S},-}\right)\big|_{\rho=\frac{d_{\Gamma}}{\epsilon}-h_{A}^{\epsilon}}=0.
\]

Let $\epsilon\in\left(0,\epsilon_{1}\right)$. Using
the inner equations derived in Chapter \ref{constrappr}  we get
\begin{align}
\partial_{t}c_{I} & +\mathbf{v}_{I}\cdot\nabla c_{I}+\epsilon^{M-\frac{1}{2}}\left.\mathbf{w}_{1}^{\epsilon}\right|_{\Gamma}\cdot\nabla c_{I}-\Delta\mu_{I}\nonumber \\
  =&\epsilon^{M}\left(\partial_{\rho}c_{M+1}\partial_{t}d_{\Gamma}-\partial_{\rho}\mu_{M+1}\Delta d_{\Gamma}-2\nabla\partial_{\rho}\mu_{M+1}\cdot\mathbf{n}-j_{M}\eta'\rho-l_{M+1}\eta''\rho\right)\nonumber \\
 & +\epsilon^{M-\frac{1}{2}}\mathbf{w}_{1}^{\epsilon}|_{\Gamma}\cdot\Bigg(\sum_{i=1}^{M+1}\epsilon^{i-1}\partial_{\rho}c_{i}\mathbf{n}+\sum_{i=0}^{M+1}\Bigg(-\sum_{j\in I_{M-\frac{3}{2}}^{M}}\epsilon^{i+j}\partial_{\rho}c_{i}\nabla^{\Gamma}h_{j+1}+\epsilon^{i}\nabla c_{i}\Bigg)\Bigg)\nonumber \\
 & -\sum_{\substack{0\leq i\leq M+1,
j\in I_{M-\frac{3}{2}}^{M}\\
i+j\geq M-\frac{1}{2}
}
}\epsilon^{i+j}\partial_{\rho}c_{i}\partial_{t}^{\Gamma}h_{j+1}+\sum_{\substack{i\in I_{M-\frac{1}{2}}^{M+1},j\in I_{M-\frac{3}{2}}^{M}\\
i+j\geq M-\frac{1}{2}
}
}\epsilon^{i+j}\left(2\nabla\partial_{\rho}\mu_{i}\cdot\nabla^{\Gamma}h_{j+1}+\partial_{\rho}\mu_{i}\Delta^{\Gamma}h_{j+1}\right)\nonumber \\
 & +\sum_{i\in I_{M-\frac{1}{2}}^{M+1}}\frac{1}{\epsilon}\sum_{\substack{0\leq j\leq M+1\\
i+j\geq M+\frac{1}{2}
}
}\epsilon^{i+j}\mathbf{v}_{i}\cdot\mathbf{n}\partial_{\rho}c_{j}-\sum_{\substack{0\leq j\leq M+1,\ l\in I_{M-\frac{3}{2}}^{M}\\
i+j+l\geq M-\frac{1}{2}
}
}\epsilon^{i+j+l}\mathbf{v}_{i}\cdot\partial_{\rho}c_{j}\nabla^{\Gamma}h_{l+1}\nonumber \\
 & -\sum_{\substack{i\in I_{M-\frac{1}{2}}^{M+1},j,l\in I_{M-\frac{3}{2}}^{M}\\
i+j+l\geq M-\frac{1}{2}
}
}\epsilon^{i+j+l}\partial_{\rho\rho}\mu_{i}\nabla^{\Gamma}h_{j+1}\cdot\nabla^{\Gamma}h_{l+1}-\sum_{\substack{i\in I_{M-\frac{1}{2}}^{M+1}, j\in I_{M-\frac{3}{2}}^{M}\\
i+j\geq M+\frac{1}{2}
}
}\epsilon^{i+j-1}l_{i}\eta''h_{j+1}\nonumber \\
 & -\sum_{\substack{0\leq i\leq M, k\in I_{M-\frac{3}{2}}^{M}\\
i+j\geq M-\frac{1}{2}
}
}\epsilon^{i+k}j_{i}\eta'h_{k+1}+\sum_{i=M}^{M+1}\epsilon^{i}\left(\partial_{t}c_{i}-\Delta\mu_{i}\right)-\epsilon^{M-\frac{1}{2}}\Delta\mu_{M-\frac{1}{2}}\nonumber \\
 & +\sum_{\substack{i\in I_{M-\frac{1}{2}}^{M+1}, 0\leq j\leq M\\
i+j\geq M-\frac{1}{2}
}
}\epsilon^{i+j}\mathbf{v}_{i}\cdot\nabla c_{j}+\epsilon^{M-\frac{3}{2}}B^{M-\frac{1}{2}}=:\rci\quad \text{in }\Gamma (2\delta;T_{\epsilon}),\label{eq:CH1-rem}
\end{align}
 where $\mathbf{w}_{1}^{\epsilon}$
is given as in Theorem \ref{hM-1}. We also get 
\begin{align}
\epsilon  \Delta c^{I}-&\epsilon^{-1}f'(c^{I})+\mu^{I}=\text{\ensuremath{\mathcal{O}}}(\epsilon^{M+1})\nonumber \\
 & -\epsilon\sum_{\substack{0\leq i\leq M+1, j\in I_{M-\frac{3}{2}}^{M}\\
i+j\geq M-\frac{1}{2}
}
}\epsilon^{i+j}\left(\partial_{\rho}c_{i}\Delta^{\Gamma}h_{j+1}+2\nabla\partial_{\rho}c_{i}\cdot\nabla^{\Gamma}h_{j+1}\right)+\sum_{\substack{0\leq i\leq M,
j\in I_{M-\frac{3}{2}}^{M}\\
i+j\ge M+\frac{1}{2}
}
}\epsilon^{i+j}g_{i}\eta'h_{j+1}\nonumber \\
 & +\epsilon\sum_{\substack{0\leq i\leq M, j,l\in I_{M-\frac{3}{2}}^{M}\\
i+j+l\geq M-\frac{1}{2}
}
}\epsilon^{i+j+l}\partial_{\rho\rho}c_{i}\nabla^{\Gamma}h_{j+1}\cdot\nabla^{\Gamma}h_{l+1}-\epsilon^{M-\frac{1}{2}}A^{M-\frac{1}{2}}=:\rhi\label{eq:CH2-rem}
\end{align}
in $\Gamma(2\delta;T_{\epsilon})$, where the Landau symbol
is with respect to $L^{\infty}(\Gamma(2\delta;T_{0}))$.
Furthermore, 
\begin{align}
\operatorname{div}\mathbf{v}_{I} & =\epsilon^{M+1}\operatorname{div}\mathbf{v}_{M+1}-\sum_{\substack{i\in I_{M-\frac{1}{2}}^{M+1},j\in I_{M-\frac{3}{2}}^{M}\\
i+j\geq M+\frac{1}{2}
}
}\epsilon^{i+j}\partial_{\rho}\mathbf{v}_{i}\cdot\nabla^{\Gamma}h_{j+1}+\sum_{\substack{i\in I_{M-\frac{1}{2}}^{M+1},j\in I_{M-\frac{3}{2}}^{M}\\
i+j\geq M+\frac{1}{2}
}
}\epsilon^{i+j}\mathbf{u}_{i}\cdot\mathbf{n}\eta'h_{j+1}\nonumber \\
 & \quad+\sum_{\substack{i\in I_{M-\frac{1}{2}}^{M+1},j\in I_{M-\frac{3}{2}}^{M}\\
i+j\geq M+\frac{1}{2}
}
}\epsilon^{i+j}\mathbf{u}_{i}\cdot\nabla^{\Gamma}h_{j+1}\eta'd_{\Gamma}-\epsilon\sum_{\substack{i\in I_{M-\frac{1}{2}}^{M+1},j\in I_{M-\frac{3}{2}}^{M}\\
i+j\geq M-\frac{1}{2}
}
}\epsilon^{i+j}\mathbf{u}_{i}\cdot\nabla^{\Gamma}h_{j+1}\eta'\rho\nonumber \\
 & \quad-\epsilon\sum_{\substack{i\in I_{M-\frac{1}{2}}^{M+1},j,k\in I_{M-\frac{3}{2}}^{M}\\
i+j+k\geq M-\frac{1}{2}
}
}\epsilon^{i+j+k}\mathbf{u}_{i}\cdot\nabla^{\Gamma}h_{j+1}h_{k+1}\eta'-\epsilon^{M-\frac{1}{2}}W^{M-\frac{1}{2}}=:\rdivi\label{eq:DIV-rem}
\end{align}
and
\begin{align}
-\Delta\mathbf{v}_{I} & +\nabla p_{I}-\mu_{I}\nabla c_{I}\nonumber \\
 & =\epsilon^{M}\left(-\partial_{\rho}\mathbf{v}_{M+1}\Delta d_{\Gamma}-2\nabla\partial_{\rho}\mathbf{v}_{M+1}^{T}\mathbf{n}+\mathbf{q}_{M}\eta'\rho-\mathbf{u}_{M+1}\eta''\rho\right)\nonumber \\
 & \quad-\frac{1}{\epsilon}\sum_{\substack{i\in I_{M-\frac{1}{2}}^{M+1},
0\leq j\leq M+1\\
i+j\geq M+\frac{1}{2}
}
}\epsilon^{i+j}\mu_{i}\partial_{\rho}c_{j}\mathbf{n}-\sum_{\substack{i\in I_{M-\frac{1}{2}}^{M+1},j,l\in I_{M-\frac{3}{2}}^{M}\\
i+j+l\geq M-\frac{1}{2}
}
}\epsilon^{i+j+l}\partial_{\rho\rho}\mathbf{v}_{i}\nabla^{\Gamma}h_{j+1}\cdot \nabla^{\Gamma}h_{l+1}\nonumber \\
 & \quad-\sum_{\substack{i\in I_{M-\frac{1}{2}}^{M},j\in I_{M-\frac{3}{2}}^{M}\\
i+j\geq M-\frac{1}{2}
}
}\epsilon^{i+j}\partial_{\rho}p_{i}\nabla^{\Gamma}h_{j+1}+\sum_{i\in I_{M-\frac{1}{2}}^{M+1},l\in I_{M-\frac{3}{2}}^{M}}\sum_{\substack{0\leq j\leq M+1\\
\\
i+j+l\geq M-\frac{1}{2}
}
}\epsilon^{i+j+l}\mu_{i}\partial_{\rho}c_{j}\nabla^{\Gamma}h_{l+1}\nonumber \\
 & \quad+\sum_{\substack{i\in I_{M-\frac{1}{2}}^{M+1},j\in I_{M-\frac{3}{2}}^{M}\\
i+j\geq M-\frac{1}{2}
}
}\epsilon^{i+j}\left(\partial_{\rho}\mathbf{v}_{i}\Delta^{\Gamma}h_{j+1}+2\nabla\partial_{\rho}\mathbf{v}_{i}^{T}\nabla^{\Gamma}h_{j+1}\right)\nonumber \\
 & \quad-\frac{1}{\epsilon}\sum_{\substack{i\in I_{M-\frac{1}{2}}^{M+1},j\in I_{M-\frac{3}{2}}^{M}\\
i+j\geq M+\frac{1}{2}
}
}\epsilon^{i+j}\mathbf{u}_{i}\eta''h_{j+1}+\frac{1}{\epsilon}\sum_{\substack{0\leq i\leq M,
j\in I_{M-\frac{3}{2}}^{M}\\
i+j\geq M-\frac{1}{2}
}
}\epsilon^{i+j+1}\mathbf{q}_{i}\eta'h_{j+1}-\sum_{i=M}^{M+1}\epsilon^{i}\Delta\mathbf{v}_{i}\nonumber \\
 & \quad+\epsilon^{M}\nabla p_{M}-\epsilon^{M-\frac{1}{2}}\left(\Delta\mathbf{v}_{M-\frac{1}{2}}-\nabla p_{M-\frac{1}{2}}\right)-\sum_{\substack{0\leq j\leq M, i\in I_{M-\frac{1}{2}}^{M+1}\\
i+j\geq M-\frac{1}{2}
}
}\epsilon^{i+j}\mu_{i}\nabla c_{j}-\epsilon^{M-\frac{3}{2}}\mathbf{V}^{M-\frac{1}{2}}\nonumber \\
 & =:\rsi\label{eq:ST-rem}\qquad \text{in } \Gamma(2\delta;T_{\epsilon}).
\end{align}

\subsubsection{The Outer and Boundary Remainder Terms}

By the outer equations considered in Chapter \ref{constrappr} we
get in $\Omega_{T_{0}}^{+}\cup\Omega_{T_{0}}^{-}$ 
\begin{align}
\partial_{t}c_{O}+\mathbf{v}_{O}\cdot\nabla c_{O}-\Delta\mu_{O} & =\epsilon^{M+\frac{1}{2}}\mathbf{v}_{O,M-\frac{1}{2}}\cdot\nabla c_{O,1}+\sum_{\substack{i\in I_{M-\frac{1}{2}}^{M+1}, 0\leq j\le M+1\\
i+j\geq M+\frac{3}{2}
}
}\epsilon^{i+j}\mathbf{v}_{O,i}\cdot\nabla c_{O,j}\nonumber \\
 & =:\rco\label{eq:CH1-remO}
\end{align}
and
\begin{align}
\epsilon\Delta c_{O}-\epsilon^{-1}f'(c_{O})+\mu_{O} & =\text{\ensuremath{\mathcal{O}}}(\epsilon^{M+1})+\epsilon^{M-\frac{1}{2}}\mu_{O,M-\frac{1}{2}}=:\rhO\label{eq:CH2-remO}
\end{align}
in $L^{\infty}(\Omega_{T_{0}}^{+}\cup\Omega_{T_{0}}^{-})$.
Furthermore, 
\begin{align}
-\Delta\mathbf{v}_{O}+\nabla p_{O}-\mu_{O}\nabla c_{O} & =-\epsilon^{M+\frac{1}{2}}\mu_{O,M-\frac{1}{2}}\nabla c_{O,1}-\sum_{\substack{i\in I_{M-\frac{1}{2}}^{M+1}\\
0\leq j\leq M+1\\
i+j\geq M+\frac{3}{2}
}
}\epsilon^{i+j}\mu_{O,i}\nabla c_{O,j}=:\rso\label{eq:ST-remO}
\end{align}
and 
\begin{equation}
\operatorname{div}\mathbf{v}_{O}=0=:\rdivo.\label{eq:DIV-remO}
\end{equation}

Consider the ordinary differential equations (\ref{eq:stokes-bdry})\textendash (\ref{eq:hilliard-bdry})
satisfied by $\left(c_{k}^{\mathbf{B}},\mu_{k}^{\mathbf{B}},\mathbf{v}_{k}^{\mathbf{B}},p_{k-1}^{\mathbf{B}}\right)$
evaluated at $z(x,t)=\frac{d_{\mathbf{B}}(x,t)}{\epsilon}$
for $(x,t)\in\overline{\partial_{T_{0}}\Omega(\delta)}$
and $\epsilon\in\left(0,\epsilon_{1}\right)$ and the outer equations
as discussed in (\ref{eq:hmuv0,5sys})\textendash (\ref{eq:hmuv0,5sys-2})
for $\left(\mu_{M-\frac{1}{2}}^{\mathbf{B}},\mathbf{v}_{M-\frac{1}{2}}^{\mathbf{B}},p_{M-\frac{1}{2}}^{\mathbf{B}}\right)$.
Then 
\begin{align}
\partial_{t}c_{\mathbf{B}}+\mathbf{v}_{\mathbf{B}}\cdot\nabla c_{\mathbf{B}}-\Delta\mu_{\mathbf{B}} & =\text{\ensuremath{\mathcal{O}}}(\epsilon^{M})+\epsilon^{M+\frac{1}{2}}\sum_{\substack{1\leq j\le M+1}
}\epsilon^{j-1}\mathbf{v}_{M-\frac{1}{2}}^{-}\cdot\nabla c_{j}^{\mathbf{B}}\nonumber \\
 & \quad+\sum_{2\leq j\leq M+1}\epsilon^{M+j-\frac{3}{2}}\mathbf{v}_{M-\frac{1}{2}}^{-}\cdot\nabla d_{\mathbf{B}}\partial_{z}c_{j}^{\mathbf{B}}=:\rcb,\label{eq:CH1-remB}
\end{align}
as a consequence of $\partial_{z}c_{0}^{\mathbf{B}}=\partial_{z}c_{1}^{\mathbf{B}}=0$,
see Corollary \ref{cor:bdrycond}. Moreover, 
\begin{align}
\epsilon\Delta c_{\mathbf{B}}-\epsilon^{-1}f'\left(c_{\mathbf{B}}\right)+\mu_{\mathbf{B}} & =\mathcal{O}(\epsilon^{M+1})+\epsilon^{M-\frac{1}{2}}\mu_{M-\frac{1}{2}}^{-}=:\rhb,\label{eq:CH2-remB}\\
-\Delta\mathbf{v}_{\mathbf{B}}+\nabla p_{\mathbf{B}}-\mu_{\mathbf{B}}\nabla c_{\mathbf{B}} & =\text{\ensuremath{\mathcal{O}}}(\epsilon^{M})-\epsilon^{M+\frac{1}{2}}\mu_{M-\frac{1}{2}}^{-}\sum_{\substack{1\leq j\leq M+1}
}\epsilon^{j-1}\nabla c_{j}^{\mathbf{B}}\nonumber \\
 & \quad-\epsilon^{M+\frac{1}{2}}\mu_{M-\frac{1}{2}}^{-}\sum_{2\leq j\leq M+1}\epsilon^{j-2}\nabla d_{\mathbf{B}}\partial_{z}c_{j}^{\mathbf{B}}=:\rsb\label{eq:ST-remB}
\end{align}
in $L^{\infty}(\Omega_{T_{0}}^{+}\cup\Omega_{T_{0}}^{-})$ and 
\begin{align}
\operatorname{div}\mathbf{v}_{\mathbf{B}} & =\epsilon^{M+1}\left(\operatorname{div}\mathbf{v}_{M+1}^{\mathbf{B}}-\operatorname{div}\mathbf{v}_{M+1}^{\mathbf{B}}|_{z=0}\right)=:\rdivb\label{eq:DIV-remB}
\end{align}
in $\overline{\partial_{T_{0}}\Omega(\delta)}$. 
Moreover, 
\begin{alignat}{2}
  \mu_{\mathbf{B}} & =0,\label{eq:Rem-Dir-Mu}
  \qquad
  c_{\mathbf{B}}  =-1, &\quad& \text{on }\partial\Omega,
  \\
\left(-2D_{s}\mathbf{v}_{\mathbf{B}}+p_{\mathbf{B}}\mathbf{I}\right)\cdot\mathbf{n}_{\partial\Omega} & =\alpha_{0}\mathbf{v}_{\mathbf{B}}&\quad& \text{on }\partial\Omega.\label{eq:Rem-Navier-v}
\end{alignat}
\begin{rem}
\label{rem:Remark-bdry-remainder}We introduce the notation 
\begin{equation}
\trhb:=\rhb-\epsilon^{M-\frac{1}{2}}\mu_{O,M-\frac{1}{2}}\label{eq:repstild}
\end{equation}
in $\overline{\partial_{T_{0}}\Omega(\delta)}$ for later
use. Note that $\trhb\in\mathcal{O}(\epsilon^{M+1})$ in
$L^{\infty}(\partial_{T_{0}}\Omega(\delta))$.
\end{rem}

\subsection{First Estimates\label{sec:First-Estimates}}

In order to streamline the results, we define
\begin{align*}
  \mathcal{T}_{G} & :=\bigcup_{i\in\left\{ 0,\ldots,M+1\right\} }\left\{ c_{i},\mu_{i},l_{i}\eta,j_{i}\eta,\mathbf{v}_{i},\mathbf{u}_{i}\eta,\mathbf{q}_{i}\eta\right\},
  \\\mathcal{T}_{h}&:=\bigcup_{i,j\in I_{M-\frac{1}{2}}^{M+1}\backslash\left\{ 0\right\} }\left\{ h_{j},\nabla^{\Gamma}h_{j},\Delta^{\Gamma}h_{j},\partial_{t}^{\Gamma}h_{j},\nabla^{\Gamma}h_{j}\cdot\nabla^{\Gamma}h_{i}\right\} .
\end{align*}
 The following lemma will yield estimates for almost every term in
(\ref{eq:CH1-rem}), except for $B^{M-\frac{1}{2}}$, which is treated
in Lemma \ref{beste}.
\begin{lem}[Estimates for $\rci$ and $\rsi$]~
\label{rech1}\label{rest}\\Let Assumption~\ref{assu:remainder} hold,
let $\varphi\in L^{\infty}(0,T_{0};H^{1}(\Gamma_{t}(2\delta)))$
and let $\mathbf{z}\in L^{2}(0,T_{\epsilon};H^{1}(\Omega)^{2})$.
Then there is some constant $C(K)>0$ such that for all
$\epsilon\in (0,\epsilon_{1})$ 
\begin{align}
\big\Vert \big(\rci-\epsilon^{M-\frac{3}{2}}B^{M-\frac{1}{2}}\big)\varphi\big\Vert _{L^{1}(\Gamma(2\delta;T_{\epsilon}))}\leq C(K)\epsilon^{M}T_{\epsilon}^{\frac{1}{2}}\left\Vert \varphi\right\Vert _{L^{\infty}(0,T_{\epsilon};H^{1}(\Gamma_{t}(2\delta)))},\label{eq:rech1}\\
\big\Vert \big(\rsi+\epsilon^{M-\frac{3}{2}}\mathbf{V}^{M-\frac{1}{2}}\big)\cdot\mathbf{z}\big\Vert _{L^{1}(\Gamma(2\delta;T_{\epsilon}))}\leq C(K)\epsilon^{M}\left\Vert \mathbf{z}\right\Vert _{L^{2}(0,T_{\epsilon};H^{1}(\Omega))}.\label{eq:rest}
\end{align}
\end{lem}
\begin{proof}
The proof makes heavy use of the fact that (\ref{eq:heps0,5}) and
(\ref{eq:muv0,5est}) hold under Assumption~\ref{assu:remainder}.
We first show the inequality for the estimate involving $\rci$ in
multiple steps, estimating the terms separately:

\emph{Step 1:} Due to the matching conditions (\ref{eq:matchcon}) and the
definition of $\eta$, all $f\in\mathcal{T}_{G}$ satisfy $D_{\rho}^{l}D_{x}^{k}f\in\mathcal{R}_{\alpha}$
for $l\in\left\{ 1,2\right\} $, $k\in\left\{ 0,1\right\} $ and some
$\alpha>0$. Now let $g\in\mathcal{T}_{h}$. Since $S\colon \Gamma(2\delta)\rightarrow\mathbb{T}^{1}$
(as defined in (\ref{eq:Sfett})) and its derivatives are bounded
in $\Gamma(2\delta)$ we have 
\[
\left|g(x,t)\right|\leq C\left|a(S(x,t),t)\right|
\]
for some function $a\colon\mathbb{T}^{1}\times[0,T_{0}]\rightarrow\mathbb{R}$,
where $a$ is given by a suitable derivative of the corresponding
$h_{i},$ $i\in I_{M-\frac{1}{2}}^{M+1}\backslash\left\{ 0\right\} $,
or $h_{i}$ itself.

Thus we may use \cite[Corollary 2.7]{nsac} to get 
\begin{align*}
\int_{\Gamma\left(2\delta;T_{\epsilon}\right)}\big|D_{\rho}^{l}D_{x}^{k}f\cdot g\varphi\big|\d(x,t) & \leq C\epsilon T_{\epsilon}^{\frac{1}{2}}\left\Vert a\right\Vert _{L^{2}((0,T_{\epsilon})\times\mathbb{T}^{1})}\left\Vert \varphi\right\Vert _{L^{\infty}(0,T_{\epsilon};H^{1}(\Gamma_{t}(2\delta)))}.
\end{align*}
Now if $g$ corresponds to $h_{l}$ or its derivatives for $l\in\left\{ 1,\ldots,M+1\right\} $,
then $a$ may be  estimated in $L^{\infty}\left(\left(0,T_{0}\right)\times\mathbb{T}^{1}\right)$ uniformly in $\epsilon$.
In case $g$ corresponds to $h_{M-\frac{1}{2}}^{\epsilon}$ or its
derivatives, we use 
\[
\big\Vert \big(h_{M-\frac{1}{2}}^{\epsilon},\partial_{s}h_{M-\frac{1}{2}}^{\epsilon},\partial_{s}^{2}h_{M-\frac{1}{2}}^{\epsilon},\partial_{t}h_{M-\frac{1}{2}}^{\epsilon},\big(\partial_{s}h_{M-\frac{1}{2}}^{\epsilon}\big)^{2}\big)\big\Vert _{L^{2}((0,T_{\epsilon})\times\mathbb{T}^{1})}\leq C\big\Vert h_{M-\frac{1}{2}}^{\epsilon}\big\Vert _{X_{T_{\epsilon}}}
\]
together with (\ref{eq:heps0,5}). If $g\in L^{\infty}(\Gamma(2\delta;T_{0}))$
similar estimates follow with $a\equiv1$.

\emph{Step 2:} Concerning the terms involving $l_{M-\frac{1}{2}}^{\epsilon}$:
Since $X_{T_{\epsilon}}\hookrightarrow C^{0}([0,T_{\epsilon}]\times\mathbb{T}^{1})$
due to Proposition \ref{embedding}.2, we get by Lemma \ref{Linfeig}.2
\[
\int_{\Gamma(2\delta;T_{\epsilon})}\big|l_{M-\frac{1}{2}}^{\epsilon}\eta''h_{i}\varphi\big|\d(x,t)\leq C(K)\epsilon^{\frac{1}{2}}T_{\epsilon}^{\frac{1}{2}}\big\Vert l_{M-\frac{1}{2}}^{\epsilon}\big\Vert _{L^{2}(\Gamma(2\delta;T_{\epsilon}))}\Vert \varphi\Vert _{L^{\infty}(0,T_{\epsilon};H^{1}(\Gamma_{t}(2\delta)))}.
\]
Here we also used $H^{1}(\Gamma_{t}(2\delta))\hookrightarrow L^{2,\infty}(\Gamma_{t}(2\delta))$
due to Lemma \ref{L4inf} and again (\ref{eq:heps0,5}). Considering $l_{M-\frac{1}{2}}^{\epsilon}$
as given in (\ref{eq:lcon0,5}), we note that its numerator vanishes
on $\Gamma$ due to (\ref{eq:hmuv0,5sys-4}). Thus, the mean value
theorem implies for a function $\gamma\colon(-2\delta,2\delta)\rightarrow(-2\delta,2\delta)$
\begin{align}
\big\Vert l_{M-\frac{1}{2}}^{\epsilon}\big\Vert _{L^{2}(\Gamma(2\delta;T_{\epsilon}))}^{2} & \leq C\int_{0}^{T_{\epsilon}}\int_{\mathbb{T}^{1}}\int_{-2\delta}^{2\delta}\Big|\partial_{\mathbf{n}}\big(\big[\mu_{M-\frac{1}{2}}^{\epsilon}\big]+l_{0}h_{M-\frac{1}{2}}^{\epsilon}\big)(X(\gamma(r),s,t))\Big|^{2}\d r\d s\d t\nonumber \\
 & \leq C_{1}\int_{0}^{T_{\epsilon}}\int_{\mathbb{T}^{1}}\sup_{r\in\left(-2\delta,2\delta\right)}\big|\big[\partial_{\mathbf{n}}\mu_{M-\frac{1}{2}}^{\epsilon}\big](X(r,s,t))\big|^{2}\d s\d t+\!C_{2}\nonumber \\
 & \leq C_{1}\Big(\big\Vert \mu_{M-\frac{1}{2}}^{+,\epsilon}\big\Vert _{L^{2}(0,T_{\epsilon};H^{2}(\Omega^{+}(t)))}^{2}\!+\big\Vert \mu_{M-\frac{1}{2}}^{-,\epsilon}\big\Vert _{L^{2}(0,T_{\epsilon};H^{2}(\Omega^{-}(t)))}^{2}\Big)+\!C_{2}\label{eq:lepspeps-1}
\end{align}
Now (\ref{eq:muv0,5est}) implies the
desired estimate.

\emph{Step 3:} Concerning the terms involving $\mathbf{w}_{1}^{\epsilon}|_{\Gamma}$:
As $\partial_{\rho}c_{i}\in\mathcal{R}_{\alpha}$ for $i\in\left\{ 1,\ldots,M+1\right\} $
we may again use \cite[Corollary 2.7]{nsac} to get 
\begin{align*}
\int_{\Gamma(2\delta;T_{\epsilon})}\left|\wei|_{\Gamma}\cdot\mathbf{n}\partial_{\rho}c_{i}\varphi\right|\d(x,t) & \leq C\epsilon T_{\epsilon}^{\frac{1}{2}}\left\Vert \wei\right\Vert _{L^{2}\left(0,T;L^{2}\left(\Gamma_{t}\right)\right)}\left\Vert \varphi\right\Vert _{L^{\infty}\left(0,T_{\epsilon};H^{1}\left(\Gamma_{t}(2\delta)\right)\right)}.
\end{align*}
Since $\wei=\frac{\twe}{\epsilon^{M-\frac{1}{2}}}$ (cf.\ Theorem
\ref{hM-1}.1), we get due to Lemma \ref{Wichtig} and the continuity
of the trace operator 
\[
\int_{\Gamma(2\delta;T_{\epsilon})}\left|\wei|_{\Gamma}\cdot\mathbf{n}\partial_{\rho}c_{i}\varphi\right|\d(x,t)\leq C(K)\epsilon T_{\epsilon}^{\frac{1}{2}}\left\Vert \varphi\right\Vert _{L^{\infty}(0,T_{\epsilon};H^{1}(\Gamma_{t}(2\delta)))}.
\]
 Moreover, we get
\[
\int_{\Gamma(2\delta;T_{\epsilon})}\left|\wei|_{\Gamma}\partial_{\rho}c_{l}\nabla^{\Gamma}h_{j}\varphi\right|\d(x,t)\leq C(K)\epsilon T_{\epsilon}^{\frac{1}{2}}\left\Vert \varphi\right\Vert _{L^{\infty}(0,T_{\epsilon};H^{1}(\Gamma_{t}(2\delta)))}
\]
by similar arguments as above. Finally, we have $\nabla c_{0}=0$ (as
$c_{0}(\rho,x,t)=\theta_{0}(\rho)$) which immediately
shows the wanted estimate for $\mathbf{w}_{1}^{\epsilon}|_{\Gamma}\cdot\nabla c_{i}$,
as $\nabla c_{i}\in L^{\infty}(\mathbb{R}\times\Gamma(2\delta))$
for all $i\in\left\{ 1,\ldots,M+1\right\} $. 

\emph{Step 4:} Concerning the terms involving $\mathbf{v}_{i}$: Using the
explicit form of $\mathbf{v}_{M-\frac{1}{2}}^{\epsilon}$ as given
in (\ref{eq:veps0,5}) together with Lemma~\ref{Linfeig}.2 and (\ref{eq:heps0,5})
we get 
\begin{align*}
  &\int_{\Gamma(2\delta;T_{\epsilon})}\big|\partial_{\rho}c_{j}\mathbf{v}_{M-\frac{1}{2}}^{\epsilon}\cdot(\mathbf{n}-\nabla^{\Gamma}h_{k})\varphi\big|\d(x,t) \\
  &\quad \leq CT_{\epsilon}^{\frac{1}{2}}\epsilon\Big\Vert \big|\mathbf{v}_{M-\frac{1}{2}}^{+,\epsilon}\big|+\big|\mathbf{v}_{M-\frac{1}{2}}^{-,\epsilon}\big|\Big\Vert _{L^{2}(0,T_{\epsilon};L^{2,\infty}(\Gamma_{t}(2\delta)))}\cdot\left\Vert \varphi\right\Vert _{L^{\infty}\left(0,T_{\epsilon};H^{1}\left(\Gamma_{t}(2\delta)\right)\right)}.
\end{align*}
By $H^{1}(\Gamma_{t}(2\delta))\hookrightarrow L^{2,\infty}(\Gamma_{t}(2\delta))$,
(\ref{eq:muv0,5est}) and the continuity of the extension operator
we get the desired estimate.

\emph{Step 5:} Concerning the terms involving $\mu_{M-\frac{1}{2}}^{\epsilon}$:
We use the explicit structure of $\mu_{M-\frac{1}{2}}^{\epsilon}$
as given in (\ref{eq:mueps0,5}) and estimate 
\begin{align}
\int_{\Gamma(2\delta;T_{\epsilon})} & \big|\nabla\partial_{\rho}\mu_{M-\frac{1}{2}}^{\epsilon}\cdot\nabla^{\Gamma}h_{i}\varphi\big|\d(x,t)\nonumber \\
 & \leq C\int_{0}^{T_{\epsilon}}\int_{\mathbb{T}^{1}}|\partial_{s}h_{i}(s,t)|\int_{-2\delta}^{2\delta}\big|\nabla\partial_{\rho}\mu_{M-\frac{1}{2}}^{\epsilon}\big(\tfrac{r}{\epsilon}-h_{A}^{\epsilon}(s,t),X(r,s,t)\big)\varphi\circ X\big|\d r\d s\d t\nonumber \\
 & \leq C\epsilon\int_{0}^{T_{\epsilon}}\int_{\mathbb{T}^{1}}|\partial_{s}h_{i}(s,t|\sup_{r\in (-2\delta,2\delta)}\big|\big[\nabla\mu_{M-\frac{1}{2}}^{\epsilon}\big](X(r,s,t))\varphi\circ X\big|\d s\d t\int_{\mathbb{R}}\left|\eta'(\rho)\right|\d\rho\nonumber \\
 & \leq C\epsilon\int_{0}^{T_{\epsilon}}\big\Vert \big[\nabla\mu_{M-\frac{1}{2}}^{\epsilon}\big]\big\Vert _{L^{4,\infty}(\Gamma_{t}(2\delta))}\left\Vert \varphi\right\Vert _{L^{4,\infty}(\Gamma_{t}(2\delta))}\d t\Vert \partial_{s}h_{i}\Vert _{L^{\infty}(0,T_{\epsilon};L^{2}(\mathbb{T}^1))}\nonumber \\
 & \leq C(K)\epsilon T_{\epsilon}^{\frac{1}{2}}\Vert \varphi\Vert _{L^{\infty}(0,T_{\epsilon};H^{1}(\Gamma_{t}(2\delta)))}.\label{eq:similar}
\end{align}
Here we used $\sup_{(x,t)\in\Gamma (2\delta;T_{\epsilon})}\left|\nabla^{\Gamma}h_{j}(x,t)\right|\leq C(K)$
for $j\in I_{M-\frac{1}{2}}^{M+1}\backslash\left\{ 0\right\} $. The same procedure yields the desired estimate
for $\partial_{\rho\rho}\mu_{M-\frac{1}{2}}^{\epsilon}\nabla h_{i}\cdot\nabla h_{j}$
and $\partial_{\rho}\mu_{M-\frac{1}{2}}^{\epsilon}\Delta h_{i}$,
$i,j\in I_{M-\frac{1}{2}}^{M+1}\backslash\left\{ 0\right\} $. For
the latter, it is necessary to use $X_{T_{\epsilon}}\hookrightarrow C^{0}\left(\left[0,T_{\epsilon}\right];H^{2}(\mathbb{T}^1)\right)$.

To treat $\Delta\mu_{M-\frac{1}{2}}^{\epsilon}$ we set
\begin{equation}
C_{K}:=\sup_{\epsilon\in (0,\epsilon_{1})}\sup_{(s,t)\in\mathbb{T}^{1}\times [0,T_{\epsilon}]}|h_{A}^{\epsilon}(s,t)|\label{eq:CK}
\end{equation}
 which is well defined due (\ref{eq:haepglm}). As $\Delta\mu_{M-\frac{1}{2}}^{\epsilon}=\eta\Delta\mu_{M-\frac{1}{2}}^{+,\epsilon}+ (1-\eta)\Delta\mu_{M-\frac{1}{2}}^{-,\epsilon}$
and $\Delta\mu_{M-\frac{1}{2}}^{\pm,\epsilon}=0$ in $\Omega_{T_{\epsilon}}^{\pm}$
by (\ref{eq:hmuv0,5sys}), we find
\begin{align}
\int_{\Omega_{T_{\epsilon}}^{+}\cap\Gamma (2\delta;T_{\epsilon})} & \big|\Delta\mu_{M-\frac{1}{2}}^{\epsilon}\varphi\big|\d(x,t)\nonumber \\
 & \leq C\int_{0}^{T_{\epsilon}}\int_{\mathbb{T}^{1}}\Vert \varphi(.,s,t)\Vert _{L^{\infty}(-2\delta,2\delta)}\int_{0}^{2\delta}\big|\Delta\mu_{M-\frac{1}{2}}^{-,\epsilon}(1-\eta(\rho(r,s,t)))\big|\d r\d s\d t\nonumber \\
 & \leq CT_{\epsilon}^{\frac{1}{2}}\Vert \varphi\Vert_{L^{\infty}(0,T_{\epsilon};H^{1}(\Omega))}\big\Vert \Delta\mu_{M-\frac{1}{2}}^{-,\epsilon}\big\Vert_{L^{2}(\Gamma(2\delta;T_{\epsilon}))}\epsilon^{\frac{1}{2}}\Vert 1-\eta\Vert _{L^{2}(-C_{K},\infty)}\nonumber \\
 & \leq C(K)T_{\epsilon}^{\frac{1}{2}}\epsilon^{\frac{1}{2}}\Vert \varphi\Vert _{L^{\infty} (0,T_{\epsilon};H^{1}(\Omega))},\label{eq:irgendwashald}
\end{align}
where we used $\eta-1\equiv0$ in $(1,\infty)$, the continuity
of the extension operator for $\mu_{M-\frac{1}{2}}^{\pm,\epsilon}$
and (\ref{eq:muv0,5est}) in the last line. A similar estimate holds
on $\Omega_{T_{\epsilon}}^{-}\cap\Gamma(2\delta;T_{\epsilon})$.

Now (\ref{eq:rech1}) follows since all not considered terms may
be treated by simply using Hölder's inequality and $L^{\infty}$ bounds
(for $\mathbf{v}_{i}\cdot\nabla c_{j}$ note $\nabla c_{0}=0$ and
apply (\ref{eq:muv0,5est}) for the fractional order term). 

Regarding (\ref{eq:rest}), the same ideas may be applied with the
sole difference that $\mathbf{z}$ is only $L^{2}$ in time and as
a consequence we do not get the term $T_{\epsilon}^{\frac{1}{2}}$
in the estimates. Due to the many similarities, we only 
show three estimates in detail:

 Concerning $\partial_{\rho}p_{M-\frac{1}{2}}^{\epsilon}\nabla^{\Gamma}h_{j}\cdot\mathbf{z}$
for $j\in I_{M-\frac{1}{2}}^{M+1}\backslash\left\{ 0\right\} $, we
use the explicit form of $p_{M-\frac{1}{2}}^{\epsilon}$ as given
in (\ref{eq:peps0,5}) to estimate
\begin{align*}
\big\Vert \partial_{\rho}p_{M-\frac{1}{2}}^{\epsilon}\nabla^{\Gamma}h_{j}\mathbf{z}\big\Vert _{L^{1}(\Gamma (2\delta;T_{\epsilon}))} & \leq C(K)\epsilon\big\Vert \big[p_{M-\frac{1}{2}}^{\epsilon}\big]\mathbf{z}\big\Vert _{L^{1}(0,T_{\epsilon});L^{1,\infty}(\Gamma_{t}(2\delta)))}\Vert \eta'\Vert _{L^{1}(\mathbb{R})}\\
 & \leq C(K)\epsilon \Vert \mathbf{z}\Vert _{L^{2}\left(0,T_{\epsilon};H^{1}(\Omega)\right)},
\end{align*}
where we used Lemma~\ref{Linfeig}.1 in the first inequality and
$H^{1}(\Gamma_{t}(2\delta))\hookrightarrow L^{2,\infty}(\Gamma_{t}(2\delta))$
(cf.\ Lemma~\ref{L4inf}) as well as (\ref{eq:muv0,5est}) together
with the continuity of the extension operator for $p_{M-\frac{1}{2}}^{\pm,\epsilon}$
(cf.\ Remark \ref{extension}) in the last inequality. Here we again
used the notation $\big[p_{M-\frac{1}{2}}^{\epsilon}\big]=p_{M-\frac{1}{2}}^{+,\epsilon}-p_{M-\frac{1}{2}}^{-,\epsilon}$.

Concerning terms involving $\partial_{\rho}\mathbf{v}_{M-\frac{1}{2}}^{\epsilon}$:
Using the explicit form of $\mathbf{v}_{M-\frac{1}{2}}^{\epsilon}$
as given in (\ref{eq:veps0,5}), we exemplarily estimate
\begin{align*}
\int_{\Gamma(2\delta;T_{\epsilon})} & \big|\nabla\partial_{\rho}\mathbf{v}_{M-\frac{1}{2}}^{\epsilon}\cdot\nabla^{\Gamma}h_{i}\mathbf{z}\big|\d(x,t)\\
 & \leq C\epsilon\int_{0}^{T_{\epsilon}}\Vert \partial_{s}h_{i}\Vert _{L^{2}(\mathbb{T}^1)}\big\Vert \big[\nabla\mathbf{v}_{M-\frac{1}{2}}^{\epsilon}\big]\big\Vert _{L^{4,\infty}(\Gamma_{t}(2\delta))}\Vert \mathbf{z}\Vert _{L^{4,\infty}(\Gamma_{t}(2\delta))}\d t\Vert \eta'\Vert _{L^{1}(\mathbb{R})}\\
 & \leq C\epsilon\Vert h_{i}\Vert _{L^{4}(0,T_{\epsilon};H^{2}(\mathbb{T}^1))}\big\Vert \big[\nabla\mathbf{v}_{M-\frac{1}{2}}^{\epsilon}\big]\big\Vert _{L^{4}(0,T_{\epsilon};H^{1}(\Gamma_{t}(2\delta)))}\Vert \mathbf{z}\Vert _{L^{2}(0,T_{\epsilon};H^{1}(\Omega))}\\
 & \leq C(K)\epsilon\Vert \mathbf{z}\Vert_{L^{2}(0,T_{\epsilon};H^{1}(\Omega))}
\end{align*}
for all $i\in I_{M+1}^{M-\frac{1}{2}}\backslash\left\{ 0\right\} $,
where we used $H^{1}(\Gamma_{t}(2\delta))\hookrightarrow L^{4,\infty}(\Gamma_{t}(2\delta))$
in the second inequality and the continuity of the trace operator,
(\ref{eq:muv0,5est}) and $X_{T_{\epsilon}}\hookrightarrow H^{\frac{1}{2}}(0,T_{\epsilon};H^{2}(\mathbb{T}^1))$
in the last inequality. The same procedure can be used to estimate
$\partial_{\rho\rho}\mathbf{v}_{M-\frac{1}{2}}^{\epsilon}\nabla^{\Gamma}h_{i}\cdot\nabla^{\Gamma}h_{j}$
and $\partial_{\rho}\mathbf{v}_{M-\frac{1}{2}}^{\epsilon}\Delta^{\Gamma}h_{i}$
for all $i,j\in I_{M-\frac{1}{2}}^{M+1}\backslash\left\{ 0\right\} $.

Concerning $\Delta\mathbf{v}_{M-\frac{1}{2}}^{\epsilon}-\nabla p_{M-\frac{1}{2}}^{\epsilon}$:
Let $C_{K}$ be given as in (\ref{eq:CK}). Since
\[
\Delta\mathbf{v}_{M-\frac{1}{2}}^{\epsilon}-\nabla p_{M-\frac{1}{2}}^{\epsilon}=\big(\Delta\mathbf{v}_{M-\frac{1}{2}}^{+,\epsilon}-\nabla p_{M-\frac{1}{2}}^{+,\epsilon}\big)\eta+\big(\Delta\mathbf{v}_{M-\frac{1}{2}}^{-,\epsilon}-\nabla p_{M-\frac{1}{2}}^{-,\epsilon}\big)(1-\eta)
\]
and $\Delta\mathbf{v}_{M-\frac{1}{2}}^{\pm,\epsilon}-\nabla p_{M-\frac{1}{2}}^{\pm,\epsilon}=0$
in $\Omega_{T_{\epsilon}}^{\pm}$ by (\ref{eq:hmuv0,5sys-1}), we
can use the same approach as in (\ref{eq:irgendwashald}) together
with (\ref{eq:muv0,5est}) in the last line.
\end{proof}
The following proposition simplifies the estimates of remainder terms in (\ref{eq:CH2-rem}).
\begin{prop}
\label{RZerlprop} Let $R=c^{\epsilon}-c_{A}^{\epsilon}$. It holds
\[
\Vert R\Vert _{L^{1}(\Gamma(\delta;T_{\epsilon}))}\leq C(K)T_{\epsilon}^{\frac{1}{2}}\epsilon^{M}\qquad \text{for all }\epsilon\in(0,\epsilon_{1}).
\]
\end{prop}

\begin{proof}
Due to Assumption \ref{assu:remainder}, we may use \cite[Proposition 3.3]{NSCH1} and find that 
\begin{equation}
R=\epsilon^{-\frac{1}{2}}Z(S(x,t),t)\left(\beta(S(x,t),t)\theta_{0}'(\rho(x,t))+F_{1}^{\mathbf{R}}(x,t)\right)+F_{2}^{\mathbf{R}}(x,t)\label{eq:Rzerlecht}
\end{equation}
for $(x,t)\in\Gamma (\delta;T_{\epsilon})$ with
according estimates for $Z$, $F_{1}^{\mathbf{R}}$, $F_{2}^{\mathbf{R}}$. Using theses estimates we get
\begin{align*}
  &\int_{0}^{T_{\epsilon}}\int_{\Gamma_{t}(\delta)}|R|\d x\d t\\
  & \leq C\int_{0}^{T_{\epsilon}}\int_{\mathbb{T}^{1}}\int_{-\frac{\delta}{\epsilon}-h_{A}^{\epsilon}}^{\frac{\delta}{\epsilon}-h_{A}^{\epsilon}}\epsilon^{\frac{1}{2}}\big|Z(s,t)\big(\beta(s,t)\theta_{0}'(\rho)+F_{1}^{\mathbf{R}}(\rho,s,t)\big)\big|\left|J^{\epsilon}(\rho,s,t)\right|\d\rho\d s\d t\\
 & \quad+CT_{\epsilon}^{\frac{1}{2}}\left\Vert F_{2}^{\mathbf{R}}\right\Vert _{L^{2}(0,T_{\epsilon};L^{2}(\Gamma_{t}(\delta)))}\\
 & \leq CT_{\epsilon}^{\frac{1}{2}}\left(\epsilon^{\frac{1}{2}}\Vert Z\Vert _{L^{2}(0,T_{\epsilon};L^{2}(\mathbb{T}^1))}(1+\epsilon^{\frac{1}{2}})+C(K)\epsilon^{M+\frac{1}{2}}\right) \leq C(K)T_{\epsilon}^{\frac{1}{2}}\epsilon^{M}
\end{align*}
for all $\epsilon\in (0,\epsilon_{1})$. 
\end{proof}
When inspecting the remainder terms (\ref{eq:CH1-rem})\textendash (\ref{eq:ST-rem}), one observes that the terms $A^{M-\frac{1}{2}}$, $B^{M-\frac{1}{2}}$,
$\mathbf{V}^{M-\frac{1}{2}}$ and $W^{M-\frac{1}{2}}$ are multiplied
by a lower power of $\epsilon$ than the rest. Gaining these missing
powers of $\epsilon$ needs delicate work; the main ingredient for
this is that we have intricate structural knowledge of $A^{M-\frac{1}{2}}$
etc. due to Lemma \ref{ABstruc}.
\begin{lem}
\label{beste}Let $\varphi\in L^{\infty}(0,T_{\epsilon};H^{1}(\Omega))$, $\mathbf{z}\in L^{2}(0,T_{\epsilon};H^{1}(\Omega)^{2})$
and $R=c^{\epsilon}-c_{A}^{\epsilon}$. Then there is some $\epsilon_{2}\in (0,\epsilon_{1}]$
such that for all $\epsilon\in (0,\epsilon_{2})$
\begin{align}
\epsilon^{M-\frac{3}{2}}\int_{0}^{T_{\epsilon}}\Big|\,\int_{\Gamma_{t}(\delta)}B^{M-\frac{1}{2}}\varphi\d x\Big|\d t & \leq C(K)\epsilon^{M}\big(T_{\epsilon}^{\frac{1}{2}}+\epsilon^{\frac{1}{2}}\big)\Vert \varphi\Vert _{L^{\infty}(0,T_{\epsilon};H^{1}(\Omega))},\label{eq:besteB}\\
\epsilon^{M-\frac{3}{2}}\int_{0}^{T_{\epsilon}}\Big|\,\int_{\Gamma_{t}(\delta)}\mathbf{V}^{M-\frac{1}{2}}\cdot\mathbf{z}\d x\Big|\d t & \leq C(K)\epsilon^{M}\Vert \mathbf{z}\Vert _{L^{2}(0,T_{\epsilon};H^{1}(\Omega))},\label{eq:besteV}\\
\epsilon^{M-\frac{1}{2}}\int_{0}^{T_{\epsilon}}\Big|\,\int_{\Gamma_{t}(\delta)}A^{M-\frac{1}{2}}R\d x\Big|\d t & \leq C(K)\epsilon^{2M}\big(T_{\epsilon}^{\frac{1}{3}}+\epsilon^{\frac{1}{2}}\big),\label{eq:besteA}\\
\epsilon^{M-\frac{1}{2}}\big\Vert W^{M-\frac{1}{2}}\big\Vert _{L^{2}(0,T_{\epsilon};L^{2}(\Gamma_{t}(2\delta)))} & \leq C(K)\epsilon^{M}.\label{eq:besteW}
\end{align}
\end{lem}

\begin{proof}
For the sake of readability we will write throughout this proof
\[
\left(f\right)^{\Gamma}(\rho,x,t):=f(\rho,x,t)-f(\rho,\operatorname{Pr}_{\Gamma_{t}}(x),t)
\]
 for an arbitrary function $f$ depending on $(\rho,x,t)\in\mathbb{R}\times\Gamma(\delta;T_{\epsilon})$
(and similarly for functions depending only on $(x,t)$).
Moreover, for functions $\psi\colon \Gamma(\delta;T_{\epsilon})\rightarrow\mathbb{R}$
we use the notation $\psi(r,s,t):=\psi(X(r,s,t))$
for $(r,s,t)\in(-\delta,\delta)\times\mathbb{T}^{1}\times[0,T_{\epsilon}]$
and write 
\[
J^{\epsilon}(\rho,s,t):=J(\epsilon(\rho+h_{A}^{\epsilon}(s,t)),s,t)\quad\text{for all }(\rho,s,t)\in I_{\epsilon}^{s,t}\times\mathbb{T}^{1}\times[0,T_0]
\]
with $J(r,s,t):=\text{det}\left(D_{\left(r,s\right)}X(r,s,t)\right)$
for $(r,s,t)\in(-\delta,\delta)\times\mathbb{T}^{1}\times[0,T_{\epsilon}]$
and $$
I_{\epsilon}^{s,t}:=\left(-\tfrac{\delta}{\epsilon}-h_{A}^{\epsilon}(s,t),\tfrac{\delta}{\epsilon}-h_{A}^{\epsilon}(s,t)\right).
$$
To show (\ref{eq:besteB}) we denote $B^{M-\frac{1}{2}}|_{\Gamma}(\rho,x,t):=B^{M-\frac{1}{2}}(\rho,\operatorname{Pr}_{\Gamma_{t}}(x),t)$
and get 
\begin{align*}
  &\int_{0}^{T_{\epsilon}}\Big|\,\int_{\Gamma_{t}(\delta)}B^{M-\frac{1}{2}}\varphi\d x\Big|\d t\\
  & \le\int_{0}^{T_{\epsilon}}\Big|\,\int_{\Gamma_{t}(\delta)}B^{M-\frac{1}{2}}|_{\Gamma}\varphi\d x\Big|\d t+\int_{0}^{T_{\epsilon}}\Big|\,\int_{\Gamma_{t}(\delta)}\big(B^{M-\frac{1}{2}}-B^{M-\frac{1}{2}}|_{\Gamma}\big)\varphi\d x\Big|\d t =:\mathcal{J}_{1}+\mathcal{J}_{2}.
\end{align*}
The fundamental theorem of calculus implies $\varphi(r,s,t)=\varphi(0,s,t)+\int_{0}^{r}\partial_{\mathbf{n}}\varphi (\tilde{r},s,t)\d\tilde{r}$
for $(r,s,t)\in (-\delta,\delta)\times\mathbb{T}^{1}\times [0,T]$
and we write 
\begin{align*}
\mathcal{J}_{1}^{1} & :=\int_{0}^{T_{\epsilon}}\Big|\,\int_{\mathbb{T}^{1}}\int_{-\delta}^{\delta}B^{M-\frac{1}{2}}|_{\Gamma}\varphi(0,s,t)J(r,s,t)\d r\d s\Big|\d t\\
\mathcal{J}_{1}^{2} & :=\int_{0}^{T_{\epsilon}}\Big|\,\int_{\mathbb{T}^{1}}\int_{-\delta}^{\delta}B^{M-\frac{1}{2}}|_{\Gamma}\int_{0}^{r}\partial_{\mathbf{n}}\varphi(\tilde{r},s,t)\d\tilde{r}J(r,s,t)\d r\d s\Big|\d t.
\end{align*}
Concerning $\mathcal{J}_{1}^{1}$ we use the splitting of $B^{M-\frac{1}{2}}$
on $\mathbb{R}\times\Gamma$ as in Lemma \ref{ABstruc}.2 and get
\[
  \mathcal{J}_{1}^{1}\leq\sum_{k=1}^{K_{2}}\int_{0}^{T_{\epsilon}}\int_{\mathbb{T}^{1}}\Big|\mathtt{B}_{k}^{1,\Gamma}(0,s,t)\varphi(0,s,t)\Big|\left|\epsilon\int_{-\frac{\delta}{\epsilon}-h_{A}^{\epsilon}}^{\frac{\delta}{\epsilon}-h_{A}^{\epsilon}}\mathtt{B}_{k}^{2,\Gamma}(\rho)J^{\epsilon}(\rho,s,t)\d\rho\right|\d s\d t.
\]
Since $\sup_{\epsilon\in(0,\epsilon_{1})}\Vert h_{A}^{\epsilon}\Vert _{L^{\infty}((0,T_{\epsilon})\times\mathbb{T}^{1})}<C(K)$
due to (\ref{eq:haepglm}), it holds
\begin{equation}
\Big|\frac{\delta}{\epsilon}-h_{A}^{\epsilon}\Big|\geq\frac{\delta}{\epsilon}-C(K)\geq\frac{\delta}{2\epsilon}\quad \text{ for }\epsilon>0\text{ small enough}.\label{eq:epsklein}
\end{equation}
Moreover, we have 
\begin{equation}
J^{\epsilon}(\rho,s,t)=1+\epsilon\left(\rho+h_{A}^{\epsilon}(s,t)\right)\kappa(s,t)\label{eq:Formel1}
\end{equation}
 by \cite[p.~537, Lemma~4]{chenAC} where $\kappa(s,t)=\kappa (X_{0}(s,t))$
denotes the (principal) curvature of $\Gamma_{t}$ at a point $X_{0}(s,t)=p\in\Gamma_{t}$.
Thus, we may use that $\mathtt{B}_{k}^{2,\Gamma}$ satisfies (\ref{eq:expabB})
and that $H^{1}\left(\Gamma_{t}(\delta)\right)\hookrightarrow L^{2,\infty}\left(\Gamma_{t}(\delta)\right)$
holds for $\varphi$ to get 
\begin{align*}
\mathcal{J}_{1}^{1} & \leq CT_{\epsilon}^{\frac{1}{2}}\sum_{k=1}^{K_{2}}\left\Vert \mathtt{B}_{k}^{1,\Gamma}\right\Vert _{L^{2}\left(0,T_{\epsilon};L^{2}\left(\Gamma_{t}\right)\right)}\left\Vert \varphi\right\Vert _{L^{\infty}\left(0,T_{\epsilon};H^{1}(\Omega)\right)}\epsilon\left(e^{-\frac{\alpha\delta}{\epsilon}}+\epsilon C(K)\right)\\
 & \leq C(K)\epsilon^{2}\left\Vert \varphi\right\Vert _{L^{\infty}\left(0,T_{\epsilon};H^{1}(\Omega)\right)}
\end{align*}
for $\epsilon>0$ small enough. Here we also used the fact that $\big|\mathtt{B}_{k}^{2,\Gamma}(\rho)\big|\leq C_{1}e^{-C_{2}\left|\rho\right|}$
for $\rho\in\mathbb{R}$ , cf.\ Lemma \ref{ABstruc}.2.

To treat $\mathcal{J}_{1}^{2}$ we again use the fact that all terms
of kind $\mathtt{B}_{k}^{2,\Gamma}$ exhibit exponential decay and
thus 
\begin{align*}
\mathcal{J}_{1}^{2} & \leq C\sum_{k=1}^{K_{2}}\int_{0}^{T_{\epsilon}}\int_{\mathbb{T}^{1}}\left\Vert \partial_{\mathbf{n}}\varphi\left(.,s,t\right)\right\Vert _{L^{2}\left(-\delta,\delta\right)}\big|\mathtt{B}_{k}^{1,\Gamma}(0,s,t)\big|\int_{-\delta}^{\delta}|r|^{\frac{1}{2}}\big|\mathtt{B}_{k}^{2,\Gamma}(\rho(r,s,t))\big|\d r\d s\d t\\
 & \leq C\epsilon^{\frac{3}{2}}\sum_{k=1}^{K_{2}}T_{\epsilon}^{\frac{1}{2}}\left\Vert \varphi\right\Vert _{L^{\infty}\left(0,T_{\epsilon};H^{1}\left(\Gamma_{t}(\delta)\right)\right)}\big\Vert \mathtt{B}_{k}^{1,\Gamma}\big\Vert _{L^{2}(0,T_{\epsilon};L^{2}(\Gamma_{t}))} \leq C(K)\epsilon^{\frac{3}{2}}T_{\epsilon}^{\frac{1}{2}}\Vert \varphi\Vert _{L^{\infty}(0,T_{\epsilon};H^{1}(\Omega))}
\end{align*}
where we used (\ref{eq:b1est}) in the last inequality. 

Now we consider $\mathcal{J}_{2}$: Here we use the explicit form
of $B^{M-\frac{1}{2}}$ as given in (\ref{eq:BM-0,5}) and 
estimate the occurring terms separately. First, note that there appears no term
involving $\mathbf{w}_{1}^{\epsilon}|_{\Gamma}$ in $\big(B^{M-\frac{1}{2}}\big)^{\Gamma}$
as it cancels out. In order to estimate the term $\big(\nabla\partial_{\rho}\mu_{M-\frac{1}{2}}^{\epsilon}\cdot\mathbf{n}\big)^{\Gamma}=\eta'\big(\big[\nabla\mu_{M-\frac{1}{2}}^{\epsilon}\big]\big)^{\Gamma}\cdot\mathbf{n}$
(where the equality follows from (\ref{eq:mueps0,5})), we compute
\begin{align}
\int_{\Gamma\left(\delta;T_{\epsilon}\right)} & \Big|\eta'\big(\big[\nabla\mu_{M-\frac{1}{2}}^{\epsilon}\big]\big)^{\Gamma}\cdot\mathbf{n}\varphi\Big|\d x\d t\nonumber \\
 & \leq C\int_{0}^{T_{\epsilon}}\int_{\mathbb{T}^{1}}\int_{-\delta}^{\delta}\Big|\eta'\int_{0}^{r}\partial_{\mathbf{n}}^{2}\big[\mu_{M-\frac{1}{2}}^{\epsilon}\big](\tilde{r},s,t)\d\tilde{r}\varphi\Big|\d r\d s\d t\nonumber \\
 & \leq C\int_{0}^{T_{\epsilon}}\int_{\mathbb{T}^{1}}\Big\Vert \partial_{\mathbf{n}}^{2}\big[\mu_{M-\frac{1}{2}}^{\epsilon}\big]\Big\Vert _{L^{2}(-\delta,\delta)}\Vert \varphi\Vert _{L^{\infty}(-\delta,\delta)}\int_{\frac{-\delta}{\epsilon}-h_{A}^{\epsilon}}^{\frac{\delta}{\epsilon}-h_{A}^{\epsilon}}\Big|\eta'(\rho)\big(\rho+h_{A}^{\epsilon}\big)^{\frac{1}{2}}\Big|\epsilon^{\frac{3}{2}}\d\rho\d s\d t\nonumber \\
 & \leq C(K)T_{\epsilon}^{\frac{1}{2}}\epsilon^{\frac{3}{2}}\Vert \varphi\Vert _{L^{\infty}(0,T_{\epsilon};H^{1}(\Gamma_{t}(\delta)))},\label{eq:Bobs}
\end{align}
where we used (\ref{eq:muv0,5est}).
$\big(\partial_{\rho}c_{0}\mathbf{v}_{M-\frac{1}{2}}^{\epsilon}\big)^{\Gamma}$
and $\big(\partial_{\rho}\mu_{M-\frac{1}{2}}^{\epsilon}\Delta d_{\Gamma}\big)^{\Gamma}$
may be treated in a very similar fashion. For $\big(l_{M-\frac{1}{2}}^{\epsilon}\eta''(\rho+h_{1})\big)^{\Gamma}$
note that by Taylor's theorem,
we get by the definition of $l_{M-\frac{1}{2}}^{\epsilon}$ in (\ref{eq:lcon0,5})
\begin{align}
\Big|\big(l_{M-\frac{1}{2}}^{\epsilon}\big)^{\Gamma}(r,s,t)\Big| & =\left|\int_{0}^{r}\frac{\left(r-\tilde{r}\right)}{r}\Big(\partial_{\mathbf{n}}^{2}\big(\mu_{M-\frac{1}{2}}^{+,\epsilon}-\mu_{M-\frac{1}{2}}^{-,\epsilon}\big)+\partial_{\mathbf{n}}^{2}l_{0}h_{M-\frac{1}{2}}^{\epsilon}\Big)(\tilde{r},s,t)\d\tilde{r}\right|\nonumber \\
 & \leq Cr^{\frac{1}{2}}\Big\Vert \partial_{\mathbf{n}}^{2}\big(\mu_{M-\frac{1}{2}}^{+,\epsilon}-\mu_{M-\frac{1}{2}}^{-,\epsilon}\big)+\partial_{\mathbf{n}}^{2}l_{0}h_{M-\frac{1}{2}}^{\epsilon}\Big\Vert _{L^{2}(-\delta,\delta)}\label{eq:lepspeps}
\end{align}
for $(r,s,t)\in(-\delta,\delta)\times\mathbb{T}^{1}\times(0,T_{\epsilon})$.
This allows for the same strategy to be used as in (\ref{eq:Bobs}).

By Remark \ref{hnotations}, we have 
\begin{align}
\big(\Delta^{\Gamma}h_{M-\frac{1}{2}}^{\epsilon}(x,t)\big)^{\Gamma} & =\Big(\left(\Delta S(x,t)\right)^{\Gamma}\partial_{s}+\big(|\nabla S(x,t)|^{2}\big)^{\Gamma}\partial_{s}^{2}\Big)h_{M-\frac{1}{2}}^{\epsilon}(S(x,t),t)\label{eq:laplacehsplit}
\end{align}
Thus \cite[Corollary 2.7]{nsac},  $\partial_{\rho}\mu_{0}\in\mathcal{R}_{\alpha}$
and (\ref{eq:heps0,5}) imply
\begin{align}
\int_{0}^{T_{\epsilon}}\Big|\,\int_{\Gamma_{t}(\delta)}\big(\partial_{\rho}\mu_{0}\Delta^{\Gamma}h_{M-\frac{1}{2}}^{\epsilon}\big)^{\Gamma}\varphi\d x\Big|\d t & \leq C(K)T_{\epsilon}^{\frac{1}{2}}\epsilon^{2}\big\Vert h_{M-\frac{1}{2}}^{\epsilon}\big\Vert _{L^{2}(0,T_{\epsilon};H^{2}(\mathbb{T}^1))}\left\Vert \varphi\right\Vert _{L^{\infty}\left(0,T_{\epsilon};H^{1}(\Omega)\right)}\nonumber \\
 & \leq C(K)T_{\epsilon}^{\frac{1}{2}}\epsilon^{2}\left\Vert \varphi\right\Vert _{L^{\infty}\left(0,T_{\epsilon};H^{1}(\Omega)\right)}.\label{eq:Tops}
\end{align}
The remaining terms in $\left(B^{M-\frac{1}{2}}\right)^{\Gamma}$
can be estimated in a similar fashion. This proves (\ref{eq:besteB}).

\noindent
\textbf{Proof of (\ref{eq:besteV}):} This can be shown analogously
to (\ref{eq:besteB}) due to Lemma \ref{ABstruc}.3. Here we note
that $\mathbf{z}$ is only in $L^{2}$ in time and thus we may not
expect $T_{\epsilon}^{\frac{1}{2}}$ to appear on the right hand side.
Due to the similarities we shorten the proof: First of all 
\begin{align*}
\int_{0}^{T_{\epsilon}}\Big|\,\int_{\Gamma_{t}(\delta)}\mathbf{V}^{M-\frac{1}{2}}\cdot\mathbf{z}\d x\Big|\d t & \leq\int_{0}^{T_{\epsilon}}\Big|\,\int_{\Gamma_{t}(\delta)}\mathbf{V}^{M-\frac{1}{2}}|_{\Gamma}\cdot\mathbf{z}\d x\Big|\d t+\int_{0}^{T_{\epsilon}}\Big|\,\int_{\Gamma_{t}(\delta)}\big(\mathbf{V}^{M-\frac{1}{2}}\big)^{\Gamma}\cdot\mathbf{z}\d x\Big|\d t.
\end{align*}
Then we use Lemma \ref{ABstruc}.3 and (\ref{eq:expabV}) to obtain e.g.\
\begin{align*}
 & \int_{0}^{T_{\epsilon}}\Big|\,\int_{\mathbb{T}^{1}}\int_{-\delta}^{\delta}\big(\mathbf{V}^{M-\frac{1}{2}}\big)\big|_{\Gamma}\cdot\mathbf{z}(0,s,t)J(r,s,t)\d r\d s\Big|\d t\\
 & \leq\sum_{k=1}^{N_{2}}\int_{0}^{T_{\epsilon}}\int_{\mathbb{T}^{1}}\big|\mathtt{V}_{k}^{1,\Gamma}(0,s,t)\cdot\mathbf{z}(0,s,t)\big|\epsilon\, \text{sup}_{\left(x,\tau\right)\in\Gamma}\Big|\int_{-\frac{\delta}{\epsilon}-h_{A}^{\epsilon}}^{\frac{\delta}{\epsilon}-h_{A}^{\epsilon}}\mathtt{V}_{k}^{2,\Gamma}(\rho,x,\tau)J^{\epsilon}(\rho,x,\tau)\d\rho\Big|\d s\d t\\
 & \leq C(K)\epsilon^{\frac{3}{2}}\Vert \mathbf{z}\Vert _{L^{2}(0,T_{\epsilon};H^{1}(\Omega))}.
\end{align*}
For the other terms, the same argumentation as before can be applied.

\noindent
\textbf{Proof of (\ref{eq:besteA}):} We use the decomposition of
$R$ as in \cite[Proposition 3.3]{NSCH1} and the decomposition of
$A^{M-\frac{1}{2}}$ as in Lemma \ref{ABstruc}.1 to get
\begin{align*}
 & \int_{0}^{T_{\epsilon}}\Big|\,\int_{\Gamma_{t}(\delta)}A^{M-\frac{1}{2}}R\d x\Big|\d t\\
 & \leq C\epsilon^{-\frac{1}{2}}\int_{0}^{T_{\epsilon}}\int_{\mathbb{T}^{1}}|Z(s,t)\beta(s,t)|\Big|\,\int_{-\delta}^{\delta}\theta_{0}'(\rho(r,s,t))A^{M-\frac{1}{2}}J(r,s,t)\d r\Big|\d s\d t\\
 & \quad+C\epsilon^{-\frac{1}{2}}\sum_{k=1}^{L_{1}}\int_{0}^{T_{\epsilon}}\int_{\mathbb{T}^{1}}\left|Z(s,t)\right|\left\Vert \mathtt{A}_{k}^{1}\left(.,s,t\right)\right\Vert _{L^{2}\left(-\delta,\delta\right)}\Big(\int_{\frac{-\delta}{\epsilon}-h_{A}^{\epsilon}}^{\frac{\delta}{\epsilon}-h_{A}^{\epsilon}}\epsilon\left|F_{1}^{\mathbf{R}}\right|^{2}J^{\epsilon}\d\rho\Big)^{\frac{1}{2}}\left\Vert \mathtt{A}_{k}^{2}\right\Vert _{L^{\infty}(\mathbb{R})}\d s\d t\\
 & \quad+C\sum_{k=1}^{L_{1}}T_{\epsilon}^{\frac{1}{3}}\left\Vert F_{2}^{\mathbf{R}}\right\Vert _{L^{2}\left(0,T_{\epsilon};L^{2}\left(\Gamma_{t}(\delta)\right)\right)}\left\Vert \mathtt{A}_{k}^{1}\right\Vert _{L^{6}\left(0,T_{\epsilon};L^{2}\left(\Gamma_{t}(\delta)\right)\right)}\left\Vert \mathtt{A}_{k}^{2}\right\Vert _{L^{\infty}(\mathbb{R})} =:\mathcal{I}_{1}+\mathcal{I}_{2}+\mathcal{I}_{3}.
\end{align*}
Concerning $\mathcal{I}_{1}$, we use $\mathcal{I}_{1}\leq \mathcal{I}_{1}^{1}+ \mathcal{I}_{1}^{2}$, where 
\begin{align*}
\mathcal{I}_{1}^{1} & :=\epsilon^{-\frac{1}{2}}\int_{0}^{T_{\epsilon}}\int_{\mathbb{T}^{1}}|Z(s,t)|\Big|\,\int_{-\delta}^{\delta}\theta_{0}'(r,s,t)A^{M-\frac{1}{2}}\big|_{\Gamma}J(r,s,t)\d r\Big|\d s\d t,\\
\mathcal{I}_{1}^{2} & :=\epsilon^{-\frac{1}{2}}\int_{0}^{T_{\epsilon}}\int_{\mathbb{T}^{1}}|Z(s,t)|\Big|\,\int_{-\delta}^{\delta}\theta_{0}'(r,s,t)\big(A^{M-\frac{1}{2}}\big)^{\Gamma}J(r,s,t)\d r\Big|\d s\d t.
\end{align*}
For $\mathcal{I}_{1}^{1}$ we use the decomposition in Lemma~\ref{ABstruc}.1 on $\mathbb{R}\times\Gamma$ to conclude
\[
\mathcal{I}_{1}^{1}\leq\sum_{k=1}^{L_{2}}\epsilon^{-\frac{1}{2}}\int_{0}^{T_{\epsilon}}\int_{\mathbb{T}^{1}}|Z(s,t)|\big|\mathtt{A}_{k}^{1,\Gamma}(0,s,t)\big|\epsilon\Big|\,\int_{-\frac{\delta}{\epsilon}-h_{A}^{\epsilon}}^{\frac{\delta}{\epsilon}-h_{A}^{\epsilon}}\mathtt{A}_{k}^{2,\Gamma}(\rho)\theta_{0}'(\rho)J^{\epsilon}(\rho,s,t)\d\rho\Big|\d s\d t.
\]
The estimate in (\ref{eq:epsklein}), (\ref{eq:Formel1}), the properties
of $\mathtt{A}_{k}^{2,\Gamma}$ as shown in (\ref{eq:expabA}) and
the exponential decay of $\theta_{0}'$ imply
\[
\mathcal{I}_{1}^{1}\leq C\epsilon^{\frac{1}{2}}\sum_{k=1}^{L_{2}}\Vert Z\Vert _{L^{2}(0,T_{\epsilon};L^{2}(\Gamma_{t}))}\big\Vert \mathtt{A}_{k}^{1,\Gamma}\big\Vert _{L^{2}(0,T_{\epsilon};L^{2}(\Gamma_{t}))}\big(e^{-\frac{\alpha\delta}{\epsilon}}+C(K)\epsilon\big)\leq C(K)\epsilon^{M+1}
\]
for $\epsilon>0$ small enough, where we used the estimate for $Z$
and (\ref{eq:a1est}) for $\mathtt{A}_{k}^{1,\Gamma}$.

In order to estimate $\mathcal{I}_{1}^{2}$, we use the explicit structure
of $A^{M-\frac{1}{2}}$ and first of all analyze the term 
\begin{equation}
  \big(\mu_{M-\frac{1}{2}}^{\epsilon}\big)^{\Gamma}(\rho,x,t)=\big(\mu_{M-\frac{1}{2}}^{+,\epsilon}\big)^{\Gamma}(x,t)\eta(\rho)+\big(\mu_{M-\frac{1}{2}}^{-,\epsilon}\big)^{\Gamma}(x,t)(1-\eta(\rho)),\label{eq:mum0,5+1}
\end{equation}
which appears in $\big(A^{M-\frac{1}{2}}\big)^{\Gamma}$. We estimate
\begin{align*}
 & \epsilon^{-\frac{1}{2}}\int_{0}^{T_{\epsilon}}\int_{\mathbb{T}^{1}}|Z(s,t)|\Big|\,\int_{-\delta}^{\delta}\theta_{0}'(\rho(r,s,t))\big(\mu_{M-\frac{1}{2}}^{+,\epsilon}\big)^{\Gamma}\eta(\rho(r,s,t))J(r,s,t)\d r\Big|\d s\d t\\
 & \leq C\epsilon^{\frac{3}{2}}\int_{0}^{T_{\epsilon}}\int_{\mathbb{T}^{1}}|Z(s,t)|\sup_{r\in(-\delta,\delta)}\big|\partial_{\mathbf{n}}\mu_{M-\frac{1}{2}}^{+,\epsilon}(r,s,t)\big|\int_{\frac{-\delta}{\epsilon}-h_{A}^{\epsilon}}^{\frac{\delta}{\epsilon}-h_{A}^{\epsilon}}|\theta_{0}'(\rho)||\rho+h_{A}^{\epsilon}|\d\rho\d s\d t\\
 & \leq C(K)\epsilon^{M+1}.
\end{align*}
Here we used Lemma \ref{L4inf}, the exponential decay of $\theta_{0}'$, (\ref{eq:muv0,5est})
and the estimate for $Z$ in the third inequality. We may treat the
term $\big(\mu_{M-\frac{1}{2}}^{-,\epsilon}\big)^{\Gamma}(x,t)(1-\eta(\rho(x,t)))$
completely analogously, which finishes the desired estimate for
$\big(\mu_{M-\frac{1}{2}}^{\epsilon}\big)^{\Gamma}$.

Due to (\ref{eq:laplacehsplit}), we will now only consider the term
$\big(|\nabla S(x,t)|^{2}\big)^{\Gamma}\partial_{s}^{2}h_{M-\frac{1}{2}}^{\epsilon}(S(x,t),t)$
in $A^{M-\frac{1}{2}}$, the other occurring terms only involve derivatives
of lower order and can be treated in the same manner. Applying similar
techniques as above, we get 
\begin{align*}
\epsilon^{-\frac{1}{2}}\int_{0}^{T_{\epsilon}}\int_{\mathbb{T}^{1}} & |Z(s,t)|\Big|\,\int_{-\delta}^{\delta}\theta_{0}'(r,s,t)\big(|\nabla S|^{2}\big)^{\Gamma}\partial_{s}^{2}h_{M-\frac{1}{2}}^{\epsilon}J(r,s,t)\d r\Big|\d s\d t\\
 & \leq CT_{\epsilon}^{\frac{1}{3}}\left\Vert Z\right\Vert _{L^{2}\left(0,T_{\epsilon};L^{2}(\mathbb{T}^1)\right)}\big\Vert h_{M-\frac{1}{2}}^{\epsilon}\big\Vert _{L^{6}\left(0,T_{\epsilon};H^{2}(\mathbb{T}^1)\right)}\epsilon^{\frac{3}{2}}.
\end{align*}
Now the estimate for $Z$, (\ref{eq:heps0,5}) and Proposition \ref{embedding}.3  together with $H^{\frac{1}{2}}(0,T_{\epsilon})\hookrightarrow L^{6}(0,T_{\epsilon})$
yield the claim.

Concerning $\mathcal{I}_{2}$ and $\mathcal{I}_{3}$: Using \cite[Proposition 3.3]{NSCH1}, the uniform boundedness
of $\mathtt{A}_{k}^{2}$ in $\mathbb{R}$, and (\ref{eq:a1est}) for
$\mathtt{A}_{k}^{1}$ we get 
\[
\mathcal{I}_{2}\leq C(K)T_{\epsilon}^{\frac{1}{3}}\epsilon^{-\frac{1}{2}}\epsilon^{M-\frac{1}{2}}\epsilon^{\frac{1}{2}}\epsilon=C(K)T_{\epsilon}^{\frac{1}{3}}\epsilon^{M+\frac{1}{2}}.
\]
Noting the estimate for $F_{2}^{\mathbf{R}}$, we also get
\[
\mathcal{I}_{3}\leq C(K)T_{\epsilon}^{\frac{1}{3}}\epsilon^{M+\frac{1}{2}}=C(K)T_{\epsilon}^{\frac{1}{3}}\epsilon^{M+\frac{1}{2}}.
\]
Combining the estimates for $\mathcal{I}_{1},\mathcal{I}_{2}$ and
$\mathcal{I}_{3}$, we obtain (\ref{eq:besteA}).

\noindent
\textbf{Proof of (\ref{eq:besteW}):} We first note that 
\[
\operatorname{div}\mathbf{v}_{M-\frac{1}{2}}^{\epsilon}(\rho,x,t)=\operatorname{div}\mathbf{v}_{M-\frac{1}{2}}^{+,\epsilon}(x,t)\eta(\rho)+\operatorname{div}\mathbf{v}_{M-\frac{1}{2}}^{-,\epsilon}(x,t)\left(1-\eta(\rho)\right)=0
\]
 by the construction (cf.\ Lemma \ref{lem-constrM-0,5}) and the properties
of the extension operator for $\mathbf{v}_{M-\frac{1}{2}}^{\pm,\epsilon}$.
We show the estimate by using the explicit form of $\mathbf{W}^{M-\frac{1}{2}}$:
We estimate
\begin{align*}
\big\Vert \partial_{\rho}\mathbf{v}_{M-\frac{1}{2}}^{\epsilon}\nabla^{\Gamma}h_{1}\big\Vert _{L^{2}(\Gamma (2\delta;T_{\epsilon}))}^{2} & \leq C\epsilon\int_{0}^{T_{\epsilon}}\int_{\mathbb{T}^{1}}\sup_{r\in(-2\delta,2\delta)}\big|\big[\mathbf{v}_{M-\frac{1}{2}}^{\epsilon}\big](r,s,t)\big|^{2}\int_{-\frac{2\delta}{\epsilon}-h_{A}^{\epsilon}}^{\frac{2\delta}{\epsilon}-h_{A}^{\epsilon}}|\eta'(\rho)|^{2}\d\rho\d s\d t\\
 & \leq C(K)\epsilon
\end{align*}
where we used again $H^{1}(\Gamma_{t}(2\delta))\hookrightarrow L^{2,\infty}(\Gamma_{t}(2\delta))$
and (\ref{eq:muv0,5est}). To treat the term with $\mathbf{u}_{M-\frac{1}{2}}^{\epsilon}\cdot\mathbf{n}\eta'(\rho+h_{1})$
term, we employ a similar strategy as we did when estimating $l_{M-\frac{1}{2}}^{\epsilon}$
in (\ref{eq:lepspeps-1}). We use the mean value theorem and the definition
of $\mathbf{u}_{M-\frac{1}{2}}^{\epsilon}$ in (\ref{eq:ucon0,5})
to estimate 
 \begin{align*}
   & \big\Vert \mathbf{u}_{M-\frac{1}{2}}^{\epsilon}\cdot\mathbf{n}\eta'(\rho+h_{1})\big\Vert _{L^{2}(\Gamma(2\delta;T_{\epsilon}))}^{2}\\
   & \leq C\int_{0}^{T_{\epsilon}}\int_{\mathbb{T}^{1}}\int_{-2\delta}^{2\delta}\Big(\partial_{\mathbf{n}}\big(\mathbf{v}_{M-\frac{1}{2}}^{+,\epsilon}-\mathbf{v}_{M-\frac{1}{2}}^{-,\epsilon}+\mathbf{u}_{0}h_{M-\frac{1}{2}}^{\epsilon}\big)(X(\gamma(r)),s,t))\eta'(\rho(r,s,t)+h_{1})\Big)^{2} \d r\d s\d t\\
   & \leq C\epsilon\big\Vert \partial_{\mathbf{n}}\big(\mathbf{v}_{M-\frac{1}{2}}^{+,\epsilon}-\mathbf{v}_{M-\frac{1}{2}}^{-,\epsilon}+\mathbf{u}_{0}h_{M-\frac{1}{2}}^{\epsilon}\big)\big\Vert _{L^{2}(0,T_{\epsilon};L^{2,\infty}(\Gamma_{t}(2\delta)))}^{2}\int_{\mathbb{R}}|\eta'(\rho+1)|^{2}\d\rho \leq C(K)\epsilon,
 \end{align*}
where $\gamma(r)$ is a suitable point in $(0,r)$. These considerations can easily be adapted to
estimate all other terms in $W^{M-\frac{1}{2}}$ accordingly.
\end{proof}

The following proposition is a substitute of the matching conditions
(\ref{eq:matchcon}) for $\mu_{M-\frac{1}{2}}^{\epsilon}$, $\mathbf{v}_{M-\frac{1}{2}}^{\epsilon}$
and $p_{M-\frac{1}{2}}^{\epsilon}$.
\begin{prop}
\label{matchit}There is some $\epsilon_{2}\in\left(0,\epsilon_{1}\right]$
such that for all $\epsilon\in\left(0,\epsilon_{2}\right)$
\begin{align*}
  D_{\rho}^{k}D_{x}^{l}\big(\mu_{M-\frac{1}{2}}^{\epsilon}(\rho,x,t)-\big(\mu_{M-\frac{1}{2}}^{+,\epsilon}\chi_{\overline{\Omega_{T_{0}}^{+}}}+\mu_{M-\frac{1}{2}}^{-,\epsilon}\chi_{\Omega_{T_{0}}^{-}}\big)(x,t)\big)\big|_{\rho=\rho(x,t)} & =0\\
  D_{t}^{m}D_{\rho}^{k}D_{x}^{l}\big(\mathbf{v}_{M-\frac{1}{2}}^{\epsilon}(\rho,x,t)-\big(\mathbf{v}_{M-\frac{1}{2}}^{+,\epsilon}\chi_{\overline{\Omega_{T_{0}}^{+}}}+\mathbf{v}_{M-\frac{1}{2}}^{-,\epsilon}\chi_{\Omega_{T_{0}}^{-}}\big)(x,t)\big)\big|_{\rho=\rho(x,t)} & =0\\
  p_{M-\frac{1}{2}}^{\epsilon}(\rho(x,t),x,t)-\big(p_{M-\frac{1}{2}}^{+,\epsilon}\chi_{\overline{\Omega_{T_{0}}^{+}}}+p_{M-\frac{1}{2}}^{-,\epsilon}\chi_{\Omega_{T_{0}}^{-}}\big)(x,t) & =0
\end{align*}
for all $(x,t)\in\Gamma(2\delta;T_{\epsilon})\backslash\Gamma(\delta;T_{\epsilon})$
and $m,k,l\geq0$.
\end{prop}
\begin{proof}
This is a direct consequence of $\frac{d_{\Gamma}(x,t)}{\epsilon}-h_{A}^{\epsilon}(x,t)\geq1$
for $(x,t)\in\Gamma (2\delta;T_{\epsilon})\backslash\Gamma(\delta;T_{\epsilon})$
and $\epsilon>0$ small enough together with $\eta\equiv1$ in $(1,\infty)$
and $\eta\equiv0$ in $(-\infty,-1)$.
\end{proof}
The next corollary is a consequence of Proposition \ref{matchit} and the
matching conditions for the integer orders.
\begin{cor}
  \label{cor:matching}There is some $\epsilon_{2}\in(0,\epsilon_{1}]$
  such that for all $\epsilon\in(0,\epsilon_{2})$ 
\begin{align*}
  \Vert D_{x}^{l}(\mu_{I}-\mu_{O})\Vert _{L^{\infty}(\Gamma(2\delta;T_{\epsilon})\backslash\Gamma(\delta;T_{\epsilon}))}+\Vert D_{x}^{l}(\mu_{O}-\mu_{\mathbf{B}})\Vert _{L^{\infty}(\partial_{T_{\epsilon}}\Omega(\delta)\backslash\partial_{T_{\epsilon}}\Omega(\frac{\delta}{2}))} & \leq C(K)e^{-\frac{\tilde{C}}{\epsilon}},\\
  \Vert D_{x}^{l}(c_{I}-c_{O})\Vert_{L^{\infty}(\Gamma(2\delta;T_{\epsilon})\backslash\Gamma(\delta;T_{\epsilon}))}+\Vert D_{x}^{l}(c_{O}-c_{\mathbf{B}})\Vert _{L^{\infty}(\partial_{T_{\epsilon}}\Omega(\delta)\backslash\partial_{T_{\epsilon}}\Omega(\frac{\delta}{2}))} & \leq C(K)e^{-\frac{\tilde{C}}{\epsilon}},\\
  \Vert D_{x}^{l}(\mathbf{v}_{I}-\mathbf{v}_{O})\Vert _{L^{\infty}(\Gamma(2\delta;T_{\epsilon})\backslash\Gamma(\delta;T_{\epsilon}))}+\Vert D_{x}^{l}(\mathbf{v}_{O}-\mathbf{v}_{\mathbf{B}})\Vert _{L^{\infty}(\partial_{T_{\epsilon}}\Omega(\delta)\backslash\partial_{T_{\epsilon}}\Omega(\frac{\delta}{2}))} & \leq C(K)\epsilon^{M+1},\\
  \Vert p_{I}-p_{O}\Vert _{L^{\infty}(\Gamma(2\delta;T_{\epsilon})\backslash\Gamma(\delta;T_{\epsilon}))}+\Vert p_{O}-p_{\mathbf{B}}\Vert _{L^{\infty}(\partial_{T_{\epsilon}}\Omega(\delta)\backslash\partial_{T_{\epsilon}}\Omega(\frac{\delta}{2}))} & \leq C(K)e^{-\frac{\tilde{C}}{\epsilon}}
\end{align*}
for $l\in\{ 0,1\} $ and constants $C(K),\tilde{C}>0$.
\end{cor}
\begin{proof}
This follows directly from (\ref{eq:matchcon}), (\ref{eq:matchcon-bdry}),
Proposition \ref{matchit} and the fact that $\frac{d_{\Gamma}(x,t)}{\epsilon}-h_{A}^{\epsilon}(x,t)\geq\frac{\delta}{2\epsilon}$
for $(x,t)\in\Gamma(2\delta;T_{\epsilon})\backslash\Gamma(\delta;T_{\epsilon})$
and for $\epsilon>0$ small enough. Note in particular $\mu_{M-\frac{1}{2}}^{\mathbf{B}}=\mu_{M-\frac{1}{2}}^{-,\epsilon}$
as defined in (\ref{eq:M0,5bdrydef}), which also holds for the other
fractional terms, and consider 
\begin{align*}
\epsilon^{M+1} & \Vert D_{x}^{l}(\mathbf{v}_{O,M+1}-\mathbf{v}_{\mathbf{B},M+1})\Vert _{L^{\infty}(\partial_{T_{\epsilon}}\Omega(\delta)\backslash\partial_{T_{\epsilon}}\Omega(\frac{\delta}{2}))}
 Ce^{-\alpha\frac{\delta}{2\epsilon}}+\epsilon^{M+1}\Vert \mathbf{v}_{M+1}^{\mathbf{B}}(0,.)\Vert _{L^{\infty}(\partial_{T_{0}}\Omega(\delta))},
\end{align*}
which accounts for the special case.
\end{proof}

\subsection{Estimating the Remainder}

The following results are at the same time proofs for the estimates
in Theorem \ref{thm:Main-Apprx-Structure}.
\begin{thm}[Remainder Terms]
  \label{remainder}~\\Let Assumption \ref{assu:remainder} hold and let
for $\epsilon\in (0,\epsilon_{0})$ the functions $c_{A}^{\epsilon}$,
$\mu_{A}^{\epsilon}$, $\mathbf{v}_{A}^{\epsilon}$, $p_{A}^{\epsilon}$,
$h_{A}^{\epsilon}$ be defined as in Definition \ref{def:apprxsol}
and $\rs$, $\rdiv$, $\rc$, $\rh$ be given as in \eqref{eq:Stokesapp}\textendash \eqref{eq:Hilliardapp},
for $\mathbf{w}_{1}^{\epsilon}:=\frac{1}{\epsilon^{M-\frac{1}{2}}}\twe$.
Here $\twe$ is the weak solution to \eqref{eq:w1}\textendash \eqref{eq:w13}
with $H=\big(h_{M-\frac{1}{2}}^{\epsilon}\big)_{\epsilon\in (0,\epsilon_{0})}$.
Moreover, let $\varphi\in L^{\infty}(0,T_{\epsilon};H^{1}(\Omega))$
and $R=c^{\epsilon}-c_{A}^{\epsilon}$. Then there is some $\epsilon_{2}\in (0,\epsilon_{1}]$
and a constant $C(K)>0$ such that for all $\epsilon\in (0,\epsilon_{2})$
\begin{align}
\int_{0}^{T_{\epsilon}}\Big|\int_{\Omega}\rc\varphi\d x\Big|\d t & \leq C(K)\big(T_{\epsilon}^{\frac{1}{2}}+\epsilon^{\frac{1}{2}}\big)\epsilon^{M}\Vert \varphi\Vert _{L^{\infty}(0,T_{\epsilon};H^{1}(\Omega))},\label{eq:eins}\\
\int_{0}^{T_{\epsilon}}\left|\int_{\Omega}\rh R\d x\right|\d t & \leq C(K)\big(T_{\epsilon}^{\frac{1}{3}}+\epsilon^{\frac{1}{2}}\big)\epsilon^{2M},\label{eq:zwei}\\
\Vert \rs\Vert _{L^{2}(0,T_{\epsilon};(H^{1}(\Omega))')} & \leq C(K)\epsilon^{M},\label{eq:drei}\\
\Vert \rdiv\Vert _{L^{2}(\Omega_{T_{\epsilon}})} & \leq C(K)\epsilon^{M}.\label{eq:vier}
\end{align}
\end{thm}

\begin{proof}
As before, we will use the notation $\psi(r,s,t):=\psi(X(r,s,t))$
for $(r,s,t)\in(-2\delta,2\delta)\times\mathbb{T}^{1}\times[0,T_{\epsilon}]$
for functions $\psi\colon \Gamma(2\delta;T_{\epsilon})\rightarrow\mathbb{R}$.
Let in the following $\tilde{\epsilon}_{2}\in (0,\epsilon_{1}]$
be chosen such that the results of Section \ref{sec:First-Estimates}
hold and let $\epsilon\in(0,\tilde{\epsilon}_{2})$.

\noindent
\textbf{Proof of (\ref{eq:eins}):} Since $\xi(d_{\Gamma})\equiv1$
in $\Gamma (\delta;T_{0})$, we have $\rc=\rci$ in $\Gamma(\delta;T_{\epsilon})$
with $\rci$ as in (\ref{eq:CH1-rem}). Now 
\[
\int_{0}^{T_{\epsilon}}\Big|\,\int_{\Gamma_{t}(\delta)}\rci\varphi\d x\Big|\d t\leq C(K)\big(T_{\epsilon}^{\frac{1}{2}}+\epsilon^{\frac{1}{2}}\big)\epsilon^{M}\left\Vert \varphi\right\Vert _{L^{\infty}\left(0,T_{\epsilon};H^{1}(\Omega)\right)}
\]
holds because of Lemma \ref{rech1} and (\ref{eq:besteB}). 

Moreover, we have $(1-\xi(d_{\Gamma}))(1-\xi(2d_{\mathbf{B}}))\equiv1$
in $\Omega_{T_{\epsilon}}\backslash\left(\Gamma(2\delta;T_{\epsilon})\cup\partial_{T_{\epsilon}}\Omega(\delta)\right)$
and thus $\rc=\rco$ in that domain, with $\rco$ as in (\ref{eq:CH1-remO}).
Now all terms in $\rco$ which do not involve $\mathbf{v}_{M-\frac{1}{2}}^{\pm,\epsilon}$
can be estimated in $L^{\infty}(\Omega_{T_{0}}\backslash\Gamma(2\delta;T_{0}))$,
yielding the desired estimate. The terms involving $\mathbf{v}_{M-\frac{1}{2}}^{\pm,\epsilon}$
can be treated by using Hölder's inequality and (\ref{eq:muv0,5est}),
i.e., 
\begin{align}
\epsilon^{M+\frac{1}{2}}\int_{0}^{T_{\epsilon}}\int_{\Omega^{+}(t)\backslash\Gamma_{t}(2\delta)}\big|\mathbf{v}_{M-\frac{1}{2}}^{\epsilon,+}\cdot\nabla c_{j}^{+}\varphi\big|\d x\d t & \leq C(K)T_{\epsilon}^{\frac{1}{2}}\epsilon^{M+\frac{1}{2}}\left\Vert \varphi\right\Vert _{L^{\infty}\left(0,T_{\epsilon};H^{1}(\Omega)\right)}\label{eq:mark1}
\end{align}
for $j\in\left\{ 1,\ldots,M+1\right\} $. The same argumentation also
holds in $\Omega^{-}(t)$.

Close to the boundary, in $\partial_{T_{\epsilon}}\Omega(\frac{\delta}{2})$,
we have $\xi(2d_{\mathbf{B}})\equiv1$ and thus $\rc=\rcb$.
As in the outer case, all terms not involving $\mathbf{v}_{M-\frac{1}{2}}^{-,\epsilon}$
may be estimated in $L^{\infty}(\partial_{T_{0}}\Omega(\delta))$,
the rest can be estimated as in (\ref{eq:mark1}).

Next, we give estimates for $\rc$ in $\Gamma(2\delta;T_{\epsilon})\backslash\Gamma(\delta;T_{\epsilon})$:
By definition of $c_{A}^{\epsilon}$ and $\mu_{A}^{\epsilon}$ in
(\ref{eq:apprxsol}) we have
\begin{align}
\rc & =\xi(d_{\Gamma})\rci+(1-\xi(d_{\Gamma}))\rco-2\xi'\mathbf{n}\cdot\nabla (\mu_{I}-\mu_{O})\nonumber \\
 & \quad+\xi'(d_{\Gamma})\big(\partial_{t}d_{\Gamma}+\mathbf{v}_{A}^{\epsilon}\cdot\mathbf{n}+\epsilon^{M-\frac{1}{2}}\mathbf{w}_{1}^{\epsilon}|_{\Gamma}\cdot\mathbf{n}\xi(d_{\Gamma})\big)(c_{I}-c_{O})\nonumber \\
 & \quad+\mathbf{v}_{A}^{\epsilon}\cdot\big(\xi(d_{\Gamma})\nabla c_{I}+(1-\xi(d_{\Gamma}))\nabla c_{O}\big)-\xi(d_{\Gamma})\mathbf{v}_{I}\cdot\nabla c_{I}-(1-\xi(d_{\Gamma}))\mathbf{v}_{O}\cdot\nabla c_{O}\nonumber \\
 & -\left(\mu_{I}-\mu_{O}\right)\left(\xi''+\xi'\Delta d_{\Gamma}\right)+\epsilon^{M-\frac{1}{2}}\mathbf{w}_{1}^{\epsilon}|_{\Gamma}\cdot\nabla c_{O}\xi(d_{\Gamma})\left(1-\xi(d_{\Gamma})\right).\label{eq:rbetween}
\end{align} 
The term $(1-\xi(d_{\Gamma}))r_{CH1,O}^{\epsilon}$
may be estimated in the same way as in the outer domain $\Omega_{T_{\epsilon}}\backslash\left(\Gamma(2\delta;T_{\epsilon})\cup\partial_{T_{\epsilon}}\Omega(\delta)\right)$,
using $\left|1-\xi(d_{\Gamma})\right|\leq1$. Regarding
$\xi(d_{\Gamma})\rci$, there is a subtlety we have to
deal with: All appearing terms in the explicit structure of the difference
$\xi(d_{\Gamma})\big(\rci-\epsilon^{M-\frac{3}{2}}B^{M-\frac{1}{2}}\big)$
can be estimated with the help of Lemma \ref{rech1}. But we may not
simply use (\ref{eq:besteB}) for $\xi(d_{\Gamma})\epsilon^{M-\frac{3}{2}}B^{M-\frac{1}{2}}\varphi$
in $\Gamma(2\delta)$.

To treat this term let $J=(-2\delta,-\delta)\cup(\delta,2\delta)$. We estimate, using Lemma \ref{ABstruc}.2,
\begin{align}
&\int_{0}^{T_{\epsilon}}  \int_{\Gamma_{t}(2\delta)\backslash\Gamma_{t}(\delta)}\big|\xi(d_{\Gamma})\epsilon^{M-\frac{3}{2}}B^{M-\frac{1}{2}}\varphi\big|\d(x,t)\leq C\epsilon^{M-\frac{3}{2}}\int_{0}^{T_{\epsilon}}\int_{\mathbb{T}^{1}}\int_{J}\big|B^{M-\frac{1}{2}}\varphi\big|\d r\d s\d t\nonumber \\
 & \leq C\epsilon^{M-\frac{3}{2}}\sum_{k=1}^{K_{1}}\int_{0}^{T_{\epsilon}}\int_{\mathbb{T}^{1}}\Vert \varphi(.,s,t)\Vert _{L^{\infty}(-2\delta,2\delta)}\left\Vert \mathtt{B}_{k}^{1}(.,s,t)\right\Vert _{L^{2}(-2\delta,2\delta)}\big\Vert \mathtt{B}_{k}^{2}(\rho)\big\Vert _{L^{2}(J)}\d s\d t.\label{eq:pt1}
\end{align}
Now since $\frac{\delta}{\epsilon}-h_{A}^{\epsilon}\geq\frac{\delta}{2\epsilon}$
for $\epsilon>0$ small enough, we may derive for $k\in\left\{ 1,\ldots,K_{1}\right\} $
\begin{align}
\int_{\delta}^{2\delta}\left|\mathtt{B}_{k}^{2}(\rho(r,p,t))\right|^{2}\d r & \leq\epsilon\int_{\frac{\delta}{2\epsilon}}^{\infty}\left|\mathtt{B}_{k}^{2}(\rho)\right|^{2}\d\rho\leq\epsilon C_{1}e^{-\frac{C_{2}}{\epsilon}}\label{eq:expaba}
\end{align}
for some constants $C_{1},C_{2}>0$, where we used $\mathtt{B}_{k}^{2}\in\mathcal{O}(e^{-\alpha\left|\rho\right|})$ due to Lemma~\ref{ABstruc}.2. A similar estimate holds
on $(-2\delta,-\delta)$, allowing for a suitable estimate
of (\ref{eq:pt1}) with the help of (\ref{eq:b1est}).

Concerning $\xi'(d_{\Gamma})\big(\partial_{t}d_{\Gamma}+\mathbf{v}_{A}^{\epsilon}\cdot\mathbf{n}+\epsilon^{M-\frac{1}{2}}\mathbf{w}_{1}^{\epsilon}|_{\Gamma}\cdot\mathbf{n}\xi(d_{\Gamma})\big)(c_{I}-c_{O})$
in (\ref{eq:rbetween}), we exemplarily estimate 
\begin{align}
&\int_{0}^{T_{\epsilon}}\int_{\Gamma_{t}(2\delta)\backslash\Gamma_{t}(\delta)}  \big|\epsilon^{M-\frac{1}{2}}\mathbf{w}_{1}^{\epsilon}|_{\Gamma}\cdot\mathbf{n}(c_{I}-c_{O})\varphi\big|\d x\d t\nonumber \\
 &\quad \leq CT_{\epsilon}^{\frac{1}{2}}\big\Vert \epsilon^{M-\frac{1}{2}}\mathbf{w}_{1}^{\epsilon}\big\Vert _{L^{2}(0,T_{\epsilon};H^{1}(\Omega))}\Vert \varphi\Vert _{L^{\infty}(0,T_{\epsilon};H^{1}(\Omega))}\Vert c_{I}-c_{O}\Vert _{L^{\infty}(\Gamma(2\delta;T_{\epsilon})\backslash\Gamma(\delta;T_{\epsilon}))}\nonumber \\
 &\quad \leq C(K)T_{\epsilon}^{\frac{1}{2}}\epsilon^{M}\Vert \varphi\Vert _{L^{\infty}(0,T_{\epsilon};H^{1}(\Omega))},\label{eq:matchterm}
\end{align}
where we used $H^{1}(\Gamma_{t}(2\delta))\hookrightarrow L^{2,\infty}(\Gamma_{t}(2\delta))$, Lemma \ref{Wichtig},
and Corollary \ref{cor:matching}. An analogous (but simpler) argumentation
may be used for $\partial_{t}d_{\Gamma}\in L^{\infty}(\Gamma(2\delta;T_{0}))$
and $\big(\mathbf{v}_{A}^{\epsilon}-\epsilon^{M-\frac{1}{2}}\mathbf{v}_{A,M-\frac{1}{2}}\big)\in L^{\infty}(\Omega_{T_{0}})$
(cf.\ Definition \ref{def:apprxsol} for notations). The estimate
for $\epsilon^{M-\frac{1}{2}}\mathbf{v}_{A,M-\frac{1}{2}}^{\epsilon}$
then follows by using (\ref{eq:muv0,5est}). The terms $2\xi'\mathbf{n}\cdot\nabla\left(\mu_{I}-\mu_{O}\right)+\left(\mu_{I}-\mu_{O}\right)\left(\xi''+\xi'\Delta d_{\Gamma}\right)$
in (\ref{eq:rbetween}) may be treated by using Corollary \ref{cor:matching}.

For the third line of (\ref{eq:rbetween}), we calculate
\begin{align*}
\mathbf{v}_{A}^{\epsilon}\cdot\nabla c_{I}-\mathbf{v}_{I}\cdot\nabla c_{I} & =(1-\xi(d_{\Gamma}))(\mathbf{v}_{O}-\mathbf{v}_{I})\cdot\nabla c_{I}\\
\mathbf{v}_{A}^{\epsilon}\cdot\nabla c_{O}-\mathbf{v}_{O}\cdot\nabla c_{O} & =\xi(d_{\Gamma})(\mathbf{v}_{I}-\mathbf{v}_{O})\cdot\nabla c_{I}
\end{align*}
and Corollary \ref{cor:matching} yields the estimate as before. 

The only remaining term in (\ref{eq:rbetween}) can
be treated by 
\begin{align*}
\int_{0}^{T_{\epsilon}}\int_{\Gamma_{t}(2\delta)\backslash\Gamma_{t}(\delta)}\big|\epsilon^{M-\frac{1}{2}}\mathbf{w}_{1}^{\epsilon}|_{\Gamma}\cdot\nabla c_{O}\varphi\big|\d x\d t & \leq C(K)\epsilon^{M+\frac{1}{2}}\Vert \varphi\Vert _{L^{\infty}(0,T_{\epsilon};H^{1}(\Omega))},
\end{align*}
where we used Lemma~\ref{Wichtig} and $\nabla c_{O}=\mathcal{O}(\epsilon)$
in $L^{\infty}(\Omega_{T_{0}}^{\pm})$.

Thus, we need only consider $\rc$ in $\partial_{T_{\epsilon}}\Omega(\delta)\backslash\partial_{T_{\epsilon}}\Omega (\frac{\delta}{2})$.
Here we get a structure very similar to (\ref{eq:rbetween}):
\begin{align*}
\rc & =(1-\xi(2d_{\mathbf{B}}))\rco+\xi(2d_{\mathbf{B}})\rcb+2\xi'(2d_{\mathbf{B}})\left(\partial_{t}d_{\mathbf{B}}+\mathbf{v}_{A}^{\epsilon}\cdot\mathbf{n}_{\partial\Omega}\right)(c_{\mathbf{B}}-c_{O})\\
 & \quad+\mathbf{v}_{A}^{\epsilon}\cdot\big((1-\xi(2d_{\mathbf{B}}))\nabla c_{O}+\xi(2d_{\mathbf{B}})\nabla c_{\mathbf{B}}\big)-(1-\xi(2d_{\mathbf{B}}))\mathbf{v}_{O}\cdot\nabla c_{O}\\
 & \quad-\xi(2d_{\mathbf{B}})\mathbf{v}_{\mathbf{B}}\cdot\nabla c_{\mathbf{B}}-4\xi'\mathbf{n}_{\partial\Omega}\cdot\nabla(\mu_{\mathbf{B}}-\mu_{O})-(\mu_{\mathbf{B}}-\mu_{O})\left(4\xi''+2\xi'\Delta d_{\mathbf{B}}\right).
\end{align*}
The proof now follows in the same manner as the one for (\ref{eq:rbetween})
using the already shown estimates for $\rco$ and $\rcb$ as
well as the estimates close to the boundary in Corollary~\ref{cor:matching}.
This shows (\ref{eq:remcahn}).

\noindent   
\textbf{Proof of (\ref{eq:zwei}):} We use a similar approach as before:
In $\Gamma(\delta;T_{\epsilon})$ we have $\rh=\rhi$,
where $\rhi$ is defined in (\ref{eq:CH2-rem}). For all terms in
$\rhi$, which can be estimated in $L^{\infty}(\Gamma(2\delta;T_{\epsilon}))$
(uniformly in $\epsilon$), we may use Proposition \ref{RZerlprop}
to show the claim. Noting (\ref{eq:haepglm}), the only terms that
may not be treated in this fashion are the ones involving $\Delta^{\Gamma}h_{M-\frac{1}{2}}^{\epsilon}$
and $A^{M-\frac{1}{2}}$. Regarding $\epsilon^{M-\frac{1}{2}}A^{M-\frac{1}{2}}$,
we may use  (\ref{eq:besteA}).
Concerning $\Delta^{\Gamma}h_{M-\frac{1}{2}}^{\epsilon}$, we obtain
\begin{align*}
 & \epsilon^{M+\frac{1}{2}}\!\int_{\Gamma(\delta;T_{\epsilon})}\big|\Delta^{\Gamma}h_{M-\frac{1}{2}}^{\epsilon}\partial_{\rho}c_{1}R\big|\d (x,t)\\
  & \leq C\epsilon^{M+1}\big\Vert \big(\partial_{s}^{2}h_{M-\frac{1}{2}}^{\epsilon},\partial_{s}h_{M-\frac{1}{2}}^{\epsilon}\big)\big\Vert _{L^{\infty}(0;T_{\epsilon};L^{2}(\mathbb{T}^1))}\!\Vert R\Vert _{L^{2}(\Omega_{T_{\epsilon}})}\cdot\Big\Vert \sup_{(x,t)\in\Gamma (2\delta;T_{0})}|\partial_{\rho}c_{1}(.,x,t)|\Big\Vert _{L^{2}(\mathbb{R})}\\
 & \leq C(K)\epsilon^{2M+\frac{1}{2}},
\end{align*}
where we used $\partial_{\rho}c_{1}\in\mathcal{R}_{\alpha}$, $X_{T}\hookrightarrow C^{0}([0,T];H^{2}(\mathbb{T}^1))$
(cf.\ Proposition \ref{embedding}.2) and the $L^{2}$-estimate
for $R$ in (\ref{eq:Main-est}).

In $\Omega_{T_{\epsilon}}\backslash(\Gamma(2\delta;T_{\epsilon})\cup\partial_{T_{\epsilon}}\Omega(\delta))$,
we have $\rh=\rhO$ with $\rhO$ as in (\ref{eq:CH2-remO}). For that,
we obtain (exemplarily in $\Omega^{+}(t)$)
\begin{align*}
  &\int_{0}^{T_{\epsilon}}\int_{\Omega^{+}(t)\backslash\Gamma_{t}(2\delta)}\big|\epsilon^{M-\frac{1}{2}}\mu_{M-\frac{1}{2}}^{+,\epsilon}R\big|\d x\d t\\
  & \quad \leq CT_{\epsilon}^{\frac{1}{3}}\epsilon^{M-\frac{1}{2}}\big\Vert \mu_{M-\frac{1}{2}}^{+,\epsilon}\big\Vert _{L^{6}(0,T_{\epsilon};L^{2}(\Omega^{+}(t)))}\Vert R\Vert _{L^{2}(L^{2}(\Omega_{T_{\epsilon}}\backslash\Gamma(\delta;T_{\epsilon})))} \leq C(K)T_{\epsilon}^{\frac{1}{3}}\epsilon^{2M},
\end{align*}
where we used  (\ref{eq:muv0,5est}) and (\ref{eq:Main-est}). As
$c_{i}^{\pm}\in L^{\infty}(\Omega_{T_{0}}^{\pm})$ for
all $i\in\left\{ 0,\ldots,M+1\right\} $, a similar estimate follows
by (\ref{eq:Main-est}) for the remaining terms in $\rhO$
(cf.\ Remark \ref{Outer-Rem} for the $\tilde{f}$ term). In $\partial_{T_{\epsilon}}\Omega(\frac{\delta}{2})$,
it holds $\rh=\rhb$ and we may proceed as in $\Omega_{T_{\epsilon}}\backslash\left(\Gamma\left(2\delta;T_{\epsilon}\right)\cup\partial_{T_{\epsilon}}\Omega(\delta)\right)$.

In $\Gamma(2\delta;T_{\epsilon})\backslash\Gamma(\delta;T_{\epsilon})$,
we have 
\begin{align}
\rh & =\xi(d_{\Gamma})(\epsilon\Delta c_{I}+\mu_{I})+(1-\xi(d_{\Gamma}))\left(\epsilon\Delta c_{O}+\mu_{O}\right)-\epsilon^{-1}f'(c_{A}^{\epsilon})\nonumber \\
 & \quad+\epsilon\left(\left(c_{I}-c_{O}\right)\left(\xi''(d_{\Gamma})+\xi'(d_{\Gamma})\Delta d_{\Gamma}\right)+\epsilon2\xi'(d_{\Gamma})\mathbf{n}\cdot\nabla\left(c_{I}-c_{O}\right)\right).\label{eq:pt2}
\end{align}
The estimate for the second line in (\ref{eq:pt2}) follows by similar
arguments as in the proof of (\ref{eq:remcahn}), by using Corollary
\ref{cor:matching}.

Using a Taylor expansion, we can rewrite the first line of (\ref{eq:pt2})
as
\begin{align}
\xi(d_{\Gamma}) & \rhi+(1-\xi(d_{\Gamma}))(\rhO)\nonumber \\
 & +\epsilon^{-1}\left(c_{O}-c_{I}\right)\xi(d_{\Gamma})(1-\xi(d_{\Gamma}))\left(-f''(\sigma_{2}(c_{A}^{\epsilon},c_{O}))+f''(\sigma_{1}(c_{A}^{\epsilon},c_{I}))\right),\label{eq:pt2,2}
\end{align}
where $\sigma_{1/2}(c_{A}^{\epsilon},c_{I/O})$ are suitable
intermediate points. Now $c_{A}^{\epsilon},c_{O},c_{I}\in L^{\infty}(\Gamma(2\delta;T_{0})\backslash\Gamma(\delta;T_{0}))$
uniformly in $\epsilon$ and thus $|f''(\sigma_{1})|,|f''(\sigma_{2})|\leq C$.
As a consequence of Corollary \ref{cor:matching}, we may estimate
the last part in (\ref{eq:pt2,2}) as before and the term involving
$\rhO$ as in the case of $\Omega_{T_{\epsilon}}\backslash\left(\Gamma(2\delta;T_{\epsilon})\cup\partial_{T_{\epsilon}}\Omega(\delta)\right)$.
Regarding $\rhi$, although we may not use the decomposition of $R$
anymore (Proposition \ref{RZerlprop} only holds in $\Gamma(\delta;T_{\epsilon})$),
we may now use $\Vert R\Vert _{L^{2}(0,T_{\epsilon};L^{2}(\Omega\backslash\Gamma_{t}(\delta)))}\leq C(K)\epsilon^{M+\frac{1}{2}}$
due to (\ref{eq:Main-est-a}). Thus, all terms in $\rhi$, which can
be estimated in $L^{\infty}(\Gamma(2\delta;T_{\epsilon}))$
(uniformly in $\epsilon$), are of no concern. This leaves us with
terms involving $\Delta^{\Gamma}h_{M-\frac{1}{2}}^{\epsilon}$ (which
may be treated as before) and $\xi(d_{\Gamma})\epsilon^{M-\frac{1}{2}}A^{M-\frac{1}{2}}$
since (\ref{eq:besteA}) only holds inside $\Gamma(\delta;T_{\epsilon})$.
According to (\ref{eq:Main-est}) and Lemma \ref{ABstruc} 1) we may
estimate 
\begin{align*}
\epsilon^{M-\frac{1}{2}}\int_{\Gamma(2\delta;T_{\epsilon})\backslash\Gamma(\delta;T_{\epsilon})}\big|A^{M-\frac{1}{2}}R\big|\d(x,t) & \leq C(K)\epsilon^{2M}\sum_{k=1}^{L_{1}}\big\Vert \mathtt{A}_{k}^{1}\big\Vert _{L^{2}(0,T_{\epsilon};L^{2}(\Gamma_{t}(2\delta)))} \leq C(K)\epsilon^{2M}T_{\epsilon}^{\frac{1}{3}}.
\end{align*}
The situation in $\partial_{T_{\epsilon}}\Omega(\delta)\backslash\partial_{T_{\epsilon}}\Omega(\frac{\delta}{2})$
heavily resembles (\ref{eq:pt2}) and the estimate follows in a similar
way as for (\ref{eq:pt2}). Thus, we have estimated all terms in $\rh$.

\noindent
\textbf{Proof of (\ref{eq:drei}):} The approach to show (\ref{eq:remstokes})
is very similar to the one used for (\ref{eq:remcahn}): We have $\rs=\rsi$
in $\Gamma(\delta;T_{\epsilon})$ with $\rsi$ as in (\ref{eq:ST-rem})
and may then use  Lemma~\ref{rest} and Lemma~\ref{beste} (more
precisely (\ref{eq:besteV})) to get the estimate in $\Gamma(\delta;T_{\epsilon})$.
In $\Omega_{T_{\epsilon}}\backslash\left(\Gamma(2\delta;T_{\epsilon})\cup\partial_{T_{\epsilon}}\Omega(\delta)\right)$
we have $\rs=\rso$ and we may simply estimate the occurring terms
in $L^{\infty}(\Omega_{T_{0}})$ or with the help of (\ref{eq:muv0,5est}).
In $\partial_{T_{\epsilon}}\Omega(\frac{\delta}{2})$ it
holds $\rs=\rsb$, allowing for a similar approach as for the outer
remainder.

In $\Gamma\left(2\delta;T_{\epsilon}\right)\backslash\Gamma\left(\delta,T_{\epsilon}\right)$,
we have 
\begin{align}
\rs & =\xi(d_{\Gamma})\rsi+\left(1-\xi\left(d_{\Gamma}\right)\right)\rso-\left(\xi'\left(d_{\Gamma}\right)\Delta d_{\Gamma}+\xi''\left(d_{\Gamma}\right)\right)\left(\mathbf{v}_{I}-\mathbf{v}_{O}\right)\nonumber \\
 & \quad-2\xi'\left(d_{\Gamma}\right)D\left(\mathbf{v}_{I}-\mathbf{v}_{O}\right)\mathbf{n}+\xi'\left(d_{\Gamma}\right)\mathbf{n}\left(p_{I}-p_{O}\right)-\mu_{A}^{\epsilon}\xi'\left(d_{\Gamma}\right)\mathbf{n}\left(c_{I}-c_{O}\right)\nonumber \\
 & \quad+\left(-\mu_{A}^{\epsilon}\left(\xi\left(d_{\Gamma}\right)\nabla c_{I}+\left(1-\xi\left(d_{\Gamma}\right)\right)\nabla c_{O}\right)+\xi\left(d_{\Gamma}\right)\mu_{I}\nabla c_{I}+\left(1-\xi\left(d_{\Gamma}\right)\right)\mu_{O}\nabla c_{O}\right).\label{eq:ptstokesrem}
\end{align}
To estimate $\rsi$, we may use Lemma~\ref{rest}  inside
$\Gamma\left(\delta;T_{\epsilon}\right)$ again, but have to be careful
when estimating $\epsilon^{M-\frac{3}{2}}\xi(d_{\Gamma})\big(\mathbf{V}^{M-\frac{1}{2}}\big)\mathbf{z}$
since (\ref{eq:besteV}) cannot be used. But, as for $\rci$, we can
get the desired inequality in $\Gamma(2\delta;T_{\epsilon})\backslash\Gamma(\delta,T_{\epsilon})$
by using an approach analogous to (\ref{eq:pt1}), which is possible
since Lemma~\ref{ABstruc}.3 guarantees $\mathtt{V}_{k}^{2}\in\mathcal{R}_{\alpha}$.
$\rso$ may be treated as in $\Omega_{T_{\epsilon}}\backslash(\Gamma(2\delta;T_{\epsilon})\cup\partial_{T_{\epsilon}}\Omega(\delta))$
and due to Corollary~\ref{cor:matching} we get the right estimate
for the terms involving $\mathbf{v}_{I}-\mathbf{v}_{O}$,
$\nabla\left(\mathbf{v}_{I}-\mathbf{v}_{O}\right)$, $p_{I}-p_{O}$
and $c_{I}-c_{O}$.

Regarding the last line of (\ref{eq:ptstokesrem}), we have 
\begin{align*}
(-\mu_{A}^{\epsilon}+\mu_{I})\nabla c_{I} & =(1-\xi(d_{\Gamma}))(\mu_{I}-\mu_{O})\nabla c_{I}\\
(-\mu_{A}^{\epsilon}+\mu_{O})\nabla c_{O} & =\xi(d_{\Gamma})(\mu_{O}-\mu_{I})\nabla c_{O},
\end{align*}
allowing to apply Corollary \ref{cor:matching}. As in the
proofs before, the estimates in $\partial_{T_{\epsilon}}\Omega(\delta)\backslash\partial_{T_{\epsilon}}\Omega(\frac{\delta}{2})$
may be shown as in the case $\Gamma(2\delta;T_{\epsilon})\backslash\Gamma(\delta,T_{\epsilon})$.

\noindent
\textbf{Proof of (\ref{eq:vier}):} We observe that in $\Omega_{T_{\epsilon}}\backslash\left(\Gamma(2\delta;T_{\epsilon})\cup\partial_{T_{\epsilon}}\Omega(\delta)\right)$
it holds $\rdivo=0$ by (\ref{eq:DIV-remO}) and thus in particular
$\rdiv=0$ in $\Omega_{T_{\epsilon}}\backslash\left(\Gamma(2\delta;T_{\epsilon})\cup\partial_{T_{\epsilon}}\Omega(\delta)\right)$.
In $\Gamma(2\delta;T_{\epsilon})$ we have 
\[
\rdiv=\xi(d_{\Gamma})\rdivi+\xi'(d_{\Gamma})\mathbf{n}\cdot(\mathbf{v}_{I}-\mathbf{v}_{O}).
\]
As before, we can treat the term $\xi'(d_{\Gamma})\mathbf{n}\cdot(\mathbf{v}_{I}-\mathbf{v}_{O})$
by using Corollary \ref{cor:matching}. For $\rdivi$, as defined
in (\ref{eq:DIV-rem}), we first note that we may use (\ref{eq:besteW})
to estimate $\epsilon^{M-\frac{1}{2}}W^{M-\frac{1}{2}}$ suitably.
Moreover, $\operatorname{div}\mathbf{v}_{M+1}\in L^{\infty}(\mathbb{R}\times\Gamma (2\delta;T_{0}))$
by construction and to estimate the products $\partial_{\rho}\mathbf{v}_{i}\cdot\nabla^{\Gamma}h_{j+1}$,
where $i+j\geq M+\frac{1}{2}$, we use that $\left\Vert \partial_{\rho}\mathbf{v}_{i}\right\Vert _{L^{2}(\Gamma(2\delta;T_{\epsilon}))}\big\Vert \nabla^{\Gamma}h_{j+1}\big\Vert _{L^{\infty}(\Gamma(2\delta;T_{\epsilon}))}\leq C(K)$
for all $i\in I_{M-\frac{1}{2}}^{M+1}$, $j\in I_{M-\frac{3}{2}}^{M}$,
due to construction in the case of $i,j\in\left\{ 0,\ldots,M\right\} $
and $i=M+1$ and due to (\ref{eq:heps0,5}) resp.\ (\ref{eq:muv0,5est})
in the case of $j=M-\frac{3}{2}$ resp.\ $i=M-\frac{1}{2}$. Similarly,
we get $\left\Vert \mathbf{u}_{i}\cdot\mathbf{n}\right\Vert _{L^{2}\left(\Gamma\left(2\delta;T_{\epsilon}\right)\right)}\left\Vert h_{j+1}\right\Vert _{L^{\infty}\left(\left(0,T_{\epsilon}\right)\times\mathbb{T}^{1}\right)}\leq C(K)$,
where we obtain an $L^{2}-L^{2}$ estimate for $\mathbf{u}_{M-\frac{1}{2}}^{\epsilon}$
in the same way as in (\ref{eq:lepspeps-1}). The other terms
appearing in the definition of $\rdivi$ can then be treated in the
same way. 
In $\partial_{T_{\epsilon}}\Omega(\delta)$, we finally
have 
\[
\rdiv=\xi(2d_{\mathbf{B}})\rdivb+2\xi'(2d_{\mathbf{B}})\mathbf{n}_{\partial\Omega}\cdot (\mathbf{v}_{\mathbf{B}}-\mathbf{v}_{O})
\]
 and the form of $\rdivb$ together with Corollary \ref{cor:matching}
implies the estimate. Thus, we have proven the claim.
\end{proof}
\begin{lem}
\label{lem:Rem-More-Estimates}Let the assumptions of Theorem \ref{remainder}
hold. Then there are $\epsilon_{2}\in(0,\epsilon_{1}]$
and a constant $C(K)>0$ such that for all $\epsilon\in(0,\epsilon_{2})$
\begin{align*}
\left\Vert \rh\nabla c_{A}^{\epsilon}\right\Vert _{L^{2}\left(0,T_{\epsilon};\left(H^{1}(\Omega)^{2}\right)'\right)} & \leq C(K)C(T_{\epsilon},\epsilon)\epsilon^{M}\\
\left\Vert \rc\right\Vert _{L^{2}\left(\partial_{T_{\epsilon}}\Omega\left(\frac{\delta}{2}\right)\right)} & \leq C(K)\epsilon^{M}
\end{align*}
where $C(T,\epsilon)\rightarrow0$ as $(T,\epsilon)\rightarrow0$.
\end{lem}

\begin{proof}
We start by showing (\ref{eq:rch2-nablacae}). For $\psi\in H^{1}(\Omega)^{2}$,
we consider 
\begin{equation}
\Big|\int_{\Omega}\rh\nabla c_{A}^{\epsilon}\cdot\psi\d x\Big|\leq\Big|\,\int_{\Gamma_{t}(\delta)}\rhi\nabla c_{A}^{\epsilon}\cdot\psi\d x\Big|+\Big|\,\int_{\Omega\backslash\Gamma_{t}(\delta)}\rh\nabla c_{A}^{\epsilon}\cdot\psi\d x\Big|\label{eq:helf2}
\end{equation}
and begin with analyzing the integral over $\Gamma_{t}(\delta)$.
First of all, we note that 
\begin{equation}
\nabla c_{A}^{\epsilon}=\theta_{0}'(\tfrac{1}{\epsilon}\mathbf{n}-\nabla^{\Gamma}h_{A}^{\epsilon})+\partial_{\rho}c_{1}\mathbf{n}+\mathcal{O}(\epsilon)\label{eq:gradcaepsstruc}
\end{equation}
 in $L^{\infty}(\Gamma(\delta;T_{\epsilon}))$
by construction and the fact that $\big\Vert \nabla^{\Gamma}h_{A}^{\epsilon}\big\Vert _{L^{\infty}(\Gamma(2\delta;T_{\epsilon}))}\leq C(K)$
by (\ref{eq:haepglm}). Thus, for all terms $g\colon\Gamma(2\delta)\rightarrow\mathbb{R}$
appearing in $\rhi$, which are multiplied by at least $\epsilon^{M}$
and which may be estimated in $L^{\infty}(\Gamma(2\delta;T_{\epsilon}))$
uniformly in $\epsilon$, we may use the estimate 
\begin{align*}
\Big\Vert \,\int_{\Gamma_{t}(\delta)}g\theta_{0}'\big(\tfrac{1}{\epsilon}\mathbf{n}-\nabla^{\Gamma}h_{A}^{\epsilon}\big)\cdot\psi\d x\Big\Vert _{L^{2}\left(0,T_{\epsilon}\right)} & \leq C(K)T_{\epsilon}^{\frac{1}{2}}\left\Vert g\right\Vert _{L^{\infty}\left(\Gamma\left(2\delta;T_{\epsilon}\right)\right)}\left\Vert \psi\right\Vert _{H^{1}(\Omega)},
\end{align*}
where we used $H^{1}\left(\Gamma_{t}(2\delta)\right)\hookrightarrow L^{2,\infty}\left(\Gamma_{t}(2\delta)\right)$
and the exponential decay of $\theta_{0}'$. As
discussed in the proof of Theorem \ref{remainder}, a similar approach
also works for the terms involving $\Delta^{\Gamma}h_{M-\frac{1}{2}}$.

Thus we have to show
\begin{equation}
\epsilon^{M-\frac{1}{2}}\Big\Vert \,\int_{\Gamma_{t}(\delta)}A^{M-\frac{1}{2}}\nabla c_{A}^{\epsilon}\cdot\psi\d x\Big\Vert _{L^{2}\left(0,T_{\epsilon}\right)}\leq C(K)C(T_{\epsilon},\epsilon)\epsilon^{M}\Vert \psi\Vert _{H^{1}(\Omega)}.\label{eq:helf1}
\end{equation}
To this end we will use the same notations as discussed right at the
beginning of the proof of Lemma \ref{beste}. We will first consider
$\frac{1}{\epsilon}\theta_{0}'\mathbf{n}$ instead of $\nabla c_{A}^{\epsilon}$.
Using the fundamental theorem of calculus we have $\psi(r,s)=\psi(0,s)+\int_{0}^{r}\partial_{\mathbf{n}}\psi(\tilde{r},s)\d r$
for $(r,s,t)\in(-\delta,\delta)\times\mathbb{T}^{1}$
and write
\begin{align*}
&\Big|\,\int_{\Gamma_{t}(\delta)}A^{M-\frac{1}{2}}\frac{1}{\epsilon}\theta_{0}'\mathbf{n}\cdot\psi\d x\Big|  \leq\frac{1}{\epsilon}\int_{\mathbb{T}^{1}}\left|\psi(0,s)\right|\Big|\int_{-\delta}^{\delta}A^{M-\frac{1}{2}}\big|_{\Gamma}\theta_{0}'J(r,s,t)\d r\Big|\d s\\
 & \quad+\frac{C_{1}}{\epsilon}\int_{\mathbb{T}^{^{1}}}\int_{-\delta}^{\delta}\Big|A^{M-\frac{1}{2}}|_{\Gamma}\theta_{0}'\int_{0}^{r}\partial_{\mathbf{n}}\psi(\tilde{r},s,t)\d\tilde{r}\Big|\d r\d s+\frac{C_{2}}{\epsilon}\int_{\mathbb{T}^{1}}\int_{-\delta}^{\delta}\Big|\big(A^{M-\frac{1}{2}}\big)^{\Gamma}\theta_{0}'\psi\Big|\d r\d s\\
 & \quad =:\mathcal{I}_{1}^{1}+\mathcal{I}_{1}^{2}+\mathcal{I}_{2}.
\end{align*}
By Lemma \ref{ABstruc} (after choosing $\epsilon>0$ small enough
such that (\ref{eq:epsklein}) holds), we may estimate
\begin{align*}
\mathcal{I}_{1}^{1} & \leq\sum_{k=1}^{L_{2}}\int_{\mathbb{T}^{1}}\big|\psi(0,s)\mathtt{A}_{k}^{1,\Gamma}\big|\Big|\int_{-\frac{\delta}{\epsilon}-h_{A}^{\epsilon}}^{\frac{\delta}{\epsilon}-h_{A}^{\epsilon}}\mathtt{A}_{k}^{2,\Gamma}\theta_{0}'J^{\epsilon}\d\rho\Big|\d s \leq C_{1}\Vert \psi\Vert _{H^{1}(\Omega)}\big\Vert \mathtt{A}_{k}^{1,\Gamma}\big\Vert _{L^{2}(\Gamma_{t})}\big(e^{-C_{2}\frac{\delta}{\epsilon}}+C(K)\epsilon\big)
\end{align*}
and thus get $\left\Vert \mathcal{I}_{1}^{1}\right\Vert _{L^{2}\left(0,T_{\epsilon}\right)}\leq C(K)\epsilon\left\Vert \psi\right\Vert _{H^{1}(\Omega)}$
due to (\ref{eq:a1est}). 
Concerning $\mathcal{I}_{1}^{2}$, we have 
\begin{align*}
  \Vert \mathcal{I}_{1}^{2}\Vert _{L^{2}(0,T_{\epsilon})} & \leq\Big\Vert \frac{1}{\epsilon}\int_{\mathbb{T}^{1}}\Vert \psi\Vert _{H^{1}(-\delta,\delta)}\int_{-\delta}^{\delta}\big|A^{M-\frac{1}{2}}\big|_{\Gamma}\theta_{0}'r^{\frac{1}{2}}\big|\d r\d s\Big\Vert _{L^{2}\left(0,T_{\epsilon}\right)}\\
 & \leq C(K)T_{\epsilon}^{\frac{1}{3}}\epsilon^{\frac{1}{2}}\sum_{k=1}^{L_{2}}\Vert \psi\Vert _{H^{1}(\Omega)}\big\Vert \mathtt{A}_{k}^{1,\Gamma}\big\Vert _{L^{6}(0,T_{\epsilon};L^{2}(\Gamma_{t}))}
\end{align*}
and may use (\ref{eq:a1est}). Here we used $\big\Vert \mathtt{A}_{k}^{2,\Gamma}\big\Vert _{L^{\infty}(\mathbb{R})}\leq C$
for all $k\in\left\{ 1,\ldots,L_{2}\right\} $.

For $\mathcal{I}_{2}$, we need to consider the explicit structure
of $A^{M-\frac{1}{2}}$ and show only two exemplary estimates, all
others follow along the same lines. Firstly, we consider the term $\big(\mu_{M-\frac{1}{2}}^{\epsilon}\big)^{\Gamma}$
appearing in $\big(A^{M-\frac{1}{2}}\big)^{\Gamma}$ (see also
(\ref{eq:mum0,5+1}) for the detailed structure):
\begin{align*}
\frac{1}{\epsilon}\int_{\mathbb{T}^{1}}\int_{-\delta}^{\delta} & \Big|\big(\mu_{M-\frac{1}{2}}^{+,\epsilon}\big)^{\Gamma}\eta\theta_{0}'\psi\Big|\d r\d s\\
 & \leq C\int_{\mathbb{T}^{1}}\sup_{r\in\left(-\delta,\delta\right)}|\psi(r,s)|\sup_{r\in(-\delta,\delta)}\big|\partial_{\mathbf{n}}\mu_{M-\frac{1}{2}}^{+,\epsilon}(r,s,t)\big|\int_{-\frac{\delta}{\epsilon}-h_{A}^{\epsilon}}^{\frac{\delta}{\epsilon}-h_{A}^{\epsilon}}\left|\epsilon(\rho+h_{A}^{\epsilon})\right||\theta_{0}'|\d\rho\d s\\
 & \leq C(K)\epsilon\Vert \psi\Vert _{H^{1}(\Omega)}\big\Vert \mu_{M-\frac{1}{2}}^{+,\epsilon}\big\Vert _{H^{2}(\Omega^{+}(t))}.
\end{align*}
The estimate for $\big(\mu_{M-\frac{1}{2}}^{-,\epsilon}\big)^{\Gamma}\eta$
follows analogously. 

Secondly, we consider the term $\big(|\nabla S(x,t)|^{2}\big)^{\Gamma}\partial_{s}^{2}h_{M-\frac{1}{2}}^{\epsilon}(S(x,t),t)$,
as all other occurring terms in $\big(A^{M-\frac{1}{2}}\big)^{\Gamma}$
consist of lower derivatives of $h_{M-\frac{1}{2}}^{\epsilon}$ and
can be treated in the same way. Using similar techniques as in the
estimate above, we get 
\[
\frac{1}{\epsilon}\int_{\mathbb{T}^{1}}\int_{-\delta}^{\delta}\big|\big(|\nabla S|^{2}\big)^{\Gamma}\partial_{s}^{2}h_{M-\frac{1}{2}}^{\epsilon}\theta_{0}'\psi\big|\d r\d s\leq C(K)\epsilon \Vert \psi\Vert _{H^{1}(\Omega)}\big\Vert h_{M-\frac{1}{2}}^{\epsilon}\big\Vert _{H^{2}(\mathbb{T}^1)}.
\]
Thus, we get by (\ref{eq:heps0,5}) and (\ref{eq:muv0,5est}) $\left\Vert \mathcal{I}_{2}\right\Vert _{L^{2}\left(0,T_{\epsilon}\right)}\leq C(K)\epsilon\left\Vert \psi\right\Vert _{H^{1}(\Omega)}.$
Altogether we obtain
\[
\epsilon^{M-\frac{1}{2}}\Big\Vert \,\int_{\Gamma_{t}(\delta)}A^{M-\frac{1}{2}}\frac{1}{\epsilon}\theta_{0}'\mathbf{n}\cdot\psi\d x\Big\Vert _{L^{2}(0,T_{\epsilon})}\leq C(K)C(T_{\epsilon},\epsilon)\epsilon^{M}\Vert \psi\Vert _{H^{1}(\Omega)}.
\]
Regarding (\ref{eq:helf1}), we conclude
\begin{align*}
\Big|\,\int_{\Gamma_{t}(\delta)}A^{M-\frac{1}{2}}\theta_{0}'\nabla^{\Gamma}h_{A}^{\epsilon}\cdot\psi\d x\Big| & \leq C(K)\sum_{k=1}^{L_{1}}\big\Vert \mathtt{A}_{k}^{1}\big\Vert _{L^{2}(\Gamma_{t}(2\delta))}\Vert \psi\Vert _{L^{2,\infty}(\Gamma_{t}(2\delta))}\epsilon^{\frac{1}{2}}\Vert \theta_{0}'\Vert _{L^{2}(\mathbb{R})}
\end{align*}
by Lemma \ref{ABstruc}. As $\partial_{\rho}c_{1}\in\mathcal{R}_{\alpha}$
and all other terms appearing in $\nabla c_{A}^{\epsilon}$ are already
of higher order in $\epsilon$  (see (\ref{eq:gradcaepsstruc})).
This proves (\ref{eq:helf1}) and as a consequence also 
\[
\Big\Vert \,\int_{\Gamma_{t}(\delta)}\rhi\nabla c_{A}^{\epsilon}\cdot\psi\d x\Big\Vert _{L^{2}(0,T_{\epsilon})}\leq C(K)C(T,\epsilon)\epsilon^{M}\Vert \psi\Vert _{H^{1}(\Omega)}.
\]
In view of (\ref{eq:helf2}), we still need to consider $\big|\int_{\Omega\backslash\Gamma_{t}(\delta)}\rh\nabla c_{A}^{\epsilon}\cdot\psi\d x\big|$.
But this term may be treated with similar techniques as used in the
proof of (\ref{eq:remhill}). This shows (\ref{eq:rch2-nablacae}).

Finally, (\ref{eq:rch1-rch2-Linfbdry}) follows immediately by noting that
$\rc=\rcb$ in $\partial_{T_{0}}\Omega\left(\frac{\delta}{2}\right)$,
the form of the boundary remainder terms (\ref{eq:CH1-remB}) and
the fact that all occurring terms in those boundary remainders are
either uniformly bounded in $L^{\infty}(\partial_{T_{0}}\Omega(\delta))$
or may be estimated in $L^{2}(\Omega_{T_{\epsilon}}^{-})$
with the help of (\ref{eq:muv0,5est}).
\end{proof}
\begin{proof}[Proof of Theorem \ref{thm:Main-Apprx-Structure}]
 Let the approximations be defined as in Definition \ref{def:apprxsol},
let $\mathbf{w}_{1}^{\epsilon}$ be given as in Theorem \ref{hM-1}
and let $\rc$, $\rh$, $\rs$ and $\rdiv$ be defined via (\ref{eq:Stokesapp})\textendash (\ref{eq:Hilliardapp}).
(\ref{eq:boundaryconditions}) is a result of (\ref{eq:Rem-Dir-Mu})\textendash (\ref{eq:Rem-Navier-v})
and $\rdiv=0$ on $\partial_{T_{0}}\Omega$ of (\ref{eq:DIV-remB}).
The estimates (\ref{eq:remcahn})\textendash (\ref{eq:rch1-rch2-Linfbdry})
are a consequence of Theorem \ref{remainder} and Lemma \ref{lem:Rem-More-Estimates}.
\end{proof}
\section*{Acknowledgement} 
The results of this paper are part of the second author's PhD Thesis. The authors acknowledge support by the SPP 1506 ``Transport Processes
at Fluidic Interfaces'' of the German Science Foundation (DFG) through the grant AB285/4-2.

\providecommand{\bysame}{\leavevmode\hbox to3em{\hrulefill}\thinspace}
\providecommand{\MR}{\relax\ifhmode\unskip\space\fi MR }
\providecommand{\MRhref}[2]{%
  \href{http://www.ams.org/mathscinet-getitem?mr=#1}{#2}
}
\providecommand{\href}[2]{#2}

\end{document}